\documentclass[11pt, a4paper,leqno]{amsart}
\usepackage{amsmath,amsthm,amscd,amssymb,amsfonts, amsbsy}
\usepackage{latexsym}
\usepackage{txfonts}
\usepackage{exscale}

\usepackage[colorlinks,citecolor=red,pagebackref,hypertexnames=false]{hyperref}
\usepackage{color}

\calclayout
\allowdisplaybreaks

\theoremstyle{plain}
\newtheorem{theorem}[equation]{Theorem}
\newtheorem{lemma}[equation]{Lemma}

\newtheorem{proposition}[equation]{Proposition}

\theoremstyle{definition}
\newtheorem{definition}[equation]{Definition}

\theoremstyle{remark}
\newtheorem{remark}[equation]{Remark}

\newtheorem{claim}[equation]{Claim}

\numberwithin{equation}{section}

\newcommand{\RR}{{\mathbb{R}}}
\newcommand{\eps}{\varepsilon}

\newcommand{\dint}{\int\!\!\!\int}

\newcommand{\dist}{\operatorname{dist}}

\newcommand{\td}{\Delta_{\star}}
\newcommand{\hx}{\widehat{x}_Q}

\newcommand{\re}{\mathbb{R}}
\newcommand{\rn}{\mathbb{R}^n}

\newcommand{\reu}{\mathbb{R}^{n+1}_+}
\newcommand{\ree}{\mathbb{R}^{n+1}}

\newcommand{\dd}{\mathbb{D}}

\newcommand{\C}{\mathcal{C}}
\newcommand{\osp}{\Omega^+_{\sbf}}
\newcommand{\om}{\Omega}
\newcommand{\ot}{\Omega_0}
\newcommand{\ott}{\Omega_1}
\newcommand{\ospp}{\Omega^+_{\sbf'}}
\newcommand{\posp}{\partial\Omega^+_{\sbf}}
\newcommand{\pospp}{\partial\Omega^+_{\sbf'}}

\newcommand{\F}{\mathcal{F}}
\newcommand{\Qt}{\widetilde{Q}}

\newcommand{\M}{\mathcal{M}}

\newcommand{\nn}{\mathcal{N}}
\newcommand{\MS}{\mathcal{M}_{\sbf}}
\newcommand{\W}{\mathcal{W}}
\newcommand{\B}{\mathcal{B}}
\newcommand{\Bt}{\mathcal{B}^*}

\newcommand{\R}{\mathcal{R}}

\newcommand{\mfs}{\mathfrak{S}}
\newcommand{\sbf}{{\bf S}}

\newcommand{\G}{\mathcal{G}}

\newcommand{\GS}{\Gamma_{\sbf}}

\newcommand{\bt}{\widetilde{B}}

\newcommand{\mut}{\mathfrak{m}}
\newcommand{\mutt}{\widetilde{\mathfrak{m}}}

\newcommand{\pom}{\partial\Omega}

\newcommand{\hm}{\omega}
\newcommand{\vp}{\varphi}

\renewcommand{\emptyset}{\mbox{\textup{\O}}}

\DeclareMathOperator{\supp}{supp}

\DeclareMathOperator{\diam}{diam}
\DeclareMathOperator{\osc}{osc}
\DeclareMathOperator{\interior}{int}

\def\div{\mathop{\operatorname{div}}\nolimits}

\begin{document}
\allowdisplaybreaks

\title[Uniform Rectifiability, Carleson measure estimates, and approximation]{Uniform Rectifiability,
Carleson measure estimates, and approximation of harmonic functions}

\author{Steve Hofmann}

\address{Steve Hofmann
\\
Department of Mathematics
\\
University of Missouri
\\
Columbia, MO 65211, USA} \email{hofmanns@missouri.edu}

\author{Jos\'e Mar{\'\i}a Martell}

\address{Jos\'e Mar{\'\i}a Martell\\
Instituto de Ciencias Matem\'aticas CSIC-UAM-UC3M-UCM\\
Consejo Superior de Investigaciones Cient{\'\i}ficas\\
C/ Nicol\'as Cabrera, 13-15\\
E-28049 Madrid, Spain} \email{chema.martell@icmat.es}

\author{Svitlana Mayboroda}

\address{Svitlana Mayboroda
\\
Department of Mathematics
\\
University of Minnesota
\\
Minneapolis, MN 55455, USA} \email{svitlana@math.umn.edu}

\thanks{The first author was supported by NSF grant DMS-1361701. The second author was supported in part by MINECO Grant
MTM2010-16518, ICMAT Severo Ochoa project SEV-2011-0087. He also acknowledges that
the research leading to these results has received funding from the European Research
Council under the European Union's Seventh Framework Programme (FP7/2007-2013)/ ERC
agreement no. 615112 HAPDEGMT. The third author was  supported by the Alfred P. Sloan Fellowship, the NSF CAREER Award DMS 1056004, the NSF INSPIRE Award DMS 1344235, and the NSF Materials Research Science and Engineering Center Seed Grant DMR 0212302.}

\date{\today}


\begin{abstract}
Let $E\subset \ree$, $n\ge 2$, be a uniformly rectifiable set of dimension $n$.
Then bounded harmonic functions in $\Omega:= \ree\setminus E$ satisfy
Carleson measure estimates,
and are ``$\eps$-approximable".    Our results may be viewed as generalized versions of the 
classical F. and M. Riesz theorem,
since the estimates that we prove are equivalent, in more topologically friendly settings, to quantitative mutual
absolute continuity of harmonic measure, and surface measure.
\end{abstract}

\maketitle

\tableofcontents

\section{Introduction}
In this paper, we establish generalized versions of a classical theorem of F. and M. Riesz \cite{Rfm},
who showed that for a simply connected domain $\Omega$ in the complex plane,
with a rectifiable boundary, harmonic measure is absolutely continuous with respect to arclength measure.
Our results are scale-invariant, higher dimensional versions of the result of \cite{Rfm}, whose main novelty
lies in the fact that we completely dispense with any hypothesis of connectivity. Despite recent successes of harmonic analysis on general uniformly rectifiable sets, for a long time connectivity seemed to be a vital hypothesis from the PDE point of view. Indeed, 
Bishop and Jones \cite{BiJo} have produced a counter-example to show
that the F. and M. Riesz Theorem
does not hold, in a literal way, in the absence of connectivity:  they construct a one dimensional
(uniformly) rectifiable set
$E$ in the complex plane, for which harmonic measure with respect to $\Omega= \mathbb{C}\setminus E$,
is singular with respect to Hausdorff $H^1$ measure on $E$.  The main result of this paper shows that, in spite of Bishop-Jones counterexample, suitable substitute estimates on harmonic functions remain valid in the absence of connectivity, in general uniformly rectifiable domains. 
In more topologically benign environments,  the latter are indeed  equivalent to
(scale-invariant) mutual absolute continuity of harmonic measure $\hm$ and surface measure $\sigma$ on $\pom$.

Let us be more precise.  In the setting of a Lipschitz domain $\Omega\subset \ree,\, n\geq 1$,
for any divergence form elliptic operator $L=-\div A\nabla$ with bounded measurable
coefficients,
the following are equivalent:
\begin{list}{$(\theenumi)$}{\usecounter{enumi}\leftmargin=.8cm
\labelwidth=.8cm\itemsep=0.2cm\topsep=.1cm
\renewcommand{\theenumi}{\roman{enumi}}}

\item Every bounded solution $u$, of the equation $Lu=0$ in $\Omega$, satisfies the
{\it Carleson measure estimate} \eqref{eq1.carl*} below.

\item Every bounded solution $u$, of the equation $Lu=0$ in $\Omega$, is {\it $\eps$-approximable}, for every $\eps>0$
(see Definition \ref{def1.3}).

\item  $\hm\in A_\infty(\sigma)$ on $\pom$ (see Definition \ref{defAinfty}).

\item  Uniform {\it Square function/Non-tangential maximal function} (``$S/N$")
estimates 
hold locally in  ``sawtooth" subdomains of $\Omega$.

\end{list}

Item $(iii)$ says, of course, that harmonic measure and surface measure are mutually absolutely continuous,
in a quantitative, scale-invariant way.  
We will not give a precise definition of the terms in item $(iv)$,   since these estimates are not the primary concern
of the present paper  (but see, e.g., \cite{DJK}, as well as our forthcoming companion paper to this one).  
On the other hand,
the Carleson measure estimate \eqref{eq1.carl*} is a special case
(which in fact implies the other cases) of one direction of the $S/N$ estimates (the direction ``$S<N$",
in which one controls the square function, in some $L^p$ norm, in terms of the
non-tangential maximal function).
We shall discuss the connections among these four properties
in more detail below.

In the present work, we show that if  $\om:= \ree\setminus E$, where
$E\subset \ree$ is a
uniformly rectifiable set (see Definition \ref{defur}) of co-dimension 1, then
$(i)$ and $(ii)$ continue to hold (see Theorems \ref{t1} and \ref{t2} below), 
even though $(iii)$ may now fail, by the example of \cite{BiJo}
mentioned above.   Moreover, we develop a general technique that yields transference
from NTA sub-domains to the complement of a uniformly rectifiable set and ultimately will allow us to attack a wide range of PDE questions on uniformly rectifiable domains. 
In a forthcoming sequel to the present paper, we shall
show that in this setting, both local and global ``$S<N$" estimates hold for harmonic functions and for solutions to general elliptic PDEs
(topological obstructions preclude the opposite direction).
We shall also present there a general transference principle by 
which one may transmit Carleson measure estimates and $S<N$ bounds from Lipschitz sub-domains to
NTA domains (as a companion to the transference
from NTA sub-domains to the complement of a uniformly rectifiable set achieved here).

The main results of this paper are as follows.  The terminology used in the statements of the theorems
will be defined momentarily, but for now let us note that
in particular, a UR set is closed by definition, so that $\Omega:=\ree\setminus E$ is open,
but need not be a connected domain.
For the sake of notational convenience, we set $\delta(X):=\dist(X,E)$.
As usual, $B(x,r)$ will denote the Euclidean ball of center $x$ and radius $r$ in $\ree$.

\begin{theorem}\label{t1} Let $E\subset \ree$ be a UR (uniformly rectifiable)  set of co-dimension 1.
Suppose that $u$ is harmonic and bounded in $\Omega:=\ree\setminus E$.
Then we have the Carleson measure estimate
\begin{equation}\label{eq1.carl*}
\sup_{x\in E,\, 0<r<\infty} \,\frac1{r^n}\iint_{B(x,r)} |\nabla u(Y)|^2 \delta(Y) \,dY\,
\leq \,C\, \|u\|^2_{L^\infty(\Omega)}\, ,
\end{equation}
where the constant $C$ depends only upon $n$ and the ``UR character" of $E$.
\end{theorem}

\begin{theorem}\label{t2} Let $E\subset \ree$ be a UR set of co-dimension 1.
Suppose that $u$ is harmonic and bounded in $\Omega:=\ree\setminus E$,
with $\|u\|_{L^\infty} \leq 1$.  Then $u$ is $\eps$-approximable
for every $\eps\in (0,1)$.
\end{theorem}

We conjecture that converses to Theorems \ref{t1} and \ref{t2} (or perhaps the combination of the two),
should hold.  Such results would be analogues of our work in \cite{HMU}.

Let us now define the terms used in the statements of our theorems.
The following notions have meaning in co-dimensions greater than 1, but here
we shall discuss only
the co-dimension 1 case that is of interest to us in this work. 

\begin{definition}\label{defadr} ({\bf  ADR})  (aka {\it Ahlfors-David regular}).
We say that a  set $E \subset \ree$, of Hausdorff dimension $n$, is ADR
if it is closed, and if there is some uniform constant $C$ such that
\begin{equation} \label{eq1.ADR}
\frac1C\, r^n \leq \sigma\big(\Delta(x,r)\big)
\leq C\, r^n,\quad\forall r\in(0,\diam (E)),\ x \in E,
\end{equation}
where $\diam(E)$ may be infinite.
Here, $\Delta(x,r):= E\cap B(x,r)$ is the ``surface ball" of radius $r$,
and $\sigma:= H^n|_E$ 
is the ``surface measure" on $E$, where $H^n$ denotes $n$-dimensional
Hausdorff measure.
\end{definition}

\begin{definition}\label{defur} ({\bf UR}) (aka {\it uniformly rectifiable}).
An $n$-dimensional ADR (hence closed) set $E\subset \ree$
is UR if and only if it contains ``Big Pieces of
Lipschitz Images" of $\rn$ (``BPLI").   This means that there are positive constants $\theta$ and
$M_0$, such that for each
$x\in E$ and each $r\in (0,\diam (E))$, there is a
Lipschitz mapping $\rho= \rho_{x,r}: \rn\to \ree$, with Lipschitz constant
no larger than $M_0$,
such that 
$$
H^n\Big(E\cap B(x,r)\cap  \rho\left(\{z\in\rn:|z|<r\}\right)\Big)\,\geq\,\theta\, r^n\,.
$$
\end{definition}

We recall that $n$-dimensional rectifiable sets are characterized by the
property that they can be
covered, up to a set of
$H^n$ measure 0, by a countable union of Lipschitz images of $\rn$;
we observe that BPLI  is a quantitative version
of this fact.

We remark
that, at least among the class of ADR sets, the UR sets
are precisely those for which all ``sufficiently nice" singular integrals
are $L^2$-bounded  \cite{DS1}.    In fact, for $n$-dimensional ADR sets
in $\ree$, the $L^2$ boundedness of certain special singular integral operators
(the ``Riesz Transforms"), suffices to characterize uniform rectifiability (see \cite{MMV} for the case $n=1$, and 
\cite{NToV} in general). 
We further remark that
there exist sets that are ADR (and that even form the boundary of a domain satisfying 
interior Corkscrew and Harnack Chain conditions),
but that are totally non-rectifiable (e.g., see the construction of Garnett's ``4-corners Cantor set"
in \cite[Chapter1]{DS2}).  Finally, we mention that there are numerous other characterizations of UR sets
(many of which remain valid in higher co-dimensions); cf. \cite{DS1,DS2}.

\begin{definition}\label{defurchar} ({\bf ``UR character"}).   Given a UR set $E\subset \ree$, its ``UR character"
is just the pair of constants $(\theta,M_0)$ involved in the definition of uniform rectifiability,
along with the ADR constant; or equivalently,
the quantitative bounds involved in any particular characterization of uniform rectifiability.
\end{definition}

Let $\Omega:= \ree\setminus E$, where $E\subset \ree$ is an $n$-dimensional ADR set (hence closed);
thus $\Omega$ is open, but need not be a connected domain.
\begin{definition}\label{def1.3} Let $u\in L^\infty(\Omega)$,
with $\|u\|_\infty \leq 1$, and let $\eps \in (0,1)$.   We say that $u$
is $\eps$-{\bf approximable}, if there is a constant $C_\eps$, and
a function $\vp =\vp^\eps\in W^{1,1}_{\rm loc}(\Omega)$
satisfying
\begin{equation}\label{eq1.4}\|u-\vp\|_{L^\infty(\Omega)}<\eps\,,
\end{equation}
and
\begin{equation}\label{eq1.5}
\sup_{x\in E,\, 0<r<\infty} \,\frac1{r^n}\iint_{B(x,r)}|\nabla \vp(Y)| \,dY\leq C_\eps\,.
\end{equation}
\end{definition}

We observe that \eqref{eq1.5} is an ``enhanced" version of the Carleson
estimate \eqref{eq1.carl*}.
On the other hand, even in the classical case that $\Omega$ is a half-space or a ball,
one cannot expect that
the $L^1$ Carleson measure bound \eqref{eq1.5} should hold, in general, with a bounded
harmonic function
$u$ in place of $\vp$ (there are counter-examples, see \cite[Ch. VIII]{G}).

The notion of $\eps$-approximability was introduced
by Varopoulos \cite{Va}, and (in sharper form) by Garnett \cite{G},
who were motivated in part by its connections with both
the $H^1/BMO$ duality theorem of Fefferman \cite{FS}, and the ``Corona Theorem" of Carleson
\cite{Car}.  In particular, the $\eps$-approximability property is the main ingredient in
the proof of Varopoulos's extension theorem, which states that every $f\in BMO(\rn)$ 
has an extension $F\in C^\infty(\reu)$, such that $|\nabla F(x,t)| dx dt$ is a Carleson measure.
Using ideas related to the proof of the Corona theorem, Garnett
showed that the $\eps$-approximability  property is enjoyed, for all $\eps\in (0,1)$, by
bounded harmonic functions in 
the half-space.  Garnett then uses this fact to establish a ``quantitative Fatou theorem", which provided
the first hint that $\eps$-approximability is related to quantitative properties
of harmonic measure.

As we have noted, the properties $(i)$-$(iv)$ listed above are equivalent, given suitable quantitative
connectivity of $\Omega$.  Let us recall, for example, the known results in the setting of a Lipschitz domain.
In that setting,
Dahlberg \cite{D} obtained an
extension Garnett's $\eps$-approximability result,
observing that
$(iv)$ implies $(ii)$\footnote{This implication holds more generally for null solutions of
divergence form elliptic equations, see \cite{KKPT} and \cite{HKMP}.}.
The explicit
connection of $\eps$-approximability
with the $A_\infty$ property of harmonic measure, i.e.,
that $(ii)\implies(iii)$, appears in \cite{KKPT}
(where this implication is established not only for the Laplacian,
but for general divergence form elliptic operators).
That $(iii)$ implies $(iv)$ is proved for harmonic functions in \cite{D2}\footnote{And thus all three properties
hold for harmonic functions in Lipschitz domains, by the result of \cite{D1}.}, and, for null solutions of
general divergence form elliptic operators,
in \cite{DJK}.
Finally, Kenig, Kirchheim and Toro \cite{KKT} have recently shown  
that $(i)$ implies $(iii)$ in a Lipschitz domain, whereas, on the other hand,  
$(i)$ may be seen, via good-lambda and John-Nirenberg arguments, to be equivalent to the local version of one 
direction of $(iv)$ (the ``$S<N$" direction)\footnote{The latter equivalence does not require any 
connectivity hypothesis, as we shall
show in a forthcoming sequel to the present paper.}.

The results of the present paper should also be compared to
those of the papers \cite{HMU} and \cite{AHMNT} (see also the earlier paper
\cite{HM-I}) which say, in combination,
that for a ``1-sided NTA" (aka ``uniform") domain $\Omega$
(i.e., a domain in which one has interior Corkscrew and Harnack Chain conditions,
see Definitions \ref{def1.cork}, \ref{def1.hc}), 
with ADR boundary,
then $\pom$ is UR if and only if
$\hm \in A_\infty(\sigma)$ if and only if $\Omega$ is an NTA domain (Definition \ref{def1.nta}). We refer the reader to these papers for details and further historical context.
This chain of implications underlines the strength of the UR of the boundary under the background hypothesis that the domain
is 1-sided NTA, which serves as a scale-invariant connectivity. The present paper, on the other hand, introduces a general mechanism
allowing one to dispose of the connectivity assumption and still obtain Carleson measure bounds and $\eps$-approximability. We would like to emphasize that in this paper we work with a UR set $E$, for which the open set $\ree\setminus E$ fails to satisfy the Harnack chain condition. Otherwise, we would have that $\ree\setminus E$ is a 1-sided NTA domain (the Corkscrew condition holds since $E$ is ADR), and thus NTA, by [AHMNT]. This cannot happen since $\ree\setminus E$ has null exterior.

\subsection{Further Notation and Definitions}

\begin{list}{$\bullet$}{\leftmargin=0.4cm  \itemsep=0.2cm}

\item We use the letters $c,C$ to denote harmless positive constants, not necessarily
the same at each occurrence, which depend only on dimension and the
constants appearing in the hypotheses of the theorems (which we refer to as the
``allowable parameters'').  We shall also
sometimes write $a\lesssim b$ and $a \approx b$ to mean, respectively,
that $a \leq C b$ and $0< c \leq a/b\leq C$, where the constants $c$ and $C$ are as above, unless
explicitly noted to the contrary.  At times, we shall designate by $M$ a particular constant whose value will remain unchanged throughout the proof of a given lemma or proposition, but
which may have a different value during the proof of a different lemma or proposition.

\item Given a closed set $E \subset \ree$, we shall
use lower case letters $x,y,z$, etc., to denote points on $E$, and capital letters
$X,Y,Z$, etc., to denote generic points in $\ree$ (especially those in $\ree\setminus E$).

\item The open $(n+1)$-dimensional Euclidean ball of radius $r$ will be denoted
$B(x,r)$ when the center $x$ lies on $E$, or $B(X,r)$ when the center
$X \in \ree\setminus E$.  A ``surface ball'' is denoted
$\Delta(x,r):= B(x,r) \cap\partial\Omega.$

\item Given a Euclidean ball $B$ or surface ball $\Delta$, its radius will be denoted
$r_B$ or $r_\Delta$, respectively.

\item Given a Euclidean or surface ball $B= B(X,r)$ or $\Delta = \Delta(x,r)$, its concentric
dilate by a factor of $\kappa >0$ will be denoted
$\kappa B := B(X,\kappa r)$ or $\kappa \Delta := \Delta(x,\kappa r).$

\item Given a (fixed) closed set $E \subset \ree$, for $X \in \ree$, we set $\delta(X):= \dist(X,E)$.

\item We let $H^n$ denote $n$-dimensional Hausdorff measure, and let
$\sigma := H^n\big|_{E}$ denote the ``surface measure'' on a closed set $E$
of co-dimension 1.

\item For a Borel set $A\subset \ree$, we let $1_A$ denote the usual
indicator function of $A$, i.e. $1_A(x) = 1$ if $x\in A$, and $1_A(x)= 0$ if $x\notin A$.

\item For a Borel set $A\subset \ree$,  we let $\interior(A)$ denote the interior of $A$.


\item Given a Borel measure $\mu$, and a Borel set $A$, with positive and finite $\mu$ measure, we
set $\fint_A f d\mu := \mu(A)^{-1} \int_A f d\mu$.

\item We shall use the letter $I$ (and sometimes $J$)
to denote a closed $(n+1)$-dimensional Euclidean dyadic cube with sides
parallel to the co-ordinate axes, and we let $\ell(I)$ denote the side length of $I$.
If $\ell(I) =2^{-k}$, then we set $k_I:= k$.
Given an ADR set $E\subset \ree$, we use $Q$ to denote a dyadic ``cube''
on $E$.  The
latter exist (cf. \cite{DS1}, \cite{Ch}), and enjoy certain properties
which we enumerate in Lemma \ref{lemmaCh} below.

\end{list}

\begin{definition} ({\bf Corkscrew condition}).  \label{def1.cork}
Following
\cite{JK}, we say that a domain $\Omega\subset \ree$
satisfies the ``Corkscrew condition'' if for some uniform constant $c>0$ and
for every surface ball $\Delta:=\Delta(x,r),$ with $x\in \partial\Omega$ and
$0<r<\diam(\partial\Omega)$, there is a ball
$B(X_\Delta,cr)\subset B(x,r)\cap\Omega$.  The point $X_\Delta\subset \Omega$ is called
a ``Corkscrew point'' relative to $\Delta.$  We note that  we may allow
$r<C\diam(\pom)$ for any fixed $C$, simply by adjusting the constant $c$.
\end{definition}

\begin{definition}({\bf Harnack Chain condition}).  \label{def1.hc} Again following \cite{JK}, we say
that $\Omega$ satisfies the Harnack Chain condition if there is a uniform constant $C$ such that
for every $\rho >0,\, \Lambda\geq 1$, and every pair of points
$X,X' \in \Omega$ with $\delta(X),\,\delta(X') \geq\rho$ and $|X-X'|<\Lambda\,\rho$, there is a chain of
open balls
$B_1,\dots,B_N \subset \Omega$, $N\leq C(\Lambda)$,
with $X\in B_1,\, X'\in B_N,$ $B_k\cap B_{k+1}\neq \emptyset$
and $C^{-1}\diam (B_k) \leq \dist (B_k,\partial\Omega)\leq C\diam (B_k).$  The chain of balls is called
a ``Harnack Chain''.
\end{definition}

\begin{definition}({\bf NTA}). \label{def1.nta} Again following \cite{JK}, we say that a
domain $\Omega\subset \ree$ is NTA (``Non-tangentially accessible") if it satisfies the
Harnack Chain condition, and if both $\Omega$ and
$\Omega_{\rm ext}:= \ree\setminus \overline{\Omega}$ satisfy the Corkscrew condition.
\end{definition}

\begin{definition}\label{defAinfty}
({\bf $A_\infty$}). \label{ss.ainfty}
Given an ADR set $E\subset\ree$, 
and a surface ball
$\Delta_0:= B_0 \cap E$,
we say that a Borel measure $\mu$ defined on $E$ belongs to
$A_\infty(\Delta_0)$ if there are positive constants $C$ and $\theta$
such that for each surface ball $\Delta = B\cap E$, with $B\subseteq B_0$,
we have
\begin{equation}\label{eq1.wainfty}
\mu (F) \leq C \left(\frac{\sigma(F)}{\sigma(\Delta)}\right)^\theta\,\mu (\Delta)\,,
\qquad \mbox{for every Borel set } F\subset \Delta\,.
\end{equation}
\end{definition}

\begin{lemma}\label{lemmaCh}\textup{({\bf Existence and properties of the ``dyadic grid''})
\cite{DS1,DS2}, \cite{Ch}.}
Suppose that $E\subset \ree$ is closed $n$-dimensional ADR set.  Then there exist
constants $ a_0>0,\, \gamma>0$ and $C_1<\infty$, depending only on dimension and the
ADR constant, such that for each $k \in \mathbb{Z},$
there is a collection of Borel sets (``cubes'')
$$
\mathbb{D}_k:=\{Q_{j}^k\subset E: j\in \mathfrak{I}_k\},$$ where
$\mathfrak{I}_k$ denotes some (possibly finite) index set depending on $k$, satisfying

\begin{list}{$(\theenumi)$}{\usecounter{enumi}\leftmargin=.8cm
\labelwidth=.8cm\itemsep=0.2cm\topsep=.1cm
\renewcommand{\theenumi}{\roman{enumi}}}

\item $E=\cup_{j}Q_{j}^k\,\,$ for each
$k\in{\mathbb Z}$.

\item If $m\geq k$ then either $Q_{i}^{m}\subset Q_{j}^{k}$ or
$Q_{i}^{m}\cap Q_{j}^{k}=\emptyset$.

\item For each $(j,k)$ and each $m<k$, there is a unique
$i$ such that $Q_{j}^k\subset Q_{i}^m$.

\item $\diam\big(Q_{j}^k\big)\leq C_1 2^{-k}$.

\item Each $Q_{j}^k$ contains some ``surface ball'' $\Delta \big(x^k_{j},a_02^{-k}\big):=
B\big(x^k_{j},a_02^{-k}\big)\cap E$.

\item $H^n\big(\big\{x\in Q^k_j:{\rm dist}(x,E\setminus Q^k_j)\leq \varrho \,2^{-k}\big\}\big)\leq
C_1\,\varrho^\gamma\,H^n\big(Q^k_j\big),$ for all $k,j$ and for all $\varrho\in (0,a_0)$.
\end{list}
\end{lemma}

A few remarks are in order concerning this lemma.

\begin{list}{$\bullet$}{\leftmargin=0.4cm  \itemsep=0.2cm}

\item In the setting of a general space of homogeneous type, this lemma has been proved by Christ
\cite{Ch}, with the
dyadic parameter $1/2$ replaced by some constant $\delta \in (0,1)$.
In fact, one may always take $\delta = 1/2$ (cf.  \cite[Proof of Proposition 2.12]{HMMM}).
In the presence of the Ahlfors-David
property (\ref{eq1.ADR}), the result already appears in \cite{DS1,DS2}.

\item  For our purposes, we may ignore those
$k\in \mathbb{Z}$ such that $2^{-k} \gtrsim {\rm diam}(E)$, in the case that the latter is finite.

\item  We shall denote by  $\mathbb{D}=\mathbb{D}(E)$ the collection of all relevant
$Q^k_j$, i.e., $$\mathbb{D} := \cup_{k} \mathbb{D}_k,$$
where, if $\diam (E)$ is finite, the union runs
over those $k$ such that $2^{-k} \lesssim  {\rm diam}(E)$.

\item Properties $(iv)$ and $(v)$ imply that for each cube $Q\in\mathbb{D}_k$,
there is a point $x_Q\in E$, a Euclidean ball $B(x_Q,r)$ and a surface ball
$\Delta(x_Q,r):= B(x_Q,r)\cap E$ such that
$r\approx 2^{-k} \approx {\rm diam}(Q)$
and \begin{equation}\label{cube-ball}
\Delta(x_Q,r)\subset Q \subset \Delta(x_Q,Cr),\end{equation}
for some uniform constant $C$.
We shall denote this ball and surface ball by
\begin{equation}\label{cube-ball2}
B_Q:= B(x_Q,r) \,,\qquad\Delta_Q:= \Delta(x_Q,r),\end{equation}
and we shall refer to the point $x_Q$ as the ``center'' of $Q$.

\item For a dyadic cube $Q\in \mathbb{D}_k$, we shall
set $\ell(Q) = 2^{-k}$, and we shall refer to this quantity as the ``length''
of $Q$.  Evidently, $\ell(Q)\approx \diam(Q).$

\item For a dyadic cube $Q \in \mathbb{D}$, we let $k(Q)$ denote the ``dyadic generation''
to which $Q$ belongs, i.e., we set  $k = k(Q)$ if
$Q\in \mathbb{D}_k$; thus, $\ell(Q) =2^{-k(Q)}$.

\end{list}

\section{A bilateral corona decomposition}


In this section, we prove a bilateral version of the ``corona decomposition" of David and Semmes
\cite{DS1,DS2}.  Before doing that let us introduce the notions of ``coherency'' and ``semi-coherency'':

\begin{definition}\cite{DS2}.\label{d3.11}   
Let $\sbf\subset \dd(E)$. We say that $\sbf$ is
``coherent" if the following conditions hold:
\begin{itemize}\itemsep=0.1cm

\item[$(a)$] $\sbf$ contains a unique maximal element $Q(\sbf)$ which contains all other elements of $\sbf$ as subsets.

\item[$(b)$] If $Q$  belongs to $\sbf$, and if $Q\subset \widetilde{Q}\subset Q(\sbf)$, then $\widetilde{Q}\in {\bf S}$.

\item[$(c)$] Given a cube $Q\in \sbf$, either all of its children belong to $\sbf$, or none of them do.

\end{itemize}
We say that $\sbf$ is ``semi-coherent'' if only conditions $(a)$ and $(b)$ hold. 
\end{definition}

\smallskip

We are now ready to state our bilateral ``corona decomposition''.
\begin{lemma}\label{lemma2.1}  Suppose that $E\subset \ree$ is $n$-dimensional UR.  Then given any positive constants
$\eta\ll 1$
and $K\gg 1$, there is a disjoint decomposition
$\dd(E) = \G\cup\B$, satisfying the following properties.
\begin{enumerate}
\item  The ``Good"collection $\G$ is further subdivided into
disjoint stopping time regimes, such that each such regime {\bf S} is coherent (cf. Definition \ref{d3.11}).

\item The ``Bad" cubes, as well as the maximal cubes $Q(\sbf)$ satisfy a Carleson
packing condition:
$$\sum_{Q'\subset Q, \,Q'\in\B} \sigma(Q')
\,\,+\,\sum_{\sbf: Q(\sbf)\subset Q}\sigma\big(Q(\sbf)\big)\,\leq\, C_{\eta,K}\, \sigma(Q)\,,
\quad \forall Q\in \dd(E)\,.$$
\item For each $\sbf$, there is a Lipschitz graph $\Gamma_{\sbf}$, with Lipschitz constant
at most $\eta$, such that, for every $Q\in \sbf$,
\begin{equation}\label{eq2.2a}
\sup_{x\in \Delta_Q^*} \dist(x,\Gamma_{\sbf} )\,
+\,\sup_{y\in B_Q^*\cap\Gamma_{\sbf}}\dist(y,E) < \eta\,\ell(Q)\,,
\end{equation}
where $B_Q^*:= B(x_Q,K\ell(Q))$ and $\Delta_Q^*:= B_Q^*\cap E$.
\end{enumerate}
\end{lemma}

Before proving the lemma, we recall the ``Bilateral Weak Geometric Lemma"
\cite[p. 32]{DS2}.
\begin{lemma}[\cite{DS2}]\label{BWGL} Let $E\subset \ree$ be a closed, $n$-dimensional ADR set.
Then $E$ is UR if and only if for every pair of positive constants $\eta \ll1$
and $K\gg 1$,
there is a disjoint decomposition
$\dd(E) = \G_0\cup\B_0$,
such that the cubes in $\B_0$ satisfy a Carleson packing condition
\begin{equation}\label{eq2.3}\sum_{Q'\subset Q, \,Q'\in\B_0} \sigma(Q')
\,\leq\, C_{\eta,K}\, \sigma(Q)\,,
\quad \forall Q\in \dd(E)\,,
\end{equation}
and such that for every $Q\in \G_0$, we have
\begin{equation}\label{eq2.4}
\inf_H\left(\sup_{x\in\Delta_Q^*} \dist(x,H)\,+\, \sup_{y\in H\cap B_Q^*} \dist(y,E)\right)
\,<\, \eta\, \ell(Q)\,,
\end{equation}
where the infimum runs over all hyperplanes $H$, and where $B_Q^*$ and $\Delta_Q^*$ are defined as in Lemma \ref{lemma2.1}.
\end{lemma}

\begin{proof}[Proof of Lemma \ref{lemma2.1}]
A ``unilateral" version of Lemma \ref{lemma2.1}
has already appeared in \cite{DS1}, i.e.,
by \cite{DS1}, we know that Lemma \ref{lemma2.1} holds,
but with the bilateral estimate \eqref{eq2.2a} 
replaced by the unilateral bound
\begin{equation}\label{eq2.5}
\sup_{x\in \Delta_Q^*} \dist(x,\Gamma_{\sbf} )\,
<\, \eta\,\ell(Q)\,,\qquad \forall\,Q\in \sbf.
\end{equation}
The proof of Lemma \ref{lemma2.1} will be a rather straightforward
combination of this result  of \cite{DS1}, and Lemma \ref{BWGL}.

We choose $K_1\gg 1$, and $\eta_1\ll K_1^{-1}$,
and let $\dd=\G_1\cup\B_1$, and $\dd=\G_0\cup\B_0$, be,
respectively, the unilateral corona decomposition of \cite{DS1},
and the decomposition of Lemma \ref{BWGL}, corresponding to this choice of
$\eta$ and $K$.  Given $\sbf$, a stopping time regime
of the unilateral corona decomposition, we let
$\MS$ denote the set of $Q\in\sbf\cap\G_0$ for which
either $Q=Q(\sbf)$, or else the dyadic parent of $Q$, or  one of the brothers of $Q$,
belongs to $\B_0$.  For each $Q\in \MS$, we form a new stopping time regime,
call it $\sbf'$, as follows.  We set $Q(\sbf'):= Q$, and we then subdivide
$Q(\sbf')$ dyadically, stopping as soon as we reach a subcube $Q'$ such that
either $Q'\notin \sbf$, or else $Q'$, or one of its brothers, belongs to
$\B_0$.  In any such scenario, $Q'$ and all of its brothers are omitted from
$\sbf'$, and the parent of $Q'$ is then a minimal cube
of $\sbf'$.
We note that each such $\sbf'$ enjoys the following properties:
\begin{enumerate}
\item[(i)] $\sbf'\subset \sbf\cap\G_0$ (by definition).
\item[(ii)] $\sbf'$ is coherent, in the sense of  Lemma \ref{lemma2.1} (1)
(by the stopping time construction).
\end{enumerate}

If  $Q\in \sbf\cap\B_0$, for some $\sbf$,
then we add $Q$ to our new ``bad" collection,
call it $\B$,
i.e.,  $\B=\B_1\cup\B_0$.  Then clearly $\B$ satisfies a packing condition,
since it is the union of two collections, each of which packs.  Moreover,
 the collection
$\{Q(\sbf')\}_{\sbf'}$ satisfies a packing condition.  Indeed, by construction
$$\{Q(\sbf')\}_{\sbf'} \subset \{Q(\sbf)\}_{\sbf} \cup \M_1\,,$$
where $\M_1$ denotes the collection of cubes $Q$ having a parent or brother in $\B_0$.
Now for $\{Q(\sbf)\}_{\sbf}$ we already have packing.  For the cubes in $\M_1$,
and for any $R\in \dd(E)$, with dyadic parent $R^*$,
we have 
$$
\sum_{Q\in\M_1:\, Q\subset R} \sigma(Q)
\,\lesssim
\sum_{\Qt\in\B_0:\, \Qt\subset R^*}\sigma(\Qt)\,\lesssim
\sigma(R^*)\,\lesssim \sigma(R)\,,
$$
where $\Qt$ is either the parent or a brother of $Q$, belonging to $\B_0$,
and where we have used the packing condition for $\B_0$, and the doubling property
of $\sigma$.
Setting $\G:= \dd(E) \setminus \B$, we note that at this point we have verified
properties (1) and (2) of Lemma
\ref{lemma2.1}, for the decomposition $\dd(E) =\G\cup\B$, and the stopping time
regimes $\{\sbf'\}$.  It remains to verify property (3).

To this end, we consider one of the new stopping time regimes $\sbf'$, which by construction,
is contained in some $\sbf$.
Set $\Gamma_{\sbf'}:=\Gamma_{\sbf}$, and
fix $Q\in \sbf'$.  Let us now
prove \eqref{eq2.2a}.  The bound
\begin{equation}\label{eq2.7}
\sup_{x\in \Delta_Q^*} \dist(x,\Gamma_{\sbf'} )\,
<\, \eta_1\,\ell(Q)\,
\end{equation}
is inherited immediately from the unilateral condition \eqref{eq2.5}.  We now claim that
for $\eta_1\ll K_1^{-1}$,
\begin{equation}\label{eq2.8}
\sup_{y\in \frac12B_Q^*\cap \Gamma_{\sbf}} \dist(y,E )\,
<\, C K_1 \eta_1\,\ell(Q)\,.
\end{equation}
Taking the claim for granted momentarily, and having specified some $\eta,\,K$, we may obtain
\eqref{eq2.2a}  by choosing $K_1:= 2K$, and $\eta_1:= \eta/(CK_1)<\eta$.

We now establish the claim.
By construction of $\sbf'$, $Q\in \G_0$, so by \eqref{eq2.4},
there is a hyperplane $H_Q$ such that
\begin{equation}\label{eq2.9}
\sup_{x\in \Delta_Q^*}\dist(x,H_Q)\,+\,
\sup_{y\in H_Q\cap B_Q^*}\dist(y,E) \,<\, \eta_1\ell(Q)\,.
\end{equation}
There is another hyperplane $H_{\sbf} = H_{\sbf'}$ such that,
with respect to the co-ordinate system $\{(z,t): z\in H_{\sbf}, t\in \RR\}$, we can realize
$\Gamma_{\sbf}$ as a Lipschitz graph with
constant no larger than
$\eta_1$, i.e., $\Gamma_{\sbf}= \{(z,\vp_{\sbf}(z)):\, z\in H_{\sbf}\}$, with
$\|\vp\|_{\rm Lip} \leq \eta_1$.
Let $\pi_Q$ be the orthogonal projection onto $H_Q$,
and set $\hx:= \pi_Q(x_Q)$.
Thus $|x_Q-\hx|<\eta_1\ell(Q)$, by \eqref{eq2.9}.  Consequently,
for $\eta_1$ small, we have
$$B_1:= B\left(\hx, \frac34 K_1\ell(Q)\right)\, \subset\, \frac78 B_Q^*\,,$$
and
\begin{equation}\label{eq2.9a}
\frac12 B_Q^*\subset \frac78 B_1\,.
\end{equation}
Therefore, by \eqref{eq2.9}
\begin{equation}\label{eq2.10}
\dist(y,E) \leq\eta_1\ell(Q)\,,\qquad \forall y\in B_1\cap H_Q\,,
\end{equation}
and also, for $K_1$ large,
\begin{equation*}
A_1:=\left\{x\in E:\,\dist(x,B_1\cap H_Q)\leq \ell(Q)\right\}\subset \frac {15}{16} B_Q^*\,.
\end{equation*}
 Thus, $A_1\subset \Delta_Q^*$, so that, in particular, for
 $x\in A_1$, we have $\dist(x,\Gamma_{\sbf}) <\eta_1\ell(Q)$, by \eqref{eq2.5}.
 Combining the latter fact with \eqref{eq2.10} and the definition of $A_1$, we find that
\begin{equation}\label{eq2.11}
\dist(y,\Gamma_{\sbf}) \leq 2\eta_1\ell(Q)\,,\qquad \forall y \in B_1\cap H_Q\,.
\end{equation}
 We 
 cover $(7/8)B_1\cap H_Q$ 
 by non-overlapping $n$-dimensional cubes $P_k\subset B_1\cap H_Q$,
 centered at $y_k$, with side length $10\eta_1 \ell(Q)$, and we extend these
along an axis perpendicular to $H_Q$ to construct
$(n+1)$-dimensional cubes $I_k$,
  of the same length, also centered at $y_k$.
 By \eqref{eq2.11}, each $I_k$ meets $\Gamma_{\sbf}$.  Therefore, for $\eta_1$ small,
 $H_Q$ ``meets'' $H_{\sbf}$ at an angle $\theta$ satisfying
 $$\theta\approx\tan \theta \lesssim \eta_1\,,$$
 and $\Gamma_{\sbf}$ is a Lipschitz graph with respect to $H_Q$,
 with Lipschitz constant no larger than $C\eta_1$.
Also, by \eqref{eq2.11}, applied to $y=\hx$, there is a point $y_Q\in \Gamma_{\sbf}$ with
 $|\hx-y_Q|\leq 2\eta_1\ell(Q)$.
 Thus, for $y\in (1/2)B_Q^*\cap \Gamma_{\sbf}\subset (7/8) B_1\cap \Gamma_{\sbf}$
 (where we have used \eqref{eq2.9a}), we have
\begin{equation}\label{eq2.13}
 \dist(y,H_Q) \leq C K_1\eta_1\ell(Q)\ll\ell(Q)\,,
 \end{equation}
so that $\pi_Q(y)\in B_1\cap H_Q \subset B_Q^* \cap H_Q$.
Hence,
\begin{equation}\label{eq2.14}
\dist(\pi_Q(y),E)\,\leq \,\eta_1\ell(Q)\,,
\end{equation}
by \eqref{eq2.9}.  Combining \eqref{eq2.13} and \eqref{eq2.14},
we obtain \eqref{eq2.8}, as claimed.
 \end{proof}

\section{Corona type approximation by NTA domains with ADR boundaries}\label{s3}
In this section, we construct, for each stopping time regime $\sbf$
in Lemma \ref{lemma2.1}, a pair of NTA domains $\Omega_{\sbf}^\pm$, with ADR boundaries, which
provide a good approximation to $E$, at the scales within $\sbf$, in some appropriate sense.  To be a bit more precise,
$\Omega_{\sbf}:= \Omega_{\sbf}^+\cup \Omega_{\sbf}^-$ will be constructed as a sawtooth region
relative to some family of dyadic cubes, and the nature of this construction will be essential to the dyadic analysis that we will use below.
We first discuss some preliminary matters.

Let $\mathcal{W}=\W(\ree\setminus E)$ denote a collection
of (closed) dyadic Whitney cubes of $\ree\setminus E$, so that the cubes in $\mathcal{W}$
form a pairwise non-overlapping covering of $\ree\setminus E$, which satisfy
\begin{equation}\label{Whintey-4I}
4 \diam(I)\leq
\dist(4I,E)\leq \dist(I,E) \leq 40\diam(I)\,,\qquad \forall\, I\in \mathcal{W}\,\end{equation}
(just dyadically divide the standard Whitney cubes, as constructed in  \cite[Chapter VI]{St},
into cubes with side length 1/8 as large)
and also
$$(1/4)\diam(I_1)\leq\diam(I_2)\leq 4\diam(I_1)\,,$$
whenever $I_1$ and $I_2$ touch.


Let $E$ be an $n$-dimensional ADR set and pick two parameters $\eta\ll 1$ and $K\gg 1$. Define
\begin{equation}\label{eq3.1}
\W^0_Q:= \left\{I\in \W:\,\eta^{1/4} \ell(Q)\leq \ell(I)
 \leq K^{1/2}\ell(Q),\ \dist(I,Q)\leq K^{1/2} \ell(Q)\right\}.
 \end{equation}
 
\begin{remark}\label{remark:E-cks} 
We note that $\W^0_Q$ is non-empty,
 provided that we choose $\eta$ small enough, and $K$ large enough, depending only on dimension
 and the ADR constant of $E$.
 Indeed, given  a closed $n$-dimensional ADR set $E$, and
given $Q\in\dd(E)$, consider the ball $B_Q=B(x_Q,r)$, as defined in \eqref{cube-ball}-\eqref{cube-ball2},
with $r\approx\ell(Q)$, so that $\Delta_Q=B_Q\cap E\subset Q$.
By
 \cite[Lemma 5.3]{HM-I} , we have that for some $C=C(n,ADR)$,
$$
\big|\{Y\in\ree\setminus E: \,\delta(Y)<\epsilon r\}\cap B_Q\big|\le C\,\epsilon\,r^{n+1}\,,
$$
for every $0<\epsilon<1$. Consequently, fixing $0<\epsilon_0<1$ small enough, 
there exists $X_Q\in \frac12\,B_Q$, with $\delta(X_Q)\ge \epsilon_0\,r$. 
Thus, $B(X_Q,\epsilon_0\,r/2)\subset B_Q\setminus E$.  We shall refer to this point $X_Q$
as a ``Corkscrew point" relative to $Q$.  Now observe that $X_Q$ belongs to some
Whitney cube $I\in\W$, which will belong to $\W^0_Q$, for $\eta$ small enough and $K$ large enough.
 \end{remark}

Assume now that $E$ is UR  and make the corresponding bilateral corona decomposition of Lemma \ref{lemma2.1} with $\eta\ll 1$ and $K\gg 1$.
Given $Q\in \dd(E)$, for this choice of $\eta$ and $K$, we set (as above) $B_Q^*:= B(x_Q, K\ell(Q))$, where we recall that $x_Q$ is the ``center" of $Q$ (see \eqref{cube-ball}-\eqref{cube-ball2}). For a fixed stopping time regime $\sbf$, we choose a co-ordinate system
 so that $\Gamma_{\sbf} =\{(z,\vp_{\sbf}(z)):\, z\in \rn\}$, where $\vp_{\sbf}:\rn\mapsto
 \re$ is a Lipschitz function with
 $\|\vp\|_{\rm Lip} \leq\eta$.

\begin{claim}\label{c3.1} If $Q\in \sbf$, and $I\in \W^0_Q$, then $I$ lies either above
or below $\Gamma_{\sbf}$.  Moreover,
$\dist(I,\Gamma_{\sbf})\geq \eta^{1/2} \ell(Q)$ (and therefore, by \eqref{eq2.2a},  
$\dist(I,\Gamma_{\sbf}) \approx \dist(I,E)$, with  implicit constants that may depend on $\eta$
and $K$).
\end{claim}

\begin{proof}[Proof of Claim \ref{c3.1}]
Suppose by way of contradiction that $\dist(I,\Gamma_{\sbf}) \leq \eta^{1/2}\ell(Q)$.
Then we may choose $y\in \GS$ such that
$$\dist(I,y)\leq\eta^{1/2}\ell(Q)\,.$$
By construction of $\W^0_Q$, it follows that for all $Z\in I$,
$|Z-y|\lesssim K^{1/2}\ell(Q)$.
Moreover, $|Z-x_Q|\lesssim K^{1/2}\ell(Q)$, and therefore
$|y-x_Q|\lesssim K^{1/2}\ell(Q)$.  In particular, $y\in B_Q^*\cap\GS$,
so by \eqref{eq2.2a}, $\dist(y,E)\leq\eta\,\ell(Q)$.  On the other hand,
choosing $Z_0\in I$  such that $|Z_0-y|=\dist(I,y) \leq\eta^{1/2}\ell(Q)$, we obtain
$\dist(I,E)\leq 2\eta^{1/2}\ell(Q)$.  For $\eta$ small, this contradicts the
Whitney construction, since $\dist(I,E) \approx \ell(I) \geq \eta^{1/4}\ell(Q)$.
\end{proof}

Next, given $Q\in\sbf$,  we augment $\W^0_Q$.  We split $\W^0_Q=\W_Q^{0,+} \cup \W_Q^{0,-}$,
where $I\in\W_Q^{0,+}$ if $I$ lies above
$\GS$, and  $I\in\W_Q^{0,-}$ if $I$ lies below
$\GS$.   Choosing $K$ large and $\eta$ small enough, by \eqref{eq2.2a}, 
we may assume that both $\W_Q^{0,\pm}$ are non-empty.
We focus on $\W_Q^{0,+}$, as the construction for $\W_Q^{0,-}$ is the same.
For each $I\in \W_Q^{0,+}$, let $X_I$ denote the center of $I$.
Fix one particular $I_0\in \W_Q^{0,+}$, with center $X^+_Q:= X_{I_0}$.
Let $\widetilde{Q}$ denote the dyadic parent of $Q$, unless $Q=Q(\sbf)$; in the latter case
we simply set $\Qt=Q$.  Note that
$\widetilde{Q}\in\sbf$,
by the coherency of $\sbf$.
By Claim \ref{c3.1}, for each $I$ in $\W_Q^{0,+}$, or in   $\W_{\widetilde{Q}}^{0,+}$,
we have
$$\dist(I,E)\approx\dist(I,Q)\approx\dist(I,\GS)\,,$$
where the implicit constants may depend on $\eta$ and $K$.  Thus, for each
such $I$, we may fix a Harnack chain, call it
$\mathcal{H}_I$, relative to the Lipschitz domain
$$\Omega_{\GS}^+:=\left\{(x,t) \in \ree: t>\vp_{\sbf}(x)\right\}\,,$$
connecting $X_I$ to $X_Q^+$.  By the bilateral approximation condition
\eqref{eq2.2a}, the definition of $\W^0_Q$, and the fact that $K^{1/2}\ll K$,
we may construct this Harnack Chain so that it consists of a bounded
number of balls (depending on $\eta$ and $K$), and stays a distance at least
$c\eta^{1/2}\ell(Q)$ away from $\GS$ and from $E$.
We let $\W^{*,+}_Q$ denote the set of all $J\in\W$ which meet at least one of the Harnack
chains $\mathcal{H}_I$, with $I\in \W_Q^{0,+}\cup\W_{\widetilde{Q}}^{0,+}$
(or simply $I\in \W_Q^{0,+}$, if $Q=Q(\sbf)$), i.e.,
$$\W^{*,+}_Q:=\left\{J\in\W: \,\exists \, I\in \W_Q^{0,+}\cup\W_{\Qt}^{0,+}
\,\,{\rm for\, which}\,\, \mathcal{H}_I\cap J\neq \emptyset \right\}\,,$$
where as above, $\widetilde{Q}$ is the dyadic parent of $Q$, unless
$Q= Q(\sbf)$, in which case we simply set $\Qt=Q$ (so the union is redundant).
We observe that, in particular, each $I\in \W^{0,+}_Q\cup \W^{0,+}_{\Qt}$
meets $\mathcal{H}_I$, by definition, and therefore
\begin{equation}\label{eqW}
\W_Q^{0,+}\cup \W^{0,+}_{\Qt} \subset \W_Q^{*,+}\,.
\end{equation}
Of course, we may construct $\W^{*,-}_Q$ analogously.
We then set
$$\W^*_Q:=\W^{*,+}_Q\cup \W^{*,-}_Q\,.$$
It follows from the construction of the augmented collections $\W_Q^{*,\pm}$ 
that there are uniform constants $c$ and $C$ such that
\begin{eqnarray}\label{eq2.whitney2}
& c\eta^{1/2} \ell(Q)\leq \ell(I) \leq CK^{1/2}\ell(Q)\,, \quad \forall I\in \mathcal{W}^*_Q,\\\nonumber
&\dist(I,Q)\leq CK^{1/2} \ell(Q)\,,\quad\forall I\in \mathcal{W}^*_Q.
\end{eqnarray}

Observe that $\W_Q^{*,\pm}$ and hence also $\W^*_Q$
have been defined for any $Q$ that belongs to some stopping time regime $\sbf$,
that is, for any $Q$ belonging to the ``good" collection $\G$ of Lemma \ref{lemma2.1}.  On the other hand, we
have defined $\W_Q^0$ for {\it arbitrary} $Q\in \dd(E)$.

We now set
\begin{equation}\label{Wdef}
\W_Q:=\left\{
\begin{array}{l}
\W_Q^*\,,
\,\,Q\in\G,
\\[6pt]
\W_Q^0\,,
\,\,Q\in\B
\end{array}
\right.\,,
\end{equation}
and for $Q \in\G$ we shall henceforth simply write $\W_Q^\pm$ in place of $\W_Q^{*,\pm}$.

Next, we choose a small parameter $\tau_0>0$, so that
for any $I\in \W$, and any $\tau \in (0,\tau_0]$,
the concentric dilate
$I^*(\tau):= (1+\tau) I$ still satisfies the Whitney property
\begin{equation}\label{whitney}
\diam I\approx \diam I^*(\tau) \approx \dist\left(I^*(\tau), E\right) \approx \dist(I,E)\,, \quad 0<\tau\leq \tau_0\,.
\end{equation}
Moreover,
for $\tau\leq\tau_0$ small enough, and for any $I,J\in \W$,
we have that $I^*(\tau)$ meets $J^*(\tau)$ if and only if
$I$ and $J$ have a boundary point in common, and that, if $I\neq J$,
then $I^*(\tau)$ misses $(3/4)J$.
Given an  arbitrary $Q\in\dd(E)$, we may define an associated
Whitney region $U_Q$ (not necessarily connected), as follows:
\begin{equation}\label{eq3.3aa}
U_Q=U_{Q,\tau}:= \bigcup_{I\in \W_Q} I^*(\tau)
\end{equation}
For later use, it is also convenient to introduce some fattened version of $U_Q$:  if $0<\tau\le \tau_0/2$,
\begin{equation}\label{eq3.3aa-fat}
\widehat{U}_Q=U_{Q,2\,\tau}:= \bigcup_{I\in \W_Q} I^*(2\,\tau). 
\end{equation}
If $Q\in\G$, then $U_Q$ splits into exactly two connected components
\begin{equation}\label{eq3.3b}
U_Q^\pm= U_{Q,\tau}^{\pm}:= \bigcup_{I\in \W^{\pm}_Q} I^*(\tau)\,.
\end{equation}
When the particular choice of $\tau\in (0,\tau_0]$ is not important,
for the sake of notational convenience, we may
simply write $I^*$, $U_Q$,  and $U_Q^{\pm}$ in place of
 $I^*(\tau)$, $U_{Q,\tau}$, and $U_{Q,\tau}^{\pm}$.
We note that for $Q\in\G$, each $U_Q^{\pm}$ is Harnack chain connected, by construction
(with constants depending on the implicit parameters $\tau, \eta$ and $K$);
moreover, for a fixed stopping time regime $\sbf$,
if $Q'$ is a child of $Q$, with both $Q',\,Q\in \sbf$, then
$U_{Q'}^{+}\cup U_Q^{+}$
is Harnack Chain connected, and similarly for
$U_{Q'}^{-}\cup U_Q^{-}$.  

We may also define ``Carleson Boxes" relative to any $Q\in\dd(E)$, by
\begin{equation}\label{eq3.3a}
T_Q=T_{Q,\tau}:=\interior\left(\bigcup_{Q'\in\dd_Q} U_{Q,\tau}\right)\,,
\end{equation}
where
\begin{equation}\label{eq3.4a}
\dd_Q:=\left\{Q'\in\dd(E):Q'\subset Q\right\}\,.
\end{equation}
Let us note that we may choose $K$ large enough so that, for every $Q$,
\begin{equation}\label{eq3.3aab}
T_Q \subset B_Q^*:= B\left(x_Q,K\ell(Q)\right)\,.
\end{equation}

For future reference, we also introduce dyadic sawtooth regions as follows.
Given a family $\mathcal{F}$ of disjoint cubes $\{Q_j\}\subset \mathbb{D}$, we define
the {\bf global discretized sawtooth} relative to $\F$ by
\begin{equation}\label{eq2.discretesawtooth1}
\dd_{\F}:=\dd\setminus \bigcup_{\F} \dd_{Q_j}\,,
\end{equation}
i.e., $\dd_{\F}$ is the collection of all $Q\in\dd$ that are not contained in any $Q_j\in\F$.
Given some fixed cube $Q$,
the {\bf local discretized sawtooth} relative to $\F$ by
\begin{equation}\label{eq2.discretesawtooth2}
\dd_{\F,Q}:=\dd_Q\setminus \bigcup_{\F} \dd_{Q_j}=\dd_\F\cap\dd_Q.
\end{equation}
Note that in this way $\dd_Q=\dd_{\textup{\O},Q}$.

Similarly, we may define geometric sawtooth regions as follows.
Given a family $\mathcal{F}$ of disjoint cubes $\{Q_j\}\subset \mathbb{D}$,
we define the {\bf global sawtooth} and the {\bf local sawtooth} relative to $\mathcal{F}$ by respectively
\begin{equation}\label{eq2.sawtooth1}
\Omega_{\mathcal{F}}:= {\rm int } \bigg( \bigcup_{Q'\in\dd_\F} U_{Q'}\bigg)\,,
\qquad
\Omega_{\mathcal{F},Q}:=  {\rm int } \bigg( \bigcup_{Q'\in\dd_{\F,Q}} U_{Q'}\bigg)\,.
\end{equation}
Notice that $\Omega_{\textup{\O},Q}=T_Q $.
For the sake of notational convenience, given a pairwise disjoint family $\F\in \dd$, and a cube $Q\in \dd_{\F}$, we set
\begin{equation}\label{Def-WF}
\W_{\F}:=\bigcup_{Q'\in\dd_{\F}}\W_{Q'}\,,\qquad
\W_{\F,Q}:=\bigcup_{Q'\in\dd_{\F,Q}}\W_{Q'}\,,\end{equation}
so that in particular, we may write
\begin{equation}\label{eq3.saw}
\Omega_{\mathcal{F},Q}={\rm int }\,\bigg(\bigcup_{I\in\,\W_{\F,Q}} I^*\bigg)\,.
\end{equation}

It is convenient at this point to introduce some additional terminology.
\begin{definition}\label{d3.2}  Given $Q\in \G$, and hence in  some $\sbf$, we shall refer to the point
$X_Q^+$ specified above, as the ``center" of $U^{+}_Q$ (similarly, the analogous
point $X_Q^-$, lying below
$\GS$, is the ``center" of $U^{-}_Q$).  We also set $Y_Q^\pm := X^\pm_{\Qt}$,
and we call this point the ``modified center" of $U_Q^\pm$, where as above $\Qt$ is the dyadic parent of
$Q$, unless $Q=Q(\sbf)$, in which case $Q=\Qt$, and $Y_Q^\pm=X_Q^\pm$.
\end{definition}

\smallskip

\begin{remark}\label{r3.11a}
We recall that, by construction
(cf. \eqref{eqW}, \eqref{Wdef}), $\W_{\Qt}^{0,\pm}\subset \W_Q$, and therefore
$Y_Q^\pm \in U_Q^\pm\cap U^\pm_{\Qt}$.  Moreover, since $Y_Q^\pm$ is the center of some $I\in \W_{\Qt}^{0,\pm}$,
we have that $\dist(Y_Q^\pm, \partial U_Q^\pm)\approx \dist(Y_Q^\pm, \partial U_{\Qt}^\pm) \approx \ell(Q)$
(with implicit constants possibly depending on
$\eta$ and/or $K$)
\end{remark}

\smallskip

\begin{remark}\label{r3.11}
Given a stopping time regime $\sbf$ as in Lemma \ref{lemma2.1}, for any semi-coherent
subregime (cf. Definition \ref{d3.11}) $\sbf'\subset \sbf$ (including, of course, $\sbf$ itself), we now set
\begin{equation}\label{eq3.2}
\Omega_{\sbf'}^\pm = {\rm int}\left(\bigcup_{Q\in\sbf'} U_Q^{\pm}\right)\,,
\end{equation}
and let $\Omega_{\sbf'}:= \Omega_{\sbf'}^+\cup\Omega^-_{\sbf'}$.
Note that implicitly, $\Omega_{\sbf'}$ depends upon $\tau$ (since $U_Q^\pm$ has such dependence).
When it is necessary to consider the value of $\tau$
explicitly, we shall write $\Omega_{\sbf'}(\tau)$.
\end{remark}

Our main geometric lemma is the following.
\begin{lemma}\label{lemma3.15} Let $\sbf$ be a given
stopping time regime as in Lemma \ref{lemma2.1},
and let $\sbf'$ be any nonempty, semi-coherent  subregime of $\sbf$.
Then for $0<\tau\leq\tau_0$, with $\tau_0$ small enough,
each of $\Omega^\pm_{\sbf'}$ is an NTA domain, with ADR boundary.
The constants in the NTA and ADR conditions depend only on $n,\tau,\eta, K$, and the
ADR/UR constants for $E$.
\end{lemma}

\begin{proof}  We fix a small $\tau>0$ as above, defining the dilated Whitney cubes $I^*$,
and we leave this parameter implicit.

We note that in the notation of \eqref{eq2.sawtooth1},
$\Omega_{\sbf'}$ is the dyadic sawtooth region
$\Omega_{\F,Q(\sbf')}$, where $Q(\sbf')$ is the maximal cube in
$\sbf'$, and $\F$ is the family consisting of the sub-cubes of $Q(\sbf')$
that are maximal with respect to non-membership in $\sbf'$.
Then $\pom_{\sbf'}$ satisfies the ADR property, by Appendix \ref{appa} below.
The upper ADR bound for each of $\pospp$ and $\pom_{\sbf'}^-$ is then trivially
inherited from that of $\pom_{\sbf'}$  and $E$. 
 With the upper ADR property in hand,
we obtain that in particular, each of $\Omega_{\sbf'}^\pm$ is a domain of locally finite perimeter,
by the criterion in \cite[p. 222]{EG}.
The lower ADR bound then follows immediately from the local isoperimetric
inequality  \cite[p. 190]{EG}, once we have established that each of $\Omega_{\sbf'}^\pm$ enjoys a 2-sided Corkscrew condition.  Alternatively, the lower ADR bound for  $\Omega_{\sbf'}^\pm$
can be deduced by carefully following the relevant arguments in Appendix \ref{appa},
and observing that they can be applied to each of $\Omega_{\sbf'}^\pm$ individually.

We now verify the NTA properties for $\Omega_{\sbf'}^+$ (the proof for $\Omega^-_{\sbf'}$ is the same).

\noindent{\it Corkscrew condition}.   We will show that $B(x,r)$ contains
both interior and exterior Corkscrew points for $\ospp$,
for any $x\in \pospp$, and
$0<r\leq 2\diam Q(\sbf')$.  Let $M$ be a large number to be chosen,
depending only on the various parameters given in the statement of the lemma.
There are several cases.  We recall that $\delta(X):=\dist(X,E)$.

\noindent{\bf Case 1}:  $r<M\delta(x)$.  In this case, $x$ lies on a face of a fattened
Whitney cube $I^*$ whose interior lies in $\ospp$, but also $x\in J$ for some $J\notin
\W(\sbf'):=\cup_{Q\in\sbf'} \W_Q$.  By the nature of Whitney cubes,
we have $\ell(I)\approx\ell(J) \gtrsim r/M$, so $B(x,r)\cap\ospp$ contains an interior
Corkscrew point in $I^*$, and $B(x,r)\setminus\ospp$
contains an exterior Corkscrew point in $J$ (with constants
possibly depending on $M$).

\noindent{\bf Case 2}:  $r\geq M\delta(x)$.  We recall that
$\sbf'\subset\sbf$, for some regime $\sbf$ as in Lemma \ref{lemma2.1}.
Note that
\begin{equation}\label{eq3.4}
\delta(x) \approx \dist(x,\GS)\,, \qquad \forall x\in \overline{\osp}\,\,\,{\rm (hence}\,\, \forall
x\in \overline{\ospp})\,;
\end{equation}
indeed the latter  holds for $X\in\osp$, by Claim \ref{c3.1} and
the construction of $\Omega_{\sbf}$, and therefore the same
is true for $x\in \posp$.

\noindent{\bf Case 2a}: $\delta(x)>0$.  In this case, $x$ lies on a face of some $I^*$,
with $I\in \W(\sbf')$, so $I\in \W^{+}_{Q_I}$, for some $Q_I \in \sbf'$.
We then have
$$
\ell(Q_I) \approx \ell(I)\approx \delta(x) \approx \dist(I,Q_I) \lesssim r/M\ll r\,,
$$
if $M$ is large depending on $\eta$ and $K$.  Thus,
$Q_I\subset B(x,M^{-1/2} r)$.   The semi-coherency of $\sbf'$ allows us to  choose $\Qt\in \sbf'$,
with $\ell(\Qt) \approx M^{-1/4} r$,  such that $Q_I\subset\Qt$.
 Set $\bt= B(x_{\Qt}, \ell(\Qt))$, and observe that for $M$ large,
 $\bt\subset B(x,r/2)$.  Therefore, it is enough to show that
 $\bt\cap\ospp$ and $\bt\setminus\ospp$ each contains a Corkscrew point
 at the scale $\ell(\Qt)$.  To this end, we first note that since $\Qt\in\sbf'\subset\sbf$,
 \eqref{eq2.2a} implies that
 there is a point $z_{\Qt}\in\GS$ such that
 $$|x_{\Qt}-z_{\Qt}|\leq \,\eta\, \ell(\Qt)\,.$$
 Viewing $\GS$ as the graph $t=\vp_{\sbf}(y)$, so that $z_{\Qt}=:(\tilde{y},\vp_{\sbf}(\tilde{y}))$,
 we set
\begin{equation}\label{eq3.6}
Z_{\Qt}^\pm:=\big(\tilde{y},\vp_{\sbf}(\tilde{y})\pm\eta^{1/8}\ell(\Qt)\big)\,.
\end{equation}
Then by the triangle inequality
$$|Z_{\Qt}^\pm -x_{\Qt}|\lesssim \eta^{1/8}\ell(\Qt)\,.
$$
In particular, $Z^\pm_{\Qt}\in \widetilde{B}\subset  B(x,r/2)$.  Moreover, for $\eta$ small, by \eqref{eq2.2a} and the fact that the graph $\GS$
 has small Lipschitz constant,  we have
 \begin{equation}\label{eq3.7}
\delta(Z^\pm_{\Qt})\approx \dist(Z_{\Qt}^\pm,\GS)\approx \eta^{1/8}\ell(\Qt)\,.
\end{equation}
 Consequently,
 there exist $I^\pm\in\W$ such that $Z^\pm_{\Qt}\in I^\pm$, and
 \begin{equation}\label{eq3.8}
 \ell(I^\pm)\approx \dist(I^\pm,\Qt)\approx\eta^{1/8}\ell(\Qt)\,.
 \end{equation}
 Thus, $I^\pm \in\W_{\Qt}^\pm$, so
 $Z^\pm_{\Qt}\in I^\pm \subset \interior (1+\tau)I^\pm \subset  \Omega_{\sbf'}^\pm$, and therefore
 \begin{equation}\label{Z-cks}
 \dist(Z^\pm_{\Qt},\pom_{\sbf'}^\pm)
\gtrsim 
\tau\,\ell(I^\pm)
\approx
\tau\eta^{1/8}\ell(\Qt)
\approx
\tau\eta^{1/8}M^{-1/4} r\,.
 \end{equation}
Consequently, $Z_{\Qt}^+$ and $Z_{\Qt}^-$ are respectively,
 interior and exterior Corkscrew points for $\Omega_{\sbf'}$, relative to the ball
 $B(x,r)$.

 \begin{remark}\label{r3.9} We note for future reference
 that the previous construction depended only upon the
 fact that $\Qt\in\sbf'\subset\sbf$:  i.e.,  for any such $\Qt$, we may construct $Z_{\Qt}^\pm$ as in \eqref{eq3.6},
 satisfying  \eqref{eq3.7} and  \eqref{Z-cks}, and contained in some $I^\pm \in\W$
 satisfying \eqref{eq3.8}.
 \end{remark}

\noindent{\bf Case 2b}: $\delta(x)=0$.  In this case $x\in E\cap \GS$, by \eqref{eq3.4}.
Suppose for the moment that there is a cube $Q_1\in\sbf'$, with $\max\{\diam(Q_1),\ell(Q_1)\}\leq r/100$,
such that $x\in \overline{Q_1}$;
in this case we choose $\Qt\in\sbf'$, containing $Q_1$, 
with $\ell(\Qt)\approx r$, and $B(x_{\Qt},\ell(\Qt))\subset  B(x,r/2)$,
and we may then repeat the argument of Case 2a.  We therefore need only show that there is always such a $Q_1$.

Since $x\in \pospp$, there exists a sequence $\{X_m\}\subset \Omega^+_{\sbf'}$,
with 
$|X_m-x|<2^{-m}$.  For each $m$, there is
some $Q_m \in\sbf'$, with $X_m\in I_m^*$, and $I_m\in\W^+_{Q_m}$.
By construction,
$$\ell(Q_m)\approx \ell(I_m)\approx \dist(I^*_m,Q_m)\approx\dist(I_m^*,E)\leq
\dist(I^*_m,x)\leq|X_m-x|<2^{-m}\,,$$
where the implicit constants may depend upon $\eta$ and $K$.
Thus,
$$\dist(Q_m,x)\leq C_{\eta,K}2^{-m} \ll r\,,$$
for $m$ sufficiently large.  For each such $m$, we choose $Q_1^m$ with
$Q_m\subset Q_1^m\subset Q(\sbf')$ (hence $Q_1^m\in\sbf'$), and $c_0r
\leq \max\{\diam (Q_1^m),\ell(Q_1^m)\}\leq r/100$, for some fixed constant $c_0$.  Since each such
$Q_1^m\subset B(x,r)$, there are at most a bounded number of distinct
such $Q_1^m$, so at least one of these, call it $Q_1$, occurs infinitely
often as $m\to\infty$.  Hence $\dist (x,Q_1)= 0$, i.e., $x\in \overline{Q_1}$.

\noindent{\it Harnack Chain condition}.  Fix $X_1,X_2 \in\Omega^+_{\sbf'}$.
Suppose $|X_1-X_2|=:R$.  Then $R\lesssim K^{1/2} \ell(Q(\sbf'))$.
Also, there are cubes $Q_1,Q_2\in\sbf'$, and
fattened Whitney boxes $I_1^*$, $I_2^*$ (corresponding to
$I_i\in\W_{Q_i}^{+},\,i=1,2$), such that $X_i\in I_i^*\subset U^{+}_{Q_i},\, i=1,2$,
and therefore $\delta(X_i)\approx \ell(Q_i)$ (depending on $\eta$ and $K$).
We may suppose further that
$$R\leq M^{-2} \ell(Q(\sbf'))\,,$$
where $M$ is a large number to be chosen, for otherwise,
we may connect $X_1$ to $X_2$ via a Harnack path through
$X^+_{Q(\sbf')}$ (the ``center" of $U^{+}_{Q(\sbf')}$, cf. Definition
\ref{d3.2} above).

\noindent{\bf Case 1}:   max$(\delta(X_1),\delta(X_2))\geq M^{1/2} R$;  say
WLOG that $\delta(X_1)\geq M^{1/2} R$.  Then also
$\delta(X_2)\geq (1/2)  M^{1/2} R$, by the triangle inequality, since
$|X_1-X_2|=R$.  For $M$ large enough, depending on $\eta$ and $K$,
we then have that min$(\ell(I_1),\ell(I_2))\geq M^{1/4} R$.  Note that
$\dist(I^*_1,I^*_2)\leq R$.
By the Whitney construction, for sufficiently small choice of
 the fattening parameter $\tau$,
if $\dist(I^*_1,I^*_2)\ll\min(\ell(I_1),\ell(I_2))$,
then the fattened cubes
$I_1^*$ and $I^*_2$ overlap.  In the present case, the latter scenario holds
if $M$ is chosen large enough, and we may then clearly form a Harnack Chain connecting
$X_1$ to $X_2$.

\noindent{\bf Case 2}:  max$(\delta(X_1),\delta(X_2))< M^{1/2} R$.
Then, since
$$\ell(Q_i)\approx\ell(I_i)
\approx \delta(X_i)\approx\dist(I_i,Q_i)$$
(depending on $\eta$ and $K$),  we have that
$\dist(Q_1,Q_2)\leq M^{3/4} R$,
for $M$ large enough.  We now choose $\Qt_i\in\sbf'$, with $Q_i \subset\Qt_i$,
such that $\ell(\Qt_1)=\ell(\Qt_2)\approx MR$ .  Then
 \begin{equation}\label{eq3.10}
\dist(\Qt_1,\Qt_2)\leq M^{3/4} R \approx M^{-1/4} \ell(\Qt_i)\,,\quad i=1,2\,.
\end{equation}
For $M$ large enough, it then follows that $U^{+}_{\Qt_1}$ meets
$U^{+}_{\Qt_2}$, by construction.  Indeed, let $Z^+_{\Qt_1}$ denote the
 point defined in \eqref{eq3.6}, relative to the cube $\Qt_1\in\sbf'\subset\sbf$.  Then $Z_{\Qt_1}^+$
belongs to some $I\in \W$, with
$$\ell(I) \approx \dist(I,\Qt_1) \approx\eta^{1/8} \ell(\Qt_1)$$
(cf. Remark \ref{r3.9}).     This clearly implies that $I\in \W_{\Qt_1}$.
On the other hand by \eqref{eq3.10} and since $\ell(\Qt_1)\approx \ell(\Qt_2)$,
for $M$ large enough we have
$$
\dist(I,\Qt_2)
\lesssim
\dist(I,\Qt_1)+\ell(\Qt_1)+\dist(\Qt_1,\Qt_2)
\lesssim \ell(\Qt_2)
\le
\sqrt{K}\,\ell(\Qt_2),
$$
and therefore $I\in \W_{\Qt_2}^0\subset \W_{\Qt_2}$.
Consequently, $I\in \W_{\Qt_1}\cap \W_{\Qt_2}$, so
$I^*\subset U^{+}_{\Qt_1}\cap U^{+}_{\Qt_2}$.   
We may therefore form a Harnack Chain from $X_1$ to $X_2$ by passing through $Z_{\Qt_1}^+$. 
\end{proof}

\section{Carleson measure estimate for bounded harmonic functions:  proof of Theorem \ref{t1}.}

In this section we give the proof of Theorem \ref{t1}.
We will use the method of ``extrapolation of Carleson measures", a bootstrapping procedure for
lifting the Carleson measure constant, developed by J. L. Lewis \cite{LM}, and based on
the Corona construction of Carleson \cite{Car} and Carleson and Garnett \cite{CG}
(see also \cite{HL}, \cite{AHLT}, \cite{AHMTT}, \cite{HM-TAMS}, \cite{HM-I}).

Let $E\subset \ree$ be a UR set of co-dimension 1.    We fix  positive numbers $\eta\ll 1$,
and $K\gg 1$, and for these values of $\eta$ and $K$,
we perform the bilateral Corona decomposition of $\dd(E)$ guaranteed by Lemma \ref{lemma2.1}.
Let $\M:=\{Q(\sbf)\}_{\sbf}$ denotes the collection of cubes which are the maximal elements of the stopping time regimes in $\G$.
Given a cube $Q\in \dd(E)$, we set
\begin{equation}\label{eq4.0}
\alpha_Q:= 
\begin{cases} \sigma(Q)\,,&{\rm if}\,\, Q\in \M\cup\B, \\
0\,,& {\rm otherwise}.\end{cases}
\end{equation}
Given  any collection $\dd'\subset\dd(E)$, we define
\begin{equation}\label{eq4.1}
\mut(\dd'):= \sum_{Q\in\dd'}\alpha_{Q}.
\end{equation}
We recall that $\dd_Q$ is the ``discrete Carleson region
relative to $Q$",
defined in \eqref{eq3.4a}. 
Then by Lemma \ref{lemma2.1} (2), we have the discrete Carleson measure estimate
\begin{multline}\label{eq4.3x}
\mut(\dd_Q):=\sum_{Q'\subset Q, \,Q'\in\B} \sigma(Q')
\,\,+\,\sum_{\sbf: Q(\sbf)\subset Q}\sigma\big(Q(\sbf)\big)\,\leq\, C_{\eta,K}\, \sigma(Q)\,,
\\[4pt] \forall Q\in \dd(E)\,.
\end{multline}
Given a family $\F:=\{Q_j\}\subset \dd(E)$ of
pairwise disjoint cubes, we recall that the ``discrete  sawtooth" $\dd_\F$ is 
the collection of all cubes in $\dd(E)$ that are not contained in any $Q_j\in\F$ (cf. \eqref{eq2.discretesawtooth1}),
and we define the ``restriction of $\mut$ to the sawtooth $\dd_\F$'' by
\begin{equation}\label{eq4.4x}
\mut_\F(\dd'):=\mut(\dd'\cap\dd_\F)= \sum_{Q\in\dd'\setminus (\cup_{\F} \,\dd_{Q_j})}\alpha_{Q}.
\end{equation}

We shall use the method of ``extrapolation of Carleson measures'' in the following form.

\begin{lemma}\label{lemma:extrapol}
Let $\sigma$ 
be a non-negative,
dyadically doubling Borel measure on $E$,
and let $\mut$ be a discrete Carleson measure with respect to $\sigma$, 
i.e., there exist non-negative coefficients $\alpha_Q$  so that $\mut$ is defined as in \eqref{eq4.1},
and a constant $M_0<\infty$, with
\begin{equation}\label{eq4.7a}
\|\mut\|_{\C}:= \sup_{Q\in\dd(E)}\frac{\mut(\dd_Q)}{\sigma(Q)}\,\le\, M_0\,.
\end{equation}
Let $\mutt$ be another non-negative measure on $\dd(E)$ as in \eqref{eq4.1}, say
\begin{equation}\label{eq4.8a}
\mutt(\dd'):=\sum_{Q\in\dd'}\beta_Q\,,\qquad \beta_Q\geq 0\,,\quad \forall\, \dd' \subset \dd(E),
\end{equation}
where for some uniform constant $M_1$, and for each cube $Q$,
\begin{equation}\label{eq4.8aa}
\beta_Q\,\leq\, M_1 \,\sigma(Q)\,.
\end{equation}
Suppose that there is a positive constant $\gamma$ 
such that for every $Q\in \dd(E)$ and every family of pairwise disjoint
dyadic subcubes  $\F=\{Q_j\}\subset \dd_{Q}$
verifying
\begin{equation}\label{eq4.9a}
\|\mut_{\mathcal{F}}\|_{\mathcal{C}(Q)}:= \sup_{Q'\in\dd_Q}\frac{\mut\left(\dd_{Q'}
\setminus (\cup_{\F} \,\dd_{Q_j})\right)}{\sigma(Q')}
\le\gamma\, ,
\end{equation}
we have that $\mutt_\F$ (defined as in \eqref{eq4.4x}, but with coefficients $\beta_Q$) satisfies
\begin{equation}\label{eq4.7}
\mutt_\F(\dd_Q)\,\leq\, M_1\,\sigma(Q)\,.
\end{equation}
Then $\mutt$ is a discrete Carleson measure, with
\begin{equation}\label{eq4.8}
\|\mutt\|_{\C}:= \sup_{Q\in\dd(E)}\frac{\mutt(\dd_Q)}{\sigma(Q)}\le M_2\,,
\end{equation}
for some $M_2<\infty $ depending on $n,M_0,M_1,\gamma$ and the doubling constant of $\sigma$.
\end{lemma}

Let us momentarily take the lemma for granted, and use it to prove Theorem \ref{t1}.
We begin with a preliminary reduction, which reduces matters to working with balls of radius
$r<C \diam(E)$; i.e., we claim that the desired estimate \eqref{eq1.carl*} is equivalent to
\begin{equation}\label{eq1.carl2}
\sup_{y\in E,\, 0<r<  100 \diam(E)} \,\frac1{r^n}\iint_{B(y,r)} |\nabla u(X)|^2 \delta(X) \,dX\,
\leq \,C\, \|u\|^2_\infty\, .
\end{equation}
Of course, if $E$ is unbounded the equivalence is obvious.
Thus, we suppose that $\diam(E)<\infty$, and
that \eqref{eq1.carl2} holds.   Let $u$ be bounded and harmonic  in
$\ree\setminus E$.  We may assume that  $\|u\|_\infty = 1$.
Fix  a ball $B(y,r)$, with $y\in E$, and $r\geq 100\diam(E)$.
Set $r_0:= 10 \diam(E)$.
By \eqref{eq1.carl2},
$$\iint_{B(y,r_0)} |\nabla u(X)|^2 \delta(X) \,dX\,
\leq \,C r_0^n  \leq C r^n\,.$$
Moreover,
\begin{multline*}
\iint_{B(y,r)\setminus B(y,r_0)} |\nabla u(X)|^2 \delta(X) \,dX\\[4pt]
\leq \,\sum_{0\leq k\leq\log_2(r/r_0)} \iint_{2^kr_0\leq|X-y|<2^{k+1}r_0} |\nabla u(X)|^2 \delta(X) \,dX\\[4pt]
 \lesssim \,\sum_{0\leq k\leq\log_2(r/r_0)} \left(2^k r_0\right)^n \,\lesssim \,r^n\,,
\end{multline*}
where in the second inequality we have used Caccioppoli's inequality, the normalization
$\|u\|_\infty=1$, and the fact that
$\delta(X) \approx |X-y|$ in the regime $|X-y|\geq 10 \diam(E)$.
Thus, \eqref{eq1.carl2} implies (and hence is equivalent to) \eqref{eq1.carl*}, as claimed.

We shall apply Lemma \ref{lemma:extrapol} with, as usual, $\sigma :=H^n\big|_E$,
and with $\mut$ as above, with coefficients $\alpha_Q$ defined as in \eqref{eq4.0},
so that \eqref{eq4.7a} holds with $M_0=C_{\eta,K}$, by Lemma \ref{lemma2.1} (2).
For us, $\mutt$ will be a discretized version of the measure
$|\nabla u(X)|^2\delta(X) dx$, where $u$ is bounded and harmonic in
$\Omega:=\ree\setminus E$.  We now claim that
\eqref{eq1.carl*} is equivalent to
the analogous bound
\begin{equation}\label{eq1.carl}
\sup_{Q\in \dd(E)} \,\frac1{\sigma(Q)}\iint_{T_Q} |\nabla u(X)|^2 \delta(X) \,dX\,
\leq \,C\, \|u\|^2_\infty\,.
\end{equation}
That \eqref{eq1.carl*} implies \eqref{eq1.carl} is obvious by \eqref{eq3.3aab}. 
The converse implication reduces to showing that
\eqref{eq1.carl} implies \eqref{eq1.carl2}, since,
as noted above, the latter estimate is equivalent to \eqref{eq1.carl*}.
We proceed as follows.   Fix  a ball $B(x,r)$, with $x\in E$, and $r< 100\diam(E)$.
We choose a collection
of dyadic cubes $\{Q_k\}_{k=1}^N$, with $\ell(Q_k)\approx Mr$
(unless $r>\diam(E)/M$, in which case our collection is comprised of only one cube, namely
$Q_1=E$), where
$M$ is a large fixed number to be chosen, such that
$$B(x,10r)\cap E\subset \bigcup_k Q_k\,.$$
Note that the cardinality $N$ of this collection may be taken to be uniformly bounded.
We claim that $\cup_kT_{Q_k}$ covers $B(x,r)\setminus E$, in which case it follows
immediately that
\eqref{eq1.carl} implies \eqref{eq1.carl2}.
Let us now prove the claim.
Given $Y\in B(x,r)\setminus E$,  there is
a Whitney box $I\in\W$ containing $Y$, so that
$$\ell(I)\approx \delta(Y)\leq |x-Y|<r\,.$$  Let $\hat{y}\in E$ satisfy
$|Y-\hat{y}|=\delta(Y)$, and choose $Q\in\dd(E)$ containing $\hat{y}$ so that
$\ell(Q)=\ell(I)$ (unless $\diam(I) \approx \diam(E)$, in which case we just set $Q=E$).
Note also that $\dist(I,Q)\approx \ell(Q)$ with harmless constants, so that $I\in \W_Q^0\subset \W_Q$.
Thus, $Y\in U_Q$ (cf. \eqref{eq3.3aa}).  Moreover,
by the triangle inequality, $\hat{y}\in B(x,2r)\cap E$, whence it follows
(for $M$ chosen large enough) that
$Q$ is contained in one of the cubes $Q_k$ chosen above,
call it $Q_{k_0}$.
Consequently, $Y\in T_{Q_{k_0}}$ (cf. \eqref{eq3.3a}).  This proves the claim.
Therefore, it is enough to prove \eqref{eq1.carl}.

To the latter end, we discretize \eqref{eq1.carl} as follows.
By normalizing, we may assume without loss of generality that $\|u\|_\infty = 1$.
We fix a small $\tau\in(0,\tau_0/2)$, 
and set $U_Q:= U_{Q,\tau}$, $T_Q:=T_{Q,\tau}$ as in
\eqref{eq3.3aa} and \eqref{eq3.3a}.
We now set
\begin{equation}\label{eq4.15a}
\beta_Q := \dint_{U_Q}|\nabla u(X)|^2 \,\delta(X) \,dX \,,\end{equation}
and define $\mutt$ as in \eqref{eq4.8a}.  We note that \eqref{eq4.8aa} holds
by Caccioppoli's inequality (applied in each of the fattened Whitney boxes comprising $U_Q$),  and
the definition of $U_Q$ and the ADR property of $E$.
Moreover, the Whitney
regions $U_Q$ have the bounded overlap property:
\begin{equation}\label{eq4.16}\sum_{Q\in\dd} 1_{U_Q}(X) \leq C_{n, \,ADR}\,.\end{equation}
Consequently, for every
pairwise disjoint family $\F\subset\dd(E)$, and every $Q\in\dd_\F$, we have
\begin{equation}\label{eq4.15}
\mutt_\F(\dd_Q)
\approx \dint_{\Omega_{\F,Q}}|\nabla u(X)|^2 \,\delta(X) \,dX
\end{equation}
where we recall that  (see \eqref{eq2.sawtooth1}) 
$$\Omega_{\F,Q}:= {\rm int}\left(\bigcup_{Q'\in\dd_Q\cap\dd_\F} U_{Q'}\right)\,.$$
In particular, taking $\F=\emptyset$, in which case $\dd_{\F}=\dd(E)$,
and thus $\Omega_{\F,Q}=T_{Q}$, we obtain that \eqref{eq1.carl} holds
if and only if $\mutt$ satisfies the  discrete Carleson measure estimate
\eqref{eq4.8}.  

For each
$Q\in\G$, we set $\widehat{U}_Q^{\pm} := U_{Q,2\tau}^{\pm}$, as in \eqref{eq3.3b},
and for each stopping time regime $\sbf\subset\G$,
we define the corresponding NTA subdomains $\Omega_\sbf^\pm=\Omega_\sbf^\pm(2\tau)$ as in
\eqref{eq3.2} (with $\sbf'=\sbf$).  
Let $\mut$ be the discrete Carleson measure defined in \eqref{eq4.0}-\eqref{eq4.1}.
Our goal is to verify the hypotheses of Lemma \ref{lemma:extrapol}.  We have already
observed that \eqref{eq4.8aa} holds, therefore, we
need to show that \eqref{eq4.9a} implies \eqref{eq4.7}, or more precisely,
that given a cube $Q\in\dd(E)$ and a pairwise disjoint family $\F\subset \dd_Q$, 
for which \eqref{eq4.9a} holds with suitably small $\gamma$, we may deduce
\eqref{eq4.7}.

Let us therefore suppose that \eqref{eq4.9a} holds for some $\F$, and some $Q$, and
we disregard the trivial case $\F=\{Q\}$.
By definition of $\mut$, and of $\mut_\F$ (cf. \eqref{eq4.0}-\eqref{eq4.4x}),
if $\gamma$ is sufficiently small,
then $\dd_Q\cap \dd_\F$ does not contain any $Q'\in \M\cup\B$
(recall that $\M:=\{Q(\sbf)\}_{\sbf}$ is the collection of the maximal cubes of the various stopping time regimes).  Thus,
{\it every} $Q'\in \dd_Q\cap\dd_\F$ belongs to $\G$, and moreover,
all such $Q'$ belong to the {\it same} stopping time regime  $\sbf$, since $Q\in \dd_Q \cap \dd_\F$ unless $\F=Q$, the case that we excluded above.
Consequently, $\Omega_{\F,Q}$, and more precisely,
each $U_{Q'}$, with $Q'\in \dd_Q\cap\dd_\F$, splits into two pieces, call them
$\Omega_{\F,Q}^\pm$, and $U_{Q'}^\pm$, contained in $\Omega_\sbf^\pm.$
For $Q'\in \dd_Q\cap\dd_\F$,
we make the corresponding splitting of $\beta_{Q'}$ 
into $\beta_{Q'}^\pm$ 
so that
\begin{equation}\label{eq4.18}
\beta^\pm_{Q'} := \dint_{U^\pm_{Q'}}|\nabla u(X)|^2 \,\delta(X) \,dX \,,\end{equation}
and for  $\dd'\subset \dd_Q$, we set
$$\mutt_\F^\pm(\dd'):= \sum_{Q'\in\dd'\cap\dd_\F}\beta_{Q'}^\pm\,.$$
For the sake of specificity, we shall consider $\Omega_{\F,Q}^+$, and
observe that $\Omega_{\F,Q}^-$
may be treated by exactly the same arguments.

Since we have constructed $U_Q$ with parameter $\tau$, and $\widehat{U}_Q^{\pm}$ with parameter
$2\tau$, for $X\in \Omega_{\F,Q}^+$, we have that
$$\delta(X)\approx \delta_*(X)\,,$$
where $ \delta_*(X):=\dist (X,\pom_\sbf^+)$, and where the
implicit constants depend on $\tau$.
Consequently (cf. \eqref{eq4.15}),
\begin{equation*}
\mutt^+_\F(\dd_Q)
\approx \dint_{\Omega^+_{\F,Q}}|\nabla u(X)|^2 \, \delta_*(X) \,dX\,.
\end{equation*}
As above, set $B^{*}_Q:= B(x_Q, K\ell(Q))$.  
Note that
$\Omega^+_{\F,Q}\subset B^{*}_Q\cap\Omega_\sbf^+$, by construction of $\Omega^+_{\F,Q}$  and \eqref{eq3.3aab}. 
Thus, one can find a ball $B^{**}_Q$ centered at $\partial\Omega^+_\sbf$ and with radius of the order of $\ell(Q)$, such that  
\begin{equation}\label{eqCME-NTA}
\mutt^+_\F(\dd_Q)
\leq\dint_{B^{**}_Q\cap\Omega^+_\sbf}|\nabla u(X)|^2 \, \delta_*(X) \,dX\,\lesssim 
\, \sigma(Q)\,,
\end{equation}
where in the last step we have used that $\Omega_\sbf^+$ is an NTA domain with ADR boundary,
and is therefore known to satisfy such Carleson measure estimates
(recall that we have normalized so that $\|u\|_\infty =1$).    Indeed, by \cite{DJe},
for any NTA domain with ADR boundary,
harmonic measure belongs to $A_\infty$ with respect to surface measure $\sigma$ on the boundary,
and therefore one obtains Carleson measure estimates for bounded solutions by
\cite{DJK}.  Since a similar bound holds for $\mutt^-_\F(\dd_Q)$, we obtain \eqref{eq4.7}.
Invoking Lemma \ref{lemma:extrapol}, we obtain \eqref{eq4.8}, and thus equivalently,
as noted above, \eqref{eq1.carl}.

It remains to prove Lemma \ref{lemma:extrapol}.
To this end, we shall require the following result from
\cite{HM-I}.  

\begin{lemma}[{\cite[Lemma 7.2]{HM-I}}]\label{lemma:Corona}
Suppose that $E$ is ADR.  Fix $Q\in \dd(E)$
and $\mut$ as above.  Let $a\geq 0$ and $b>0$, and suppose that
$\mut(\dd_{Q})\leq (a+b)\,\sigma(Q).$
Then there is a family $\F=\{Q_j\}\subset\dd_{Q}$
of pairwise disjoint cubes, and a constant $C$ depending only on dimension
and the ADR constant such that
\begin{equation} \label{Corona-sawtooth}
\|\mut_\F\|_{\C(Q)}
\leq C b,
\end{equation}
\begin{equation}
\label{Corona-bad-cubes}
\sigma(B)
\leq \frac{a+b}{a+2b}\, \sigma(Q)\,,
\end{equation}
where $B$ is the union of those $Q_j\in\F$ such that
$\mut\big(\dd_{Q_j}\setminus \{Q_j\}\big)>a\,\sigma(Q_j)$.
\end{lemma}

We refer the reader to \cite[Lemma 7.2]{HM-I} for the proof.   We remark that the lemma is stated
in \cite{HM-I} with $E=\pom$, the boundary of a connected domain, but the proof actually requires only that
$E$ have a dyadic cube structure, and that $\sigma$ be a
non-negative, dyadically doubling Borel measure on $E$.

\begin{proof}[Proof of Lemma \ref{lemma:extrapol}]
The proof proceeds by induction, following \cite{LM},
\cite{AHLT}, \cite{AHMTT}, \cite{HM-I}.
The induction hypothesis, which we formulate for any $a\geq 0$, is as follows:
\\[.3cm]
\null\hskip.1cm \fbox{\rule[8pt]{0pt}{0pt}$H(a)$}\hskip8pt \fbox{\ \parbox[c]{.85\textwidth}{%
%
\rule[10pt]{0pt}{0pt}\it There exist $\eta_a\in(0,1)$ and $C_a<\infty$ such that, for every $Q\in \dd(E)$ satisfying  $\mut(\dd_Q)\le a\,\sigma(Q)$, there is a pairwise disjoint
family $\{P_k\}\subset \dd_Q$, with
\begin{equation}\label{eq4.19}
\sigma \big(Q\setminus (\cup_kP_k)\big)\,\geq\,\eta_a\,\sigma(Q)\,,\qquad\qquad\null
\end{equation}
such that
\begin{equation}\label{eq4.20}
\mutt\left(\dd_Q\setminus (\cup_k \dd_{P_k})\right)\, \leq \,C_a\, \sigma(Q)\,.\qquad\qquad\null
\end{equation}}\ }

\

It suffices to show that $H(a)$ holds with $a=M_0$.  Indeed, once this is done, then
invoking \eqref{eq4.7a},
we will obtain that  there are constants $\eta_a = \eta(M_0)$
and $C_a=C(M_0)$, such that for every $Q\in\dd(E)$,
there is a family $\{P_k\}\subset \dd_Q$
as above for which \eqref{eq4.19} and \eqref{eq4.20} hold.
We may then invoke a standard John-Nirenberg lemma for Carleson measures (whose proof iterates these estimates and sums a geometric series)
to conclude that  \eqref{eq4.8} holds, as desired.

In turn, to obtain $H(M_0)$, we proceed in two steps.

\noindent{\bf Step 1}:  establish $H(0)$.

\noindent{\bf Step 2}: show that there is a constant $b>0$, depending only upon
the specified parameters in the hypotheses of Lemma \ref{lemma:extrapol},
such that $H(a)$ implies $H(a+b)$.

Once steps 1 and 2 have been accomplished, we then obtain $H(M_0)$ by iterating
Step 2 roughly $M_0/b$ times.

\medskip

\noindent{\bf Proof of Step 1: $H(0)$ holds.}
If $\mut(\dd_Q)=0$ then \eqref{eq4.9a} holds, with $\F=\emptyset$, and for $\gamma$ as small as
we like.
Thus, by hypothesis, we have that  
\eqref{eq4.7} holds, with $\mutt_\F=\mutt$ (since in this case $\F$ is vacuous).
Hence, \eqref{eq4.19}-\eqref{eq4.20} hold, with $\{P_k\}=\emptyset,\, \eta_0=1/2$,
and $C_0 = M_1$.

\medskip

\noindent{\bf Proof of Step 2: $H(a) \implies H(a+b)$}  Suppose that $a\geq 0$ and that $H(a)$ holds.  We set $b:=\gamma/C$, where
$\gamma$ is specified in \eqref{eq4.9a}, and $C$ is the constant in \eqref{Corona-sawtooth}.
Fix a cube $Q$ such that $\mut(\dd_Q)\leq (a+b)\,\sigma(Q)$.  We then apply Lemma
\ref{lemma:Corona} to construct a family $\F$ with the stated properties.
In particular, by our choice of $b$, \eqref{Corona-sawtooth} becomes \eqref{eq4.9a}.

We may suppose that
$a<M_0$, otherwise we are done.  Thus
$$\frac{a+b}{a+2b} \leq \frac{M_0+b}{M_0+2b}=:\theta<1\,.$$
Define $\eta:= 1-\theta$.
We set $A:= Q\setminus (\cup_\F Q_j)$, and let $G:= (\cup_\F Q_j)\setminus B$.
Then, \eqref{Corona-bad-cubes} gives 
\begin{equation} \label{eq4.21}
\sigma (A\cup G)\geq\eta\,\sigma(Q)\,.
\end{equation}

We consider two cases.

\noindent{\bf Case 1}:  $\sigma (A)\geq (\eta/2)\,\sigma (Q).$
In this case, we take $\{P_k\} := \F$, so that \eqref{eq4.19} holds with
$\eta_{a+b} =\eta/2$.  Moreover, since  \eqref{eq4.9a} holds by our choice of $b$,
we obtain by hypothesis that \eqref{eq4.7} holds.   The latter is equivalent to \eqref{eq4.20},
since $\F=\{P_k\}$, with $C_{a+b} = M_1$.    Thus, $H(a+b)$ holds in Case 1.

\noindent{\bf Case 2}:  $\sigma (A)< (\eta/2)\,\sigma (Q).$  In this case, by \eqref{eq4.21},
we have that
\begin{equation}\label{eq4.22}
\sigma(G) \geq (\eta/2)\sigma(Q)\,.
\end{equation}
By definition, $G$ is the union of cubes in the subcollection $\F_{\rm good}\subset \F$,
defined by
$$\F_{\rm good}:=\left\{Q_j\in\F: \,
\mut(\dd_{Q_j}\setminus \{Q_j\})\leq a\,\sigma(Q_j)\right\}\,.$$
For future reference, we set $\F_{\rm bad}:=\F\setminus \F_{\rm good}$.
We note that by pigeon-holing, each $Q_j\in\F_{\rm good}$ has at least one dyadic child,
call it $Q_j'$, such that
$$\mut(\dd_{Q'_j})\leq a\,\sigma(Q'_j)$$
(if there is more than one such child, we simply pick one).  Thus, we may invoke the induction hypothesis
$H(a)$, to obtain that for each such $Q_j'$,
there exists a pairwise disjoint
family $\{P^j_k\}\subset \dd_{Q_j'}$, with
\begin{equation}\label{eq4.23}
\sigma \big(Q_j'\setminus (\cup_kP^j_k)\big)\,\geq\,\eta_a\,\sigma(Q_j')\,
\gtrsim \eta_a\,\sigma(Q_j)\,
\end{equation}
(where in the last step we have used that
$\sigma$ is dyadically doubling),
such that
\begin{equation}\label{eq4.24}
\mutt\left(\dd_{Q_j'}\setminus (\cup_k \dd_{P^j_k})\right)\, \leq \,C_a\, \sigma(Q_j')\,.
\end{equation}

Given $Q_j\in\F_{\rm good}$, we define $\F_j'$ to be the collection of
all the dyadic brothers of $Q_j'$; i.e., $\F_j'$ is comprised of
all the dyadic children of
$Q_j$, except $Q_j'$.  We then define a collection $\{P_k\} \subset \dd_Q$ by
$$\{P_k\}:= \F_{\rm bad} \,\cup\, \left(\cup_{Q_j\in\F_{\rm good}}\F_j'\right)
\,\cup \,\left(\cup_{Q_j\in \F_{\rm good}} \{P_k^j\}\right)\,.$$
We note that \eqref{eq4.19} holds for this collection $\{P_k\}$, with
$\eta_{a+b}\gtrsim \eta_a \,\eta/2$,
by \eqref{eq4.22} and \eqref{eq4.23}:
\begin{align*}
\sigma\big(\cup_k P_k\big)
&=
\sigma(B)
+
\sum_{Q_j\in \F_{\rm good}} \sigma(Q_j\setminus Q_j')
+
\sum_{Q_j\in \F_{\rm good}} \sigma\big(\cup_k P_k^j\big)
\\
&=
\sigma(B)+\sigma(G)-\sum_{Q_j\in \F_{\rm good}} \sigma\big(Q_j'\setminus\cup_k P_k^j\big)
\\
&\le
\sigma(Q)-c\,\eta_a \,\sigma(G)
\\
&\le
\sigma(Q)-c\,\eta_a\,\frac{\eta}{2} \,\sigma(Q)
\end{align*}
It remains only to verify \eqref{eq4.20}.
To this end, we write
\begin{multline*}
\mutt\left(\dd_Q\setminus(\cup_k\dd_{P_k})\right) \\ =\,
\mutt\left(\dd_Q\setminus(\cup_\F\dd_{Q_j}\right) \,+\,
\sum_{Q_j\in\F_{\rm good}} \left(\mutt(\{Q_j\})
+\mutt\left(\dd_{Q_j'}\setminus(\cup_k\dd_{P_k^j}\right)\right)\\
= \mutt_\F(\dd_Q) + \sum_{Q_j\in\F_{\rm good}} \left(\beta_{Q_j}
+\mutt\left(\dd_{Q_j'}\setminus(\cup_k\dd_{P_k^j}\right)\right)
\\ \lesssim \,\sigma(Q) \,+ \sum_{Q_j\in\F_{\rm good}}  \sigma (Q_j) \,\lesssim\, \sigma(Q)\,,
\end{multline*}
where in third line we have used the definitions of $\mutt_\F$ (cf. \eqref{eq4.4x})
and of $\mutt$, and in the last line we have used \eqref{eq4.7}, \eqref{eq4.8aa},
and \eqref{eq4.24}, along with the pairwise disjointness of the cubes in $\F$.
\end{proof}

\begin{remark}\label{remark4.28} We note that, in fact, the proof of Theorem \ref{t1} did not require
harmonicity of $u$ {\it per se}.  Indeed, a careful examination of the preceding argument reveals that
we have only used the following three properties of $u$: 1) $u\in L^\infty(\Omega)$;
2) $u$ satisfies Caccioppoli's inequality in $\Omega$;  3) $u$ satisfies Carleson measure
estimates in every NTA sub-domain of $\Omega$ with ADR boundary.  
\end{remark}

\section{$\eps$-approximability:  proof of Theorem \ref{t2}}\label{s-eps-approx}
In this section we give the proof of Theorem \ref{t2}.
Our approach here combines the technology of the present paper
(in particular, the bilateral Corona decomposition
of Lemma \ref{lemma2.1}), 
with the original argument of \cite{G}, and its extensions
in \cite{D} and \cite{HKMP}).  Moreover, we shall invoke Theorem \ref{t1} at certain points in the argument.

The first (and main) step in our proof will be to establish a dyadic version, i.e., given
$u$ harmonic and bounded in $\om:=\ree\setminus E$, with $\|u\|_\infty\leq 1$, and given
$\eps\in (0,1)$ and $Q\in\dd(E)$,
we shall construct $\vp:=\vp^\eps_Q$, defined on the ``Carleson tent" $T_Q$, such that
$\|u-\vp\|_{L^\infty(T_Q)}<\eps$, and
\begin{equation}\label{eq6.carldyadic}
\sup_{Q'\subset Q} \frac1{|Q'|}\iint_{T_{Q'}}|\nabla \vp|\,\lesssim\eps^{-2}\,.
\end{equation}
Once we have established \eqref{eq6.carldyadic}, it will then be relatively easy
to construct $\vp$, globally defined on $\om$, and
satisfying properties \eqref{eq1.4} and \eqref{eq1.5} of Definition \ref{def1.3}.

We begin by refining the bilateral Corona decomposition of Lemma \ref{lemma2.1}.
We fix $\eta \ll1$ and $K\gg1$,  and 
we make the constructions of Lemma \ref{lemma2.1}, corresponding to this choice of
$\eta$ and $K$.  We also fix $\eps\in (0,1)$, and a parameter $\tau\in (0,\tau_0/10)$.
For each $Q\in \dd(E)$, we form the Whitney regions
$U_Q=U_{Q,\tau}$ as above, and
we split each
$U_Q$ into its various connected components
$U_Q^i$.  

Let $u$ be a bounded harmonic function in $\Omega=\ree\setminus E$, with
$\|u\|_{L^\infty(\om)}\leq 1$.   We say that $U_Q^i$ is a ``red component" if
\begin{equation}\label{eq7.1}
\osc_{U_Q^{i}}  u: = \max_{Y\in U_Q^{i}} u(Y) - \min_{Y\in U_Q^{i}} u(Y) > \frac{\eps}{10}\,,
\end{equation}
otherwise we say that $U_Q^i$ is a ``blue component".
We also say that $Q\in \dd(E)$ is a ``red cube" if its associated Whitney region $U_Q$
has at least one red component,
otherwise, if $\osc_{U_Q^i} u \leq \eps/10$ for every  connected component
$U_Q^i,\, 1\leq i\leq N$, then we say that $Q$ is a ``blue cube".
\begin{remark}\label{r6.2}
The number $N=N(Q)$
of components $U_Q^i$ is uniformly bounded, depending only on $\eta, K$ and dimension,
since each component $U_Q^i$
contains a fattened Whitney box $I^*$ with $\ell(I) \approx \ell(Q)$, and since all such $I^*$ satisfy
$\dist(I^*,Q) \lesssim \ell(Q)$.  Of course, as noted above (cf. \eqref{eq3.3b}), if $Q\in\G$, then
$U_Q$ has precisely two components $U^\pm_Q$.
\end{remark}

We now refine the stopping time regimes as follows. Given
$\sbf\subset \G$ as constructed in
Lemma \ref{lemma2.1}, set $Q^0:= Q(\sbf)$, and let  $G_0=G_0(\sbf):=\{Q^0\}$ be the
 ``zeroeth generation".
We subdivide $Q^0$
dyadically, and stop the first time that
we reach a cube $Q\subset Q^0$ for which at least one of the
following holds:
\begin{enumerate}
\item $Q$ is not in $\sbf$.
\smallskip
\item $|u(Y_Q^+) - u(Y^+_{Q^0})| > \eps/10$.
\smallskip
\item $|u(Y_Q^-) - u(Y^-_{Q^0})| > \eps/10$.
\end{enumerate}
(where we recall that $Y_Q^\pm$ is the ``modified center" of the Whitney region $U^\pm_Q$;
see Definition \ref{d3.2} and Remark \ref{r3.11a}).

Let $\F_1=\F_1(Q^0)$ denote the maximal sub-cubes of $Q^0$ extracted by this stopping time procedure, and
note that the collection of all  $Q\subset Q^0$ that are not contained in any $Q_j\in \F_1$,
forms a semi-coherent (cf. Definition \ref{d3.11}) subregime of $\sbf$, call it $\sbf'=\sbf'(Q^0)$,
with maximal element
$Q(\sbf'):= Q^0$.  Clearly, the maximality of the cubes in $\F_1$ implies that every $Q\in\sbf'$
belongs to $\sbf$,
and moreover
\begin{equation}\label{eq7.2}
\max\left(|u(Y_{Q}^+) - u(Y^+_{Q^0})|,|u(Y_{Q}^-) - u(Y^-_{Q^0})|\right) \leq \eps/10\,,\quad \forall Q\in\sbf'\,.
\end{equation}
Let $G_1=G_1(Q^0):= \F_1\cap\sbf$ denote the first generation cubes.
We observe that $G_1$ may be empty, since $\F_1$ may not contain any cubes belonging to $\sbf$.
In this case, we simply have $\sbf'(Q^0) =\sbf$.
On the other hand, if $G_1$ is non-empty, then for each $Q^1\in G_1(Q^0)$,
we repeat the stopping time construction above (with $Q^1$ in place of $Q^0$), except that in
criteria (2) and (3) we replace $Y_{Q^0}^\pm$ by $Y_{Q^1}^\pm$ (criterion (1) is unchanged, so
we continue to work only with cubes belonging to $\sbf$).
For each $Q^1\in G_1(Q^0)$, we may then define first generation cubes $G_1(Q^1)$ in the same way,
and thus, we may define recursively
$$
G_2(Q^0):= \bigcup_{Q^1\in G_1(Q^0)} G_1(Q^1)\,,
$$
and in general (modifying the stopping time criteria (2) and (3) {\it mutatis mutandi})
$$G_{k+1}(Q^0):= \bigcup_{Q^{k}\in G_{k}(Q^0)}G_1(Q^{k})\,,\quad k\geq 0\,,$$
where the case $k=0$ is a tautology, since $G_0(Q^0) :=\{Q^0\}$, and where the set of indices $\{k\}_{k\geq 0}$ may be finite or
infinite.
In addition, bearing in mind that $Q^0=Q(\sbf)$, we shall sometimes find it convenient
to emphasize the dependence on $\sbf$, so with slight abuse of notation we write
$$G_k(\sbf):=  G_k(Q^0) = G_k(Q(\sbf))\,,\quad k\geq 0\,.$$
We also set
$$G(\sbf):= \bigcup_{k\geq 0} G_k(\sbf)\,, \qquad G^*:= \bigcup_{\sbf} G(\sbf)\,,$$
to denote, respectively, the set of generation cubes in $\sbf$, and the collection of
all generation cubes.

\begin{remark}\label{r7.3}  We record some observations concerning the ``generation cubes":
Given $\sbf$ as in Lemma \ref{lemma2.1}, our construction produces a decomposition
of $\sbf$ into disjoint subcollections
$$\sbf = 
 \bigcup_{Q \in G(\sbf)}\sbf'\left(Q\right)\,,$$
where 
each $\sbf'(Q)$ is a semi-coherent
subregime of $\sbf$ with maximal element $Q$. 
Moreover,
\begin{equation}\label{eq7.4}
\max\left(|u(Y_{Q'}^+) - u(Y^+_{Q})|,|u(Y_{Q'}^-) - u(Y^-_{Q})|\right) \leq \eps/10\,,\quad \forall Q'\in\sbf'(Q)\,.
\end{equation}
\end{remark}

Next, we establish packing conditions for the red cubes, 
and for the generation cubes.  We consider first the red cubes.
Our goal is to prove that for all $Q_0\in\dd(E)$
\begin{equation}\label{eq7.5}
\sum_{Q\subset Q_0:\, Q \,\, {\rm is\, red}}\sigma (Q) \leq \,C \eps^{-2} \,\sigma(Q_0)\,,
\end{equation}
where $C$ depends upon $\eta,K,\tau,n$ and the ADR/UR constants of $E$.
To this end, let $Q$ be any red cube, let $U_Q=U_{Q,\tau}$ be its associated Whitney region,
and let $\widehat{U}_Q:= U_{Q,2\tau}$ 
be a fattened version of $U_Q$. 
Note that
$ \ell(Q)^{n+1} \approx |U_Q|\approx |\widehat{U}_Q|$, 
and similarly for each connected component 
of the Whitney regions.
By definition, if $Q$ is red, then
$U_Q$ has at least one red component
$U_Q^{i}$, and every red $U^i_Q$ satisfies
\begin{equation}\label{eq7.6} \eps^2 
\,\lesssim \,\left(\osc_{U_Q^{i}} u \right)^2
\,\lesssim\,\ell(Q)^{1-n} \iint_{\widehat{U}_Q^{i}} |\nabla u|^2
\,\lesssim\, \ell(Q)^{-n} \iint_{\widehat{U}_Q} |\nabla u(Y)|^2 \delta(Y) \, dY\,,
\end{equation}
where we have used  \eqref{eq7.1}, local boundedness estimates of Moser type, Poincar\'e's inequality, and the fact
that $\delta(Y) \approx \ell(Q)$ in $\widehat{U}_Q$.   We leave the details to the reader
(or cf. \cite[Section 4]{HM-I}),
but we remark that the key fact is that the Harnack Chain condition holds in each 
component $U^i_Q$.
Here, the various implicit constants may depend upon $\tau, \eta$ and $K$.
By the ADR property, \eqref{eq7.6} implies that
\begin{multline*}
\sum_{Q\subset Q_0:\, Q \,\, {\rm is\, red}}\sigma(Q)\,\lesssim\, \eps^{-2}
\sum_{Q\subset Q_0} \iint_{\widehat{U}_Q} |\nabla u(Y)|^2 \delta(Y) \, dY\\[4pt]
\lesssim \,\eps^{-2}
 \iint_{B^*_{Q_0}} |\nabla u(Y)|^2 \delta(Y) \, dY\, \lesssim\, \eps^{-2}
 \sigma(Q_0)\,,
 \end{multline*}
where in the second inequality we have used that the Whitney regions $\widehat{U}_Q$
have the bounded overlap property, 
and for $Q\subset Q_0$,
are contained in $B^*_{Q_0}:= B(x_{Q_0},K\ell(Q))$ by \eqref{eq3.3aab};
the third inequality is Theorem \ref{t1}, since $\|u\|_\infty\leq 1$.

We now augment the ``bad" collection $\B$ from Lemma \ref{lemma2.1} by setting
\begin{equation}\label{eqBt}\Bt:=\B \cup \{Q\in \dd(E):\, Q \,\,{\rm is\,red}\}\,.
\end{equation}
Since 
the collection $\B$ is already endowed with a packing condition,
estimate \eqref{eq7.5} immediately improves to the following
\begin{equation}\label{eq7.7}
\sum_{Q\subset Q_0:\, Q \in \Bt}\sigma (Q) \leq C \eps^{-2}\,\sigma(Q_0)\,,
\end{equation}
where again $C=C(\eta,K,\tau,n,{\rm ADR/UR})$.

Let us now turn to the packing condition for the generation cubes.
We first establish the following.
\begin{lemma}\label{lemma7.8} Let $\sbf$ be one of the stopping time regimes of Lemma
\ref{lemma2.1}, and for $k\geq 0$, let $Q^k\in G_k(\sbf)$ be a generation cube.
Then
$$\sum_{Q\in G_1(Q^k)} \sigma(Q) \leq \,C\eps^{-2}
\iint_{\Omega_{\sbf'(Q^k)}}|\nabla u(Y)|^2 \delta(Y)\, dY\,, $$
where $\sbf'(Q^k)$ is the semi-coherent subregime with maximal element $Q^k$ (cf. Remark \ref{r7.3}),
$\Omega_{\sbf'(Q^k)}$ is the associated ``sawtooth" domain (cf. Remark \ref{r3.11}), and $C$
depends on $\eta,K,\tau,n$, and the ADR/UR constants for $E$.
\end{lemma}

To prove the lemma, we shall need to introduce the non-tangential maximal function.  Given a domain
$\Omega'\subset \ree$, and $u\in C(\Omega')$, for $x\in \pom'$, set
$$N^{\Omega'}_*u(x):= \sup_{Y\in \Gamma_{\Omega'}(x)} |u(Y)|\,,$$
where for some $\kappa>0$,
\begin{equation}\label{eq7.9}
\Gamma_{\Omega'}(x):= \left\{Y\in \Omega': |Y-x|\leq (1+\kappa) \dist(Y,\pom')\right\}\,.
\end{equation}

\begin{proof}[Proof of Lemma \ref{lemma7.8}]
Let $Q \in G_1(Q^k)$, so in particular, $Q\in G_{k+1}(\sbf)$,
and let
$\Qt$ be the dyadic parent of $Q$.
We note that
$\Qt \in\sbf'(Q^k)$,
by maximality of the generation cubes (more precisely, by maximality of the stopping time family
$\F_1(Q^k)$ that contains $G_1(Q^k)$).
By the stopping time
construction, since $Q$ belongs to $\sbf$, we must have
$$\max\left(|u(Y_{Q}^+) - u(Y^+_{Q^k})|,|u(Y_{Q}^-) - u(Y^-_{Q^k})|\right) > \eps/10\,.$$
Let $G_1^+,\, G_1^-$ denote the subcollections of $G_1(Q^k)$ for which the
previous estimate holds with ``$+$", and with ``$-$", respectively (if both hold, then we arbitrarily assign $Q$ to $G_1^+$).
For the sake of specificity, we treat $G^+_1$;  the argument for $G_1^-$ is the same.
For every $Q\in G^+_1$, we have
\begin{equation}\label{eq7.10}
\frac{\eps^2}{100}\,\leq \, 
|u(Y_Q^+)-u(Y^+_{Q^k})|^2\,.
\end{equation}
To simplify notation, we set $\Omega':= \Omega^+_{\sbf'(Q_k)}$.
By construction (cf. Definition \ref{d3.2} and Remarks \ref{r3.11a} and \ref{r3.11}), since $\Qt\in\sbf'(Q^k)$,
we have that $Y^+_Q\in \interior U_{\Qt}^+\subset \Omega'$, and
$$\ell(Q) \lesssim \dist(Y_Q^+,\partial U_{\Qt}^+)\leq\dist(Y_Q^+,\partial\Omega')
\leq \delta(Y_Q^+) \approx \dist(Y^+_Q,Q) \lesssim \ell(Q)\,,$$
with implicit constants possibly depending on $\eta$ and $K$.
Consequently, there is a point $z_Q^+\in \partial\Omega'$, with
$|z_Q^+-Y_Q^+|\approx\ell(Q)\approx |x_Q-Y_Q^+|$, where as usual $x_Q$ is the ``center" of $Q$.
For each $Q\in G_1^+$, we set $B'_Q:= B(z^+_Q,\ell(Q))$, $B''_Q:= B(x_Q,M\ell(Q))$,
and we fix $M$  large enough (possibly depending on $\eta$ and $K$),
that  $B'_Q\subset B''_Q$.  By a standard covering lemma argument, we can extract
a subset of $G^+_1$, call it $G_1^{++}$, such that $B''_{Q_1}$ and $B''_{Q_2}$ are disjoint,
hence also $B'_{Q_1}$ and $B'_{Q_2}$ are disjoint,
for any pair of cubes $Q_1,\,Q_2\in G_1^{++}$, and moreover,
\begin{equation}\label{eq7.11}
\sum_{Q\in G_1^{+}}\sigma(Q)\leq C_M \sum_{Q\in G_1^{++}}\sigma(Q)=
C_{\eta,K} \sum_{Q\in G_1^{++}}\sigma(Q)\,.
\end{equation}
We may now fix the parameter $\kappa$ large enough in \eqref{eq7.9},
so that $Y_Q^+\in \Gamma_{\Omega'}(z)$,
for all $z\in B'_Q\cap\pom'$. 
Combining \eqref{eq7.10} and \eqref{eq7.11}, 
we then obtain
\begin{multline*}\eps^{2}\sum_{Q\in G_1^+} \sigma(Q)
\,\lesssim\, \eps^{2}   \sum_{Q\in G_1^{++}}\sigma(Q)\\[4pt]
\lesssim\,\sum_{Q\in G_1^{++}}  
 |u(Y_Q^+)-u(Y^+_{Q^k})|^2\,\sigma(Q) 
\\[4pt]
= \sum_{Q\in G_1^{++}}  |u(Y_Q^+)-u(Y^+_{Q^k})|^2\, \sigma(Q)
\, \fint_{B'_Q\cap \partial  \Omega'}
dH^n \\[4pt]
\lesssim \,  \sum_{Q\in G_1^{++}}\int_{B'_Q\cap \partial  \Omega'}
\left(N_*^{\Omega'} \left(u-u(Y^+_{Q^k})\right)\right)^2 dH^n \lesssim \int_{\partial  \Omega'}
\left(N_*^{\Omega'} \left(u-u(Y^+_{Q^k})\right)\right)^2 dH^n\,,
\end{multline*}
where in the last two inequalities, we have used that $\pom'$ is ADR (by Lemma \ref{lemma3.15}),
and that the balls $B'_Q$ are disjoint, for $Q\in G^{++}_1$.
The implicit constants depend on $\eta$ and $K$.
Now, by Lemma \ref{lemma3.15},  $\Omega'$ is NTA with an ADR boundary, and therefore
harmonic measure for $\Omega'=\Omega^+_{\sbf'(Q^k)}$ is $A_\infty$ with respect to surface measure
on $\pom'$, by \cite{DJe}.  Consequently, by \cite{DJK}, we have
\begin{equation}\label{eq7.11.2}\eps^2\sum_{Q\in G_1^+} \sigma(Q) \lesssim\int_{\partial \Omega^+_{\sbf'(Q^k)}}\!\!
\left(N_*^{\Omega'} (u-u(Y^+_{Q^k}))\right)^2 \lesssim
\iint_{\Omega^+_{\sbf'(Q^k)}}|\nabla u(Y)|^2 \delta(Y)\, dY\,.
\end{equation}
Combining the latter estimate with its analogue for $G_1^-$ and $\Omega^-_{\sbf'(Q^k)}$, we obtain the conclusion of the lemma.
\end{proof}
We are now ready to establish the packing property of the generation cubes.  Recall that $G^*$ denotes the collection
of all generation cubes, running over all the stopping time regimes $\sbf$ constructed in Lemma \ref{lemma2.1}.
\begin{lemma}\label{lemma7.12} Let $Q_0\in\dd(E)$.
Then
\begin{equation}\label{eq7.13}
\sum_{Q\subset Q_0:\, Q\in G^*} \sigma(Q) \,\leq\, C\eps^{-2}\sigma(Q_0)\,.
\end{equation}
\end{lemma}
\begin{proof}
Fix $Q_0\in\dd(E)$.
Let $M(Q_0)$ be the collection of maximal generation cubes contained in $Q_0$, i.e.,
$Q_1\in M(Q_0)$ if $Q_1\in G^*$, and there is no other $Q'\in G^*$ with
$Q_1\subset Q'\subset Q_0$.  By maximality, the cubes in $M(Q_0)$ are disjoint, so
it is enough to prove \eqref{eq7.13} with $Q_0$ replaced by an arbitrary $Q_1\in M(Q_0)$,
i.e., to show that for any such $Q_1$,
\begin{equation}\label{eq7.14}
\sum_{Q\subset Q_1:\, Q\in G^*} \sigma(Q)  \,\leq\, C\eps^{-2} \sigma(Q_1)\,.
\end{equation}
Since $Q_1$ is a generation cube,  it belongs, by construction, to some $\sbf$, say $\sbf_0$.
Let $\mathfrak{S}=\mfs(Q_1)$ be the collection of all stopping time regimes $\sbf$, excluding
$\sbf_0$, such that $Q(\sbf)$ meets $Q_1$ and $\sbf$ contains at least one subcube of $Q_1$. Then necessarily,
$Q(\sbf)\subsetneq Q_1$, for all $\sbf\in\mfs$.
The left hand side of \eqref{eq7.14} then equals
$$\sum_{Q\subset Q_1:\, Q\in G(\sbf_0)} \sigma(Q)\,+\,\sum_{\sbf\in \mfs}\sum_{Q\in G(\sbf)} \sigma(Q)
\,=: I + II\,. $$
We treat term $I$ first.  We define $G_0(Q_1) = \{Q_1\}$, $G_1(Q_1)$, $G_2(Q_1)$, \dots, etc.,
by analogy to the definitions of $G_k(Q^0)$ above (indeed, this analogy
was implicit in our construction).
We then have
\begin{equation*}
I= \sum_{k\geq 0} \sum_{Q\in G_k(Q_1)} \sigma(Q)
=\sigma(Q_1) + \sum_{k\geq 1}\sum_{Q'\in G_{k-1}(Q_1)} \sum_{Q\in G_1(Q')} \sigma(Q)
=: \sigma(Q_1) + I'\,.
\end{equation*}
By Lemma \ref{lemma7.8},
\begin{multline*}
I' \lesssim \, \eps^{-2}\sum_{k\geq 1}\sum_{Q'\in G_{k-1}(Q_1)}
 \iint_{\Omega_{\sbf'(Q')}}|\nabla u(Y)|^2 \delta(Y)\, dY\\[4pt]
\leq \, \eps^{-2} \sum_{k\geq 1}\sum_{Q'\in G_{k-1}(Q_1)} \sum_{Q\in \sbf'(Q')}
\iint_{U_Q}|\nabla u(Y)|^2 \delta(Y)\, dY\\[4pt]
\lesssim\, \eps^{-2} \iint_{T_{Q_1}}|\nabla u(Y)|^2 \delta(Y)\, dY\,\lesssim \,\eps^{-2} \sigma(Q_1)\,,
\end{multline*}
where in the second inequality we have used the definition of $\Omega_{\sbf'}(Q')$ (cf. Remark \ref{r3.11}), and
in the third inequality that the triple sum runs over a family of distinct cubes, all contained in $Q_1$
(cf. Remark \ref{r7.3}),
and that the Whitney regions $U_Q$ have bounded overlaps;
the last inequality is Theorem \ref{t1}, by virtue of \eqref{eq3.3aab}, since $\|u\|_\infty\leq 1$.  Thus, we have established
\eqref{eq7.14} for term $I$.

Consider now term $II$.  The inner sum in $II$,  for a given $\sbf$,  is
$$\sum_{Q\in G(\sbf)} \sigma(Q) = \sum_{k\geq 0} \sum_{Q\in G_k(\sbf)} \sigma(Q)\,.$$
But by definition, $G_k(\sbf) = G_k(Q(\sbf))$, so this inner sum is therefore
exactly the same as term $I$ above, but with $Q(\sbf)$ in place of $Q_1$.  Consequently,
we obtain, exactly as for term $I$, that
$$\sum_{Q\in G(\sbf)} \sigma(Q) \,\lesssim \,\eps^{-2} \sigma\big(Q(\sbf)\big)\,.$$
Plugging the latter estimate into term $II$, and using the definition of $\mfs$,
we have
$$II \,\lesssim\, \eps^{-2} \sum_{\sbf:\, Q(\sbf)\subset Q_1}  \sigma\big(Q(\sbf)\big)\,\lesssim \,
\eps^{-2} \sigma(Q_1)\,,$$
by the packing condition for the maximal cubes $Q(\sbf)$, established in Lemma \ref{lemma2.1}.
\end{proof}

Our next task is to define the approximating function $\vp$.  To this end, fix $Q_0\in\dd(E)$.
We shall first define certain auxiliary functions $\vp_0$, $\vp_1$, which we then blend together to get $\vp$.
 We are going to find an ordered family of cubes $\{Q_k\}_{k\ge 1}\in\mathcal{G}$ and to introduce the first cube $Q_1$ let us consider two cases. In the first case we assume that $Q_0\notin \G$ and let $Q_1$ be the subcube of $Q_0$, of largest ``side length", that belongs to $\G$.
By the packing condition for $\B$, there must of course be such a $Q_1$.
It may be that $Q_0$ has more than one proper subcube  in $\G$, all of the same
maximum side length, in this case we just pick one.
Then $Q_1$, being in $\G$, and hence in some $\sbf$,
must therefore
belong to some subregime $\sbf'_1$ (cf. Remark \ref{r7.3}), and in fact $Q_1= Q(\sbf'_1)$
(since the dyadic parent of $Q_1$ belongs to $\dd_{Q_0}\cap\B$).
The second case corresponds to  $Q_0\in \G$. Then, in particular, $Q_0$ belongs to some $\sbf$, and therefore to some $\sbf'_1$,
and again we set $Q_1:= Q(\sbf_1')$.  In this case,
$Q_0$ could be a proper subset of $Q_1$, or else  $Q_1=Q_0$. Once we have constructed $Q_1\in\mathcal{G}$ in the two cases, we then let $Q_2$ denote the subcube of
maximum side length in $(\dd_{Q_0}\cap\G) \setminus \sbf_1'$, etc.,
thus obtaining an enumeration $Q_1$, $Q_2$, \dots$\in\mathcal{G}$ 
such that
$$\ell(Q_1)\geq\ell(Q_2)\geq\ell(Q_3)\geq \dots\,,$$ 
$Q_k =Q(\sbf'_k)$,
and $\G\cap\dd_{Q_0} \subset \cup_{k\geq 1} \sbf'_k$.  The latter property follows easily from the construction, since from one step to the next one, we take a cube with maximal side length in $\mathcal{G}\cap \dd_{Q_0}$ that is not in the previous subregimes. This procedure exhausts the collection of cubes $\mathcal{G}\cap \dd_{Q_0}$. Further, we note that $\G\cap\dd_{Q_0} =\cup_{k\geq 1} \sbf'_k$ when
$Q_1\subset Q_0$.  We point out that, certainly, the various
subregimes $\sbf'_k$ need not all be contained in the same original regime $\sbf$.
We define recursively
$$A_1:= \Omega_{\sbf_1'};\qquad A_k:= \Omega_{\sbf_k'}\setminus
\left(\cup_{j=1}^{k-1} A_j\right)\,, \quad k\geq 2,
$$
so that the sets $A_k$ are pairwise disjoint.
Note that $\cup_{j=1}^{k} A_j = \cup_{j=1}^{k} \Omega_{\sbf'_j}$.
We also set
$$\ot:=\cup_k \Omega_{\sbf_k'} =\cup_k A_k\,, $$
and 
$$
A_1^\pm:=  \Omega^\pm_{\sbf_1'}\,;\qquad A^\pm_k:=  \Omega^\pm_{\sbf_k'}\setminus
\left(\cup_{j=1}^{k-1}A_j\right)\,,\,\,k\geq 2\,,
$$
which induces the corresponding splitting
$\ot=\ot^+\cup\ot^-$, 
where $\ot^\pm:= \bigcup_kA_k^\pm.$
We now define $\vp_0$ on $\ot$ by setting
$$\vp_0:= \sum_k \left(u\big(Y^+_{Q_k}\big)1_{A_k^+}\,+
\,u\big(Y^-_{Q_k}\big)1_{A_k^-}\right)\,.$$ 

Next, let $\{Q(k)\}$ be some fixed enumeration of the cubes in $\Bt\cap\dd_{Q_0}$
(cf. \eqref{eqBt} for the definition of $\B^*$).
We define recursively
$$V_1:= U_{Q(1)}\,;\qquad V_k:= U_{Q(k)}\setminus \left(\cup_{j=1}^{k-1} V_j\right)\,,\quad k\geq 2\,.$$
For each  $Q(k)$, we split the corresponding
Whitney region $U_{Q(k)}$ into its connected components $U_{Q(k)}=\cup_i U_{Q(k)}^i$
(note that the number of such components is uniformly bounded; cf. Remark \ref{r6.2}),
and we observe that this induces a corresponding splitting
$$
V^i_1:= U^i_{Q(1)}\,;\qquad V^i_k:= U^i_{Q(k)}\setminus \left(\cup_{j=1}^{k-1} V_j\right)\,,\quad k\geq 2\,.$$
On each $V^i_k$ we define
\begin{equation*}
\vp_1(Y):=\left\{
\begin{array}{l}
u(Y)\,,
\,\,\, {\rm if}\, U_{Q(k)}^i \,\,{\rm is \, red}
\\[6pt]
u(X_I)\,, \,\, {\rm if}\,  U_{Q(k)}^i \,\,{\rm is \, blue}
\end{array}
\right.\,, \quad Y\in  V^i_k\,,
\end{equation*}
where for each blue component $ U_{Q(k)}^i$ we have specified a
fixed Whitney box $I\subset  U_{Q(k)}^i$,  with center
$X_I$.  
In particular, we have thus defined $\vp_1$ on
\begin{equation}\label{omegaone}
\Omega_1:=\interior\left( \cup_{Q\in\Bt\cap\dd_{Q_0}}U_Q\right)
= \interior\left( \cup_k V_k\right)\,.
\end{equation}
We extend $\vp_0$ and $\vp_1$ to all of $T_{Q_0}$ by
setting each equal to
0 outside of its original domain of definition.
The supports of $\vp_0$ and $\vp_1$ may overlap:
it is possible that a red cube may belong to $\G$ as well as to
$\Bt$, and in any case
the various Whitney regions $U_Q$ may overlap (in a bounded way) for different cubes $Q$.
On the other hand, note that, up to a set of measure 0,  $T_{Q_0} \subset \Omega_0 \cup \Omega_1$
(with equality, again up to a set of measure 0, holding in the case that $Q_1\subset Q_0$).
Finally, we define $\vp$ as a measurable function on 
$T_{Q_0}$ by setting
\begin{equation*}
\vp(Y):=\left\{
\begin{array}{l}
\vp_0(Y)\,,
\,\,\, Y\in T_{Q_0}\setminus \overline{\ott}
\\[6pt]
\vp_1(Y)\,, \,\, Y\in \ott\,.
\end{array}
\right.
\end{equation*}
Then
$\|u-\vp\|_{L^\infty(T_{Q_0})}<\eps$.  Indeed, in $\Omega_1$, $\vp$ is equal either to $u$,
or else to $u(X_I)$, with $X_I$ in some ``blue" component with small oscillation; otherwise,
if $Y\in T_{Q_0}\setminus \overline{\Omega_1}$, then (modulo a set of measure 0),
$Y$ lies in some $A_k^\pm\subset \om_{\sbf'_k}^\pm$, and moreover, $Y$ also lies in some blue
$U_Q^\pm\subset \om_{\sbf'_k}^\pm$, whence it follows that $u(Y)-\vp(Y)=u(Y) -u(Y_{Q_k}^\pm)$ is
small by construction.

It remains to verify the Carleson measure estimate for the measure $|\nabla\vp(Y)|dY$.
We do this initially for $\vp_0$ and $\vp_1$ separately.   Let $Q'\subset Q_0$, and
consider first $\vp_0$.
 We shall require the following:
\begin{lemma}\label{lemma6.intersect}
Fix $Q\in \dd(E)$, and its associated Carleson box $T_Q$.  Let $G(Q)$ be the collection of all
generation cubes $Q'$, with $\ell(Q')\geq\ell(Q)$, such that $\Omega_{\sbf'(Q')}$ meets $T_Q$.  Then there
is a uniform constant $N_0$ such that the cardinality of
$G(Q)$ is bounded by $N_0$.
\end{lemma}
\begin{proof}
Let $Q'\in G^*$, and suppose that  $\ell(Q')\geq\ell(Q)$, and that $\Omega_{\sbf'(Q')}$ meets $T_Q$.
Then there are two cubes $P'\in \sbf'(Q')$, and
$P\subset Q$, such that there is some $I\in \W_{P'}$, and $J\in \W_{P}$, for which $I^*$ meets $J^*$
(of course, it may even be that $I=J$, but not necessarily).  By construction of the collections $\W_Q$,
$$\dist(P',P)\lesssim \ell(P')\approx \ell(I)\approx \ell(J) \approx\ell(P)\leq \ell(Q)\leq\ell(Q')\,.$$
By the semi-coherency of $\sbf'(Q')$, we may then choose $R'\in\sbf'(Q')$
such that $P'\subset R'\subset Q'$,
with $\ell(R')\approx\ell(Q)$.  Note that
$\dist(R',Q) \lesssim \ell(Q)$.
The various implicit constants are of course uniformly controlled, and therefore the number of such
$R'$ is also uniformly controlled.  There exists such an $R'$ for every $Q'\in G(Q)$; moreover,
a given $R'$ can correspond to only one $Q'$, since the regimes $\sbf'$ are pairwise disjoint.  Thus,
the cardinality of $G(Q)$ is uniformly bounded by a number $N_0$ that depends on the ADR constant.
\end{proof}

Suppose now that  
$j<k$, hence
$\ell(Q_j)\geq\ell(Q_k)$.  Since $\Omega_{\sbf'}\subset T_{Q(\sbf')}$ by construction
(cf. Remark \ref{r3.11}),
$\Omega_{\sbf'_j}$ meets $\Omega_{\sbf'_k}$ only if
$\Omega_{\sbf'_j}$ meets $T_{Q_k}$.   By Lemma \ref{lemma6.intersect}, the number of indices $j$
for which this can happen, with $k$ fixed, is bounded by $N_0$.
Consequently, since $\cup_{j=1}^{k-1} A_j = \cup_{j=1}^{k-1} \Omega_{\sbf'_j}$,
it follows that for each $k\geq 2$, there is a subsequence $\{j_1,j_2,\dots,j_{N(k)}\}\subset \{1,2,\dots,k-1\}$,
 with $\sup_k N(k)\leq N_0$, such that
$$A_k = \Omega_{\sbf'_k}\setminus \left(\cup_{i=1}^{N(k)}  \Omega_{\sbf'_{j_i}}\right)\,,$$
and hence,
\begin{equation}\label{eq6.16}
\partial A^\pm_k\, \subset\, \partial \Omega^\pm_{\sbf'_k} \cup
\left(\overline{\Omega^\pm_{\sbf'_k}}\,\cap \,\big(\cup_{i=1}^{N(k)} \partial  \Omega_{\sbf'_{j_i}}\big)\right)\,.
\end{equation}
Observe that by definition of $\vp_0$, in the sense of distributions
$$\nabla\vp_0 =  \sum_k \left(u\big(Y^+_{Q_k}\big)\nabla 1_{A_k^+}\,+
\,u\big(Y^-_{Q_k}\big)\nabla 1_{A_k^-}\right)\,,$$
so that, since $\|u\|_\infty\leq 1$,
\begin{multline*}
\iint_{T_{Q'}} |\nabla\vp_0|\, \leq\, \sum_k \iint_{T_{Q'}}\left( |\nabla 1_{A_k^+}| +|\nabla 1_{A_k^-}|\right)\\[4pt]
\leq\, 
\sum_k H^n(T_{Q'}\cap\partial A^+_k) \,
+\sum_k H^n(T_{Q'}\cap\partial A^-_k)\,=: I^+ + I^-\,.
\end{multline*}
Consider $I^+$, which we split further into
$$I^+ = \sum_{k: Q_k\subset Q'}
H^n(T_{Q'}\cap\partial A^+_k)\,+\sum_{k: Q_k\nsubseteq Q'}
H^n(T_{Q'}\cap\partial A^+_k)=: I^+_1+I^+_2\,.$$
We treat $I^+_1$ first.  Note that by Proposition \ref{prop:Sawtooths-ADR} in Appendix \ref{appa} below,
and \eqref{eq6.16}, $\partial A_k^\pm$ satisfies the upper ADR bound, because it is contained in
the union of a uniformly bounded number of sets with that property.  In addition,
$\partial A_k^\pm\subset \overline{\Omega_{\sbf'_k}}$, which has diameter
$\diam(\Omega_{\sbf'_k})\lesssim \ell(Q_k)$.  Therefore,
$$I_1^+\lesssim  \sum_{k: Q_k\subset Q'}\ell(Q_k)^n \approx  \sum_{k: Q_k\subset Q'}
\sigma(Q_k)\lesssim \eps^{-2} \sigma(Q')\,,$$
by the packing condition \eqref{eq7.13}, since each $Q_k$ is a generation cube.

Next, we consider $I_2^+$.  Recall that $\overline{A_k}\subset \overline{\Omega_{\sbf'_k}}$,
and note that 
\begin{equation}\label{eq6.17a}
T_{Q'} \,\, {\rm meets}\,\, \overline{\Omega_{\sbf'_k}}
\,\implies\, \dist(Q',Q_k) \lesssim \min(\ell(Q'),\ell(Q_k))
\end{equation}
(with implicit constants depending on $\eta$ and $K$).
By Lemma \ref{lemma6.intersect}, the
number of such $Q_k$ with $\ell(Q_k)\geq\ell(Q')$ 
is uniformly bounded
(depending on $\eta,K,$ and the ADR constant).  Moreover, as noted above,
$\partial A_k^\pm$ satisfies the upper ADR bound.  Thus,
$$\sum_{k: Q_k\nsubseteq Q',\, \ell(Q_k)\geq \ell(Q')}
H^n(T_{Q'}\cap\partial A^+_k) \lesssim \left(\diam(T_{Q'})\right)^n\approx \sigma(Q')\,.$$
On the other hand, 
if  
$\ell(Q_k)\leq \ell(Q')$, then by \eqref{eq6.17a}, every relevant
$Q_k$ is contained either in $Q'$, or in some ``neighbor" $Q''$ of $Q'$, of
the same ``side length", with $\dist(Q',Q'')\leq C\ell(Q')$ for some (uniform)
constant $C$.
Since the number of such neighbors $Q''$ is uniformly
bounded, 
the terms in $I_2^+$ with $\ell(Q_k)< \ell(Q')$
may be handled exactly like term $I^+_1$.

The term $I^-$ may be handled just like $I^+$, and therefore, combining our estimates for
$I^\pm$, we obtain
the Carleson measure bound
\begin{equation}\label{eq6.17}
\sup_{Q\subset Q_0} \frac1{|Q|}\iint_{T_Q}|\nabla \vp_0|\,\lesssim\eps^{-2}\,.
\end{equation}

Next, we consider $\vp_1$.  Again let $Q'\subset Q_0$.
Recall that $V_k \subset U_{Q(k)}$, and note that
\begin{equation}\label{eq}
U_{Q(k)}\,\,{\rm  meets}\,\, U_{Q(k')}\,\implies\,\dist(Q(k), Q(k'))\lesssim \ell(Q(k))\approx \ell(Q(k'))\,,
\end{equation}
and thus, for any given $Q(k)$, there are at most a uniformly
bounded number of such $Q(k')$ for which this can happen.  Therefore, since $\cup_{j=1}^k V_j =
\cup_{j=1}^k U_{Q(j)}$, it follows that for each $k\geq 2$, there
is a subsequence $\{j_1,j_2,\dots,j_{N'(k)}\}\subset \{1,2,\dots,k-1\}$,
 with $\sup_k N'(k)\leq N'_0$, such that
$$V_k = U_{Q(k)}\setminus \left(\cup_{i=1}^{N'(k)} U_{Q(j_i)}\right)\,,$$
and hence
\begin{equation*}
\partial V_k \subset \partial U_{Q(k)}\,\cup \left(U_{Q(k)} \cap \big(\cup_{i=1}^{N'(k)} \partial U_{Q(j_i)}\big)\right)\,,
\end{equation*}
where each $Q(j_i)$ has side length comparable to that of $Q(k)$.
Consequently,  by construction of the Whitney regions, $\partial V_k$ is covered by the union of
a uniformly bounded number of faces of fattened Whitney boxes $I^*$, each with $\ell(I^*)\approx \ell(Q(k))$,
so that
\begin{equation}\label{eq6.21}
H^n(\partial V_k) \lesssim \ell(Q(k))^n \approx \sigma(Q(k))\,.
\end{equation}
\begin{remark}\label{r6.22}
Recall that $\supp(\vp_1) \subset \overline{\ott}= \overline{\cup_k V_k}$ (cf. \eqref{omegaone}),
and note that since $V_k\subset U_{Q(k)}$, 
the closure of
a given $V_k$ can meet $T_{Q'}$ only if $\ell(Q(k))\lesssim\ell(Q')$ and
$\dist(Q(k),Q')\lesssim\ell(Q')$, thus, there is a collection $\nn(Q')$, of uniformly bounded cardinality,
comprised of cubes $Q^*$ with
$\ell(Q^*)\approx \ell(Q')$, and $\dist(Q^*,Q') \lesssim \ell(Q')$, such that $Q(k) \subset Q^*$ for some $Q^*\in\nn(Q')$,
whenever $\overline{V_k}$ meets $T_{Q'}$.
Here, the various implicit constants may depend upon $\eta$, $K$ and the ADR bounds.
\end{remark}
Using the notation of the Remark \ref{r6.22}, we then have that
\begin{multline*}
\iint_{T_{Q'}} |\nabla\vp_1|
=\iint_{T_{Q'}\cap\,\ott} |\nabla\vp_1|\\[4pt]
 \leq \sum_{Q^*\in\nn(Q')} \sum_{Q(k)\subset Q^*} \iint_{V_k}|\nabla\vp_1|
=\sum_{Q^*\in\nn(Q')} \sum_{Q(k)\subset Q^*}\sum_i \iint_{V_k^i}|\nabla\vp_1|\,.
\end{multline*}
If $U_{Q(k)}^i$ is a blue component, then, since $\|u\|_\infty \leq 1$,
$$\iint_{V^i_{k}}|\nabla\vp_1|\leq \iint_{V^i_{k}}|\nabla1_{V_{k}^i}|
\leq H^n(\partial V^i_{k}) \leq H^n(\partial V_{k})\lesssim \sigma(Q(k))\,,$$
where in the last step we have used \eqref{eq6.21}. 
Since for all $Q$, the number of components $U_Q^i$ is
uniformly bounded (cf. Remark \ref{r6.2}), we obtain
$$\sum_{Q^*\in\nn(Q')} \sum_{Q(k)\subset Q^*}\sum_{i:  U_{Q(k)}^i\, {\rm blue} }  \iint_{V_k^i}|\nabla\vp_1|
\lesssim \, \sum_{Q^*\in\nn(Q')} \sum_{Q(k)\subset Q^*} \sigma(Q(k))
\lesssim \, \eps^{-2}\sigma(Q')\,, $$
by the packing condition for $\Bt$ (cf. \eqref{eq7.7}, and recall that $\{Q(k)\}$ is an enumeration of
$\Bt\cap \dd_{Q_0}$), and the nature of the cubes $Q^*$ in $\nn(Q')$ along with the ADR property.

On the other hand,
if $U_{Q(k)}^i$ is a red component (cf. \eqref{eq7.1}), then 
by \eqref{eq7.6} and the ADR property,
\begin{equation}\label{eq7.16}
\sigma(Q(k)) \lesssim \eps^{-2}  \iint_{\widehat{U}_{Q(k)}} |\nabla u(Y)|^2 \delta(Y) \, dY\,,
\end{equation}
where $\widehat{U}_{Q(k)}:= U_{Q(k),2\tau}$ is a fattened version of $U_{Q(k)}$.
Consequently, for any red component $U_{Q(k)}^i$, bearing in mind that $\delta(Y) \approx \ell(Q(k))$ in
$U_{Q(k)}$, we have
\begin{multline*}
\iint_{V_k^i}|\nabla \vp_1|\,=\,\iint_{V_k^i}|\nabla u|
\,\lesssim\, \left(\iint_{V_k^i}|\nabla u|^2\right)^{1/2}  \ell(Q(k))^{(n+1)/2}\\[4pt]
\approx\,  \left(\iint_{V_k^i}|\nabla u(Y)|^2\delta(Y) dY \right)^{1/2}  \ell(Q(k))^{n/2}
\lesssim \, \eps^{-1}  \iint_{\widehat{U}_{Q(k)}} |\nabla u(Y)|^2 \delta(Y) \, dY\,,
\end{multline*}
where in the last step we have used \eqref{eq7.16} and the ADR property.
Thus,
\begin{multline*}
\sum_{Q^*\in\nn(Q')} \sum_{Q(k)\subset Q^*}\sum_{ i: U_{Q(k)}^i\, {\rm red} }  \iint_{V_k^i}|\nabla\vp_1|
\\[4pt]
\lesssim\,\eps^{-1} \sum_{Q^*\in\nn(Q')} \sum_{Q(k)\subset Q^*}
\iint_{\widehat{U}_{Q(k)}} |\nabla u(Y)|^2 \delta(Y) \, dY
\\[4pt]
\lesssim\,\eps^{-1} \iint_{B^*_{Q'}} |\nabla u(Y)|^2 \delta(Y) \, dY\, \lesssim
\,\eps^{-1} \sigma(Q')\,,
\end{multline*}
where $B_{Q'}^*:= B(x_{Q'},K\ell(Q'))$,
and in the last two steps we have used the bounded overlap property of
the Whitney regions $\widehat{U}_{Q}$,
the nature of $\nn(Q')$,
and Theorem \ref{t1}. 
Combining these estimates,
we obtain the Carleson measure bound
\begin{equation}\label{eq7.17}
\sup_{Q\subset Q_0} \frac1{|Q|}\iint_{T_Q}|\nabla \vp_1|\,\lesssim\eps^{-2}\,.
\end{equation}

Finally, we consider $\vp$.  By definition, in the sense of distributions,
$$\nabla \vp = \big(\nabla \vp_0 \big)1_{T_{Q_0}\setminus\overline{\ott}} \,\,
+\, \big(\nabla \vp_1 \big)1_{\ott}\, + \, J\,,$$
where $J$ accounts for the jump across $\partial\ott$.  The
contributions of the first two terms on the right hand side
may be treated by \eqref{eq6.17} and \eqref{eq7.17}, respectively.
To handle the term $J$, note that $\vp$ has a uniformly bounded jump
across $\pom_1$, since $\|u\|_\infty\leq 1$, and note also that we need only account for
the jump across $\pom_1$ in the interior of $T_{Q_0}$, thus, across the boundary of some $V_k$.
Note also that  $\partial V_k$ meets $T_Q$ only if $Q(k)\subset Q^*$, for some $Q^*\in\nn(Q)$
(see Remark \ref{r6.22}).
Hence, for $Q\subset Q_0$, we have
\begin{multline*}\iint_{T_Q}|J| \lesssim H^n (T_Q\cap \pom_1)\leq
\sum_k H^n(T_Q\cap \partial V_k)\\[4pt] \leq \sum_{Q^*\in\nn(Q)} \sum_{Q(k)\subset Q^*} H^n(\partial V_k)
\lesssim  \sum_{Q^*\in\nn(Q)} \sum_{Q(k)\subset Q^*} \sigma(Q(k)) \lesssim \eps^{-2} \sigma(Q)\,,
\end{multline*}
where in the last two steps we have used \eqref{eq6.21},
the packing condition for $\Bt$ (cf. \eqref{eq7.7}), and the nature of the cubes $Q^*$ in $\nn(Q)$
along with the ADR property.  Since $Q_0\in\dd(E)$ was arbitrary, we have therefore established
the existence of $\vp=\vp^\eps_Q$, satisfying $\|u-\vp\|_{L^\infty(T_Q)}<\eps$ and
\eqref{eq6.carldyadic}, for every $Q$.

The next step is to construct, for each $x\in E$ and each ball $B=B(x,r)$, and
for every $\eps\in (0,1)$, an appropriate
$\vp=\vp^\eps_B$ defined on $B\setminus E$.
Suppose first that $r<100\diam(E)$.  Exactly as in the proof that
\eqref{eq1.carl} implies \eqref{eq1.carl2}, there is a collection $\{Q_k\}$, of uniformly bounded cardinality,
such that $\ell(Q_k)\approx r$, for each $k$, and such that $B\setminus E\subset \cup_k T_{Q_k}$.
For each $Q_k$, we construct $\vp^\eps_{Q_k}$ as above.
Following our previous strategy, we recursively define
$$S_1:= T_{Q_1}\,,\,\,{\rm and}\,\, S_k:= T_{Q_k}\setminus\left(\cup_{j=1}^{k-1} S_j\right)\,,$$
and we define
$\vp=\vp^\eps_B:= \sum_k \vp^\eps_{Q_k} 1_{S_k}.$  The bound $\|u-\vp\|_{L^\infty(B\setminus E)}$
follows immediately from the corresponding bounds for $\vp_{Q_k}^\eps$ in $T_{Q_k}$.  Moreover, we obtain
the Carleson measure estimate
$$\sup_{z\in E, s>0, B(z,s)\subset B} \,\frac1{s^n}\iint_{B(z,s)}|\nabla\vp(Y)|dY \lesssim \eps^{-2}$$
from the corresponding bounds for $\vp^\eps_{Q_k}$ along with a now familiar argument to handle the
jumps across the boundaries of the sets $S_k$, using that the latter are covered by the union of the
boundaries of the Carleson boxes $T_{Q_k}$, which in turn are ADR by virtue of Proposition \ref{prop:Sawtooths-ADR}.
We omit the details.

Next, if $\diam(E)<\infty$, and $r\geq 100 \diam(E)$, we set $\widetilde{B}:= B(x, 10\diam(E))$, and
$$\vp =\vp_B^\eps:= \vp_{\widetilde{B}}^\eps1_{\widetilde{B}\setminus E} + u1_{B\setminus \widetilde{B}}\,,$$
and we repeat {\it mutatis mutandi}
the argument used above to show that \eqref{eq1.carl2} implies \eqref{eq1.carl*}, along with
our familiar arguments to handle the jump across $\partial \widetilde{B}$.  Again we omit the details.

Finally, we construct a globally defined $\vp=\vp^\eps$ on $\Omega$, satisfying \eqref{eq1.4} and \eqref{eq1.5},
as follows.  Fix $x_0\in E$, let $B_k:= B(x_0, 2^k),\, k=0,1,2,\dots$, and set $R_0:= B_0$, and $R_k:=B_k\setminus
B_{k-1}, k\geq 1$.  Define
$\vp=\vp^\eps:= \sum_{k=0}^\infty \vp_{B_k}^\eps 1_{R_k}.$
The reader may readily verify that $\vp$ satisfies \eqref{eq1.4} and \eqref{eq1.5}.
This concludes the proof of Theorem \ref{t2}.

\begin{remark}\label{remark5.29}  We note that the preceding proof did not require harmonicity of
$u$, {\it per se}, but only the following properties of $u$: 1) $u\in L^\infty(\Omega)$, with $\|u\|_\infty\leq 1$; 2) $u$ satisfies Moser's local boundedness estimates in $\Omega$; 3) $u$ satisfies the Carleson measure estimate \eqref{eq1.carl*}.
\end{remark}

\appendix

\section{Sawtooth boundaries inherit the ADR property} \label{appa}

\subsection{Notational conventions}

 Let us set some notational conventions that we shall follow throughout this appendix.
 If the set $E$ under consideration is merely ADR, but not UR, then we set
$\W_Q=\W^0_Q$ as defined in \eqref{eq3.1}.  If in addition, the set $E$ is UR, then we define
$\W_Q$ as in \eqref{Wdef}. In the first case, the constants involved in the construction
of $\W_Q$ depend only on the ADR constant $\eta$ and $K$, and in the UR case, on
dimension and the ADR/UR constants 
(compare \eqref{eq3.1} and \eqref{eq2.whitney2}). Therefore
there are numbers $m_0\in \mathbb{Z}_+,\, C_0 \in\mathbb{R}_+$, with the same dependence,
such that
\begin{equation}\label{eqA.1a}
2^{-m_0}\,\ell(Q)\leq \ell(I)  \leq 2^{m_0}\ell(Q),\, {\rm and}\, \dist(I,Q)\leq C_0 \ell(Q)\,,\quad \forall I\in \W_Q\,.
\end{equation} 
This dichotomy in the choice of $\W_Q$ is convenient for the results we have in mind. The main statements will pertain to the inheritance of the ADR property by local sawtooth regions and Carleson boxes whose definitions are built upon the exact choices of $\W_Q$'s described above, different for the ADR-only and ADR/UR case. 

We fix a small parameter $\tau>0$, and we define the Whitney regions $U_Q$, the Carleson boxes $T_Q$
and sawtooth regions
$\Omega_{\mathcal{F},Q}$, as in Section \ref{s3} (see \eqref{eq3.3aa}, \eqref{eq3.3a},
\eqref{eq2.discretesawtooth2} and 
\eqref{eq2.sawtooth1}), relative to $\W_Q$ as in the previous paragraph.  
We recall that if $\tau_0$ is chosen small enough, then for $\tau\leq\tau_0$, and for $I,J\in\W$, if
$I\neq J$,
then $I^*(\tau)$ misses $(3/4)J$.

For any $I\in\W$ such that $\ell(I) < \diam(E)$, we write $Q_I^*$ for the nearest dyadic cube to $I$ with $\ell(I) = \ell(Q_I^*)$ so that $I\in \W_{Q_I^*}$. Notice that there can be more than one choice of $Q_I^*$, but at this 
point we fix one so that in what follows $Q_I^*$ is unambiguously defined.

\subsection{Sawtooths have ADR boundaries}

\begin{proposition}\label{prop:Sawtooths-ADR}
Let $E\subset\ree$ be a closed $n$-dimensional ADR set\footnote{Thus,
$E$ may be UR, or not; in the former case, the parameters $m_0$ and $C_0$ may depend implicitly
on $n$ and the UR constants of $E$, as well as on $\eta$ and $K$; in either case, we follow the notational convention described above.}. Then all dyadic local sawtooths $\Omega_{\mathcal{F},Q}$ and all Carleson boxes $T_Q$ have $n$-dimensional ADR boundaries. In all cases, the implicit constants are uniform and depend only on dimension, the ADR constant of $E$ and the parameters $m_0$ and $C_0$.
\end{proposition}

The proof of this result follows the ideas from \cite[Appendix A.3]{HM-I} (see also \cite{HMM}).

We now fix $Q_0\in\dd$ and a family $\mathcal{F}$ of disjoint cubes $\F=\{Q_j\}\subset \mathbb{D}_{Q_0}$ (for the case $\F=\emptyset$ the changes are
straightforward and we leave them to the reader, also the case $\F=\{Q_0\}$ is disregarded since in that case
$\Omega_{\mathcal{F},Q_0}$ is the null set). We write $\Omega_\star=\Omega_{\mathcal{F},Q_0}$ and $\Sigma=\pom_{\star}\setminus E$.  Given $Q\in\dd$ we set
$$
\R_{Q}:=\bigcup_{Q'\in \dd_{Q}}\W_{Q'},
\qquad\mbox{and}\qquad
\Sigma_Q
=
\Sigma\bigcap \Big(\bigcup_{I\in \R_Q} I\Big).
$$

Let $C_1$ be a sufficiently large constant,
to be chosen below,
depending on $n$, the ADR constant of $E$, $m_0$ and $C_0$.
Let us introduce some new collections:
\begin{align*}
\F_{||}
&:=
\big\{
Q\in\dd\setminus\{Q_0\}:
\ell(Q)=\ell(Q_0), \ \dist(Q,Q_0)\le C_1\,\ell(Q_0)
\big\},
\\
\F_{\top}
&:=
\big\{
Q'\in \dd:
\dist(Q', Q_0)\le C_1\,\ell(Q_0),\
\ell(Q_0)<\ell(Q')\le C_1\,\ell(Q_0)
\big\},
\\
\F_{||}^*:
&=
\big\{Q\in\F_{||}: \Sigma_Q\neq\emptyset\big\}
=
\big\{Q\in\F_{||}: \exists\,I\in\R_{Q} \mbox{ such that } \Sigma\cap I\neq\emptyset\big\},
\\
\F^*:
&=
\big\{Q\in\F: \Sigma_Q\neq\emptyset\big\}
=
\big\{Q\in\F: \exists\,I\in\R_{Q} \mbox{ such that } \Sigma\cap I\neq\emptyset\big\},
\end{align*}
We also set
$$
\R_{\bot}=\bigcup_{Q\in\F^*} \R_Q,
\qquad\quad
\R_{||}=\bigcup_{Q\in\F_{||}^*} \R_Q,
\qquad\quad
\R_{\top}=\bigcup_{Q\in  \F_{\top}} \W_Q.
$$

\begin{lemma}\label{lemma:decomp-bdt-sawtooth}
Set $\W_\Sigma=\{I\in\W: I\cap \Sigma\neq\emptyset\}$ and define
\begin{align*}
\W_\Sigma^{\bot}
=
\bigcup_{Q\in\F^*} \W_{\Sigma,Q},
\qquad
\W_\Sigma^{||}
=
\bigcup_{Q\in\F_{||}^*} \W_{\Sigma,Q},
\qquad
\W_{\Sigma}^{\top}
=
\big\{I\in \W_\Sigma:Q_I^*\in\F_{\top}\big\}.
\end{align*}
where for every $Q\in\F^*\cup \F_{||}^*$ we set
$$\W_{\Sigma,Q}
=
\big\{I\in \W_\Sigma:Q_I^*\in\dd_{Q}\};
$$
and where we recall that $Q_I^*$ is the nearest dyadic cube to $I$ with $\ell(I) = \ell(Q_I^*)$ as defined above.
Then
\begin{equation}\label{decomp:Sigma:R}
\W_\Sigma
=
\W_\Sigma^{\bot}\cup\W_\Sigma^{||}\cup \W_\Sigma^{\top},
\end{equation}
where
\begin{equation}\label{decomp:Sigma:R:conta}
\W_\Sigma^{\bot}\subset \R_{\bot},
\qquad
\W_\Sigma^{||}\subset \R_{||},
\qquad
\W_\Sigma^{\top}\subset \R_{\top}.
\end{equation}
As a consequence,
\begin{equation}\label{decomp:Sigma}
\Sigma
=
\Sigma_{\bot}\cup\Sigma_{||}\cup\Sigma_{\top}
:=
\Big(\bigcup_{I\in\W_\Sigma^{\bot}} \Sigma\cap I\Big)
\bigcup
\Big(\bigcup_{I\in\W_\Sigma^{||}} \Sigma\cap I\Big)
\bigcup
\Big(\bigcup_{I\in\W_\Sigma^{\top}} \Sigma\cap I\Big).
\end{equation}
\end{lemma}

\begin{proof}
Let us first observe that if $I\in\W_\Sigma$, that is, $I\in\W$ is such that $I\cap\Sigma\neq\emptyset$,
 then $\interior(I^*)$ meets $\ree\setminus\Omega_\star$ and therefore $(3/4) I\subset \ree\setminus\Omega_{\star}$. In particular $I\notin \W_Q$, for any $Q\in \dd_{\F,Q_0}$. Also, $I$ meets a fattened Carleson box $J^*$ such that $\interior(J^*)\subset\Omega_\star$. Then there exists $Q_J\in\dd_{\F,Q_0}$ such that $J\in\W_{Q_J}$.

 As above, let $Q_I^*$ denote the  nearest dyadic cube to
$I$ with $\ell(I) = \ell(Q_I^*)$ so that $I\in \W_{Q_I^*}$. Then necessarily, $Q_I^*\notin \dd_{\F,Q_0}=\dd_\F\cap \dd_{Q_0}$.

\noindent\textbf{Case 1:} $Q_I^*\notin\dd_\F$. This implies that there is $Q\in\F$ such that $Q_I^*\subset Q$. Then $I\in\R_Q$,
since $I\in \W_{Q_I^*}$, and also $Q\in \F^*$ since $\Sigma\cap I\neq\emptyset$. Hence $I\in \W_{\Sigma,Q}\subset\W_{\Sigma}^\bot$.

\smallskip

\noindent\textbf{Case 2:}  $Q_I^*\in\dd_\F$. We must have $Q_I^*\notin \dd_{Q_0}$. Since  $Q_J\subset Q_0$ we have
$$
\ell(Q_I^*)=\ell(I)\approx\ell(J)\approx \ell(Q_J),
\qquad
\max\{ \ell(Q_I^*),\ell(Q_J),\ell(I),\ell(J)\}\le C_1\,\ell(Q_0),
$$
and
$$
\dist(Q_I^*,Q_0)
\lesssim
d(Q_I^*,I)+\ell(I)+\ell(J)+\dist(J,Q_J)+\ell(Q_0)
\le
C_1
\ell(Q_0),
$$
where the implicit constants depend on $n$, the ADR constant of $E$, $m_0$ and $C_0$, and $C_1$ is taken large enough depending on these parameters.

\smallskip

\noindent{\it Sub-case 2a}:   $\ell(Q_I^*)\le \ell(Q_0)$. We necessarily have $Q_I^*\subset Q\in\F_{||}$. Then $I\in\R_Q$ since $I\in \W_{Q_I^*}$ and also $Q\in \F_{||}^*$ since $\Sigma\cap I\neq\emptyset$. Hence $I\in \W_{\Sigma,Q}\subset\W_{\Sigma}^{||}$.

\smallskip

\noindent{\it Sub-case 2b}:  $\ell(Q_I^*)>\ell(Q_0)$. We observe that
$$
\ell(Q_0)<\ell(Q_I^*)\le C_1\,\ell(Q_0)
\qquad
\mbox{and}
\qquad
\dist(Q_I^*,Q_0)\le C_1\ell(Q_0),
$$
and therefore $Q_I^*\in \F_{\top}$ and thus $I\in\W_{\Sigma}^{\top}$.

\smallskip

This completes the proof of \eqref{decomp:Sigma:R}. Note that \eqref{decomp:Sigma:R:conta} follows at once by our construction. Let us note that for further reference the three sets $\W_\Sigma^{\bot}$, $\W_\Sigma^{||}$, and 
$\W_\Sigma^{\top}$ are pairwise disjoint by the nature of the families $\F$, $\F_{||}$ and $\F_{\top}$.

To prove \eqref{decomp:Sigma} we observe that $\Sigma$ consists of (portions of) faces of certain fattened Whitney cubes $J^*$,
with $\interior(J^*)\subset\Omega_\star$, which meet some $I\in\W$ ---there could be more than one $I$ but we chose just one--- for which $I\notin \W_Q$, for any $Q\in \dd_{\F,Q_0}$ (so that $(3/4) I\subset \ree\setminus\Omega_{\star}$) and $I\cap\Sigma\neq\emptyset$. In particular we can apply \eqref{decomp:Sigma:R} and \eqref{decomp:Sigma} follows immediately.
\end{proof}

\begin{lemma}\label{lemma:I-QI}
Given $I\in \W_\Sigma$, we can find $Q_I\in\dd$,  with $Q_I\subset Q_I^*$, such
that $\ell(I)\approx \ell(Q_I)$,  $\dist(Q_I,I)\approx \ell(I)$, and in addition,
\begin{equation}\label{Sigma:bdd-overlap:I}
\sum_{I\in\W_{\Sigma,Q}} 1_{Q_I}\lesssim 1_{Q},\qquad\mbox{for any }Q\in\F^*\cup\F^*_{||},
\end{equation}
and
\begin{equation}\label{Sigma:bdd-overlap:II}
\sum_{I\in\W_\Sigma^\top} 1_{Q_I}\lesssim 1_{B_{Q_0}^*\cap E},
\end{equation}
where the implicit constants depend on $n$, the ADR constant of $E$, $m_0$ and $C_0$, and where $B_{Q_0}^*=B(x_{Q_0}, C\,\ell(Q))$ with $C$ large enough depending on the same parameters.
\end{lemma}

\begin{proof}
Fix $I\in\W_{\Sigma}$, take $Q_I^*$ and note that, as observed before, $Q_I^*\notin\dd_{\F,Q_0}$. As in the previous proof $I$ meets a fattened Carleson box $J^*$ such that $\interior(J^*)\subset\Omega_\star$. Then there exists $Q_J\in\dd_{\F,Q_0}$ such that $J\in\W_{Q_J}^*$.

We start with the case $I\in\W_{\Sigma,Q}$ with $Q_I^*\in \dd_Q$ and $Q\in\F^*$. 
Notice that $Q_J$ is not contained in 
$Q$ and therefore, upon a moment's reflection, one 
may readily see that $\dist(Q^*_I,E\setminus Q)\lesssim\ell(Q^*_I)$.

We claim that we may select a descendant of $Q_I^*$, call it $Q_I$, of comparable size, in such a way that
\begin{equation}\label{QI,Q}
\dist(Q_I,E\setminus Q)\approx \ell(I)\approx\ell(Q_I)\,,
\end{equation}
while of course retaining the property that $\dist(Q_I,I)\approx \ell(I)$.
Indeed, let $M$ be a sufficiently large, but uniformly bounded integer to be chosen momentarily,
and let $Q_I$ be the cube of ``length" $\ell(Q_I) = 2^{-M} \ell(Q_I^*)$, that contains
$x_{Q_I^*}$ ( the ``center" of $Q_I^*$).  Since there is a ball $B_{Q_I^*}:= B(x_{Q_I^*}, r)$, with
$r\approx \ell(Q_I^*)$, such that $B_{Q_I^*}\cap E\subset Q_I^*$,
we may choose $M$ to be the smallest integer that guarantees that $\diam(Q_I)\leq r/2$, and the claim holds.

Once we have selected $Q_I\subset Q_I^*\subset Q$ with the desired properties we shall see that the cubes $\{Q_I\}_{I\in \W_{\Sigma,Q}}$ have bounded overlap.
Indeed, given $Q_I$, suppose that $Q_{I'}$ meets $Q_I$. By \eqref{QI,Q}, $\ell(I)\approx\ell(I')$ in which case $\dist(I,I')\lesssim \ell(I)$. But the properties of the Whitney cubes easily imply that the number of such $I'$ is uniformly bounded and therefore the $Q_I$ have bounded overlap.

We now consider the case $I\in\W_{\Sigma,Q}$ with $Q_I^*\in \dd_Q$ and $Q\in\F_{||}^*$. As before $Q_J$ is not contained in $Q$ since $Q_J\subset Q_0$ and $Q\in\F_{||}$ means that $Q\neq Q_0$ and $\ell(Q)=\ell(Q_0)$. Then, as before, $\dist(Q^*_I,E\setminus Q)\lesssim\ell(Q^*_I)$ and we may select a descendant of $Q_I^*$, call it $Q_I$, of comparable size, such that \eqref{QI,Q} holds and $\dist(Q_I,I)\approx \ell(I)$. Notice that $Q_I\subset Q_I^*\subset Q$ and the fact that the cubes $\{Q_I\}_{I\in \W_{\Sigma,Q}}$ have bounded overlap follows as before.

Finally let $I\in\W_\Sigma^{\top}$ then $Q_I^*\in\F_\top$. In this case we set $Q_I=Q_I^*$ which clearly has the desired properties. It is trivial to show that $Q_I\subset B_{Q_0}^*$. To obtain the bounded overlap property we observe that if $Q_I\cap Q_I'\neq\emptyset$ with $Q_I, Q_I'\in\F_\top$ then
$\ell(I)\approx\ell(Q_I)\approx \ell(Q_0)\approx\ell(Q_I')\approx \ell(I')$ and also $\dist(I,I')\lesssim \ell(I)$. Thus only for a bounded number of $I'$'s we can have that  $Q_{I'}$ meets $Q_I$. This in turns gives the bounded overlap property.
\end{proof}

\begin{lemma}
For every $x\in\pom_{\star}$ and $0<r\lesssim\ell(Q_0)\approx\diam(\Omega_\star)$, if $Q\in\F^*\cup\F_{||}^*$ then
\begin{equation}\label{eqn:Sigma-I-B}
\sum_{I\in\W_{\Sigma,Q}} H^n\big(B(x,r)\cap\Sigma\cap I\big)
\lesssim
\big(\min\{r,\ell(Q)\}\big)^n,
\end{equation}
where the implicit constants depend on $n$, the ADR constant of $E$, $m_0$, $C_0$.
\end{lemma}

\begin{proof}  We set $B:=B(x,r)$.
We first assume that $\ell(Q)\lesssim r$. Then we use the estimate $H^n(\Sigma\cap I)\lesssim\ell(I)^n$ (which follows easily from the nature of the Whitney cubes), Lemma \ref{lemma:I-QI} and the ADR property of $E$ to obtain as desired
that
$$
\sum_{I\in\W_{\Sigma,Q}} H^n(B\cap\Sigma\cap I)
\lesssim
\sum_{I\in\W_{\Sigma,Q}} \ell(I)^n
\approx
\sum_{I\in\W_{\Sigma,Q}} \ell(Q_I)^n
\approx
\sum_{I\in\W_{\Sigma,Q}} \sigma(Q_I)
\lesssim
\sigma(Q)
\lesssim
\ell(Q)^n.
$$
Suppose next that $\ell(Q)\gg r$ and that $\delta(x)\gg r$ (in particular $x\notin E$). By the nature of the Whitney cubes $B\cap \Sigma\cap I$ consists of portions (of diameter at most $2r$) of faces of Whitney boxes and only a bounded number of $I$'s can contribute in the sum. Hence,
$$
\sum_{I\in\W_{\Sigma,Q}} H^n(B\cap\Sigma\cap I)
\lesssim
r^n.
$$
Finally, consider the case where $\ell(Q)\gg r$ and that $\delta(x)\lesssim r$ (which includes the case $x\in E$). Pick $\hat{x}\in E$ such that $|x-\hat{x}|=\delta(x)$. Let $I\cap B\neq\emptyset$ and pick $z\in I\cap B\neq\emptyset$. Then
$$
\ell(I)
\approx
\dist(I,E)
\le
|z-x|+\delta(x)
\lesssim
r.
$$
Also, by Lemma \ref{lemma:I-QI} we have that $Q_I\subset B(\hat{x},C\,r)$ for some uniform constants $C>1$: for every $y\in Q_I$ we have
$$
|y-\hat{x}|
\lesssim
\ell(Q_I)+
\dist(Q_I,I)
+
\ell(I)
+|z-x|
+
|\hat{x}-x|
\lesssim
r.
$$
Proceeding as before, Lemma \ref{lemma:I-QI} and the ADR property of $E$ yield
$$
\sum_{I\in\W_{\Sigma,Q}} H^n(B\cap\Sigma\cap I)
\lesssim
\sum_{I\in\W_{\Sigma,Q}} \sigma(Q_I)
\lesssim
\sigma\Big(\bigcup_{I\in\W_{\Sigma,Q}}Q_I\Big)
\lesssim
\sigma(B(\hat{x},C\,r)\cap E)
\lesssim
r^n.
$$

\end{proof}

\begin{proof}[Proof of Proposition \ref{prop:Sawtooths-ADR}: Upper ADR bound]
We are now ready to establish that for every $x\in\pom_{\star}$ and $0<r\lesssim\ell(Q_0)$ we have that
\begin{equation}\label{eqn:upper-ADR-sawtooh}
H^n\big(B(x,r)\cap \pom_\star\big)\lesssim r^n
\end{equation}
where the implicit constant only depends on dimension, the ADR constant of $E$ and the parameters $m_0$ and $C_0$.

Write $B:=B(x,r)$ and note first that
$$
H^n(B\cap \pom_\star)
\le
 H^n(B\cap \pom_\star\cap E)
 +
 H^n(B\cap \Sigma).
 $$
For the first term in the right hand side, we may assume that  there exists $x'\in B\cap \pom_\star\cap E$ in which case we have that
$B(x,r)\subset B(x',2\,r)$ and therefore
$$
H^n(B\cap \pom_\star\cap E)
\le
H^n\big(B(x',2\,r) \cap E\big)
\lesssim
r^n,
$$
by the ADR property of $E$ since $r\lesssim\ell(Q_0)\lesssim\diam(E)$.

Let us then establish the bound for the portion corresponding to $\Sigma$. We use \eqref{decomp:Sigma} to write
\begin{multline*}
H^n(B\cap \Sigma)
\le
\sum_{I\in\W_\Sigma^{\bot}} H^n(B\cap\Sigma\cap I)
+
\sum_{I\in\W_\Sigma^{||}} H^n(B\cap\Sigma\cap I)
+
\sum_{I\in\W_\Sigma^{\top}} H^n(B\cap\Sigma\cap I).
\\
\le
\sum_{Q\in \widetilde{\F}_B}\sum_{I\in\W_{\Sigma,Q}} H^n(B\cap\Sigma\cap I)
+
\sum_{I\in\W_\Sigma^{\top}} H^n(B\cap\Sigma\cap I)
=:
S_1+S_2,
\end{multline*}
where $\widetilde{\F}_B$ is the collection of cubes in $Q\in \F^*\cup\F_{||}^*$ such that there is $I\in \W_{\Sigma,Q}$ with $B\cap \Sigma\cap I\neq\emptyset$.
For $S_1$ we write
$$
\widetilde{\F}_B
=
\F_1\cup\F_2
:=
\big\{Q\in\widetilde{\F}_B:\ell(Q)<r \big\}
\cup
\big\{Q\in\widetilde{\F}_B:\ell(Q)\ge r \big\}.
$$
Suppose first that $Q\in\F_1\subset\widetilde{\F}_B$ and pick $z\in B\cap\Sigma\cap I$ with $I\in\W_{\Sigma,Q}$. Then, for any $y\in Q$ we have
$$
|y-x|
\lesssim
\ell(Q)+\dist(Q,I)+\ell(I)+|z-x|
\lesssim
r
$$
and therefore $Q\subset B^*=B(x,C\,r)$. Then \eqref{eqn:Sigma-I-B} gives
$$
\sum_{Q\in \F_1} \sum_{I\in\W_{\Sigma,Q}} H^n(B\cap\Sigma\cap I)
\lesssim
\sum_{Q\in \F_1} \ell(Q)^n
\lesssim
H^n\Big(\bigcup_{Q\in\F_1} Q\Big)
\le
H^n(B^*\cap E)
\lesssim
r^n,
$$
where we have used that $\F_1\subset\F\cup\F_{||}$ and each family is comprised of pairwise disjoint sets. In the last estimate we have employed that $E$ is ADR: note that although $B^*$ is not centered at a point in $E$, we have that either $B^*\cap E=\emptyset$ (in which case the desired estimate is trivial) or $B^*\subset B(x',2\,C\,r)$ for some $x'\in E$ (in which case we can legitimately use the ADR condition).

We next see that the cardinality of $\F_2$ is uniformly bounded. Let $Q_1$, $Q_2\in\F_2$ and assume, without loss of generality, that $r\le \ell(Q_1)\le \ell(Q_2)$. For $i=1,2$ pick $z_i\in B\cap\Sigma\cap I_i$ with $I_i\in \W_{\Sigma, Q_i}$. Then
$$
\ell(I_2)
\approx
\dist(I_2,E)
\le
|z_2-z_1|+\dist(z_1,E)
\lesssim
r+\ell(I_1)
\le
\ell(Q_1)+
\ell(Q_{I_1}^*)
\lesssim
\ell(Q_1)
$$
and consequently
$$
\dist(Q_2,Q_1)
\lesssim
\dist(Q_{I_2}^*, I_2)+\ell(I_2)+|z_2-z_1|+\ell(I_1)+\dist(Q_{I_1}^*, I_1)
\lesssim
\ell(Q_1).
$$
Therefore, for any pair $Q_1$, $Q_2\in\F_2$ we have that $\dist(Q_1,Q_2)\lesssim\min\{\ell(Q_1),\ell(Q_2)\}$ and, since the cubes in $\F_2$ are disjoint we clearly have that the cardinality of $\F_2$ is uniformly bounded. Thus \eqref{eqn:Sigma-I-B} easily gives the desired estimate
$$
\sum_{Q\in \F_2} \sum_{I\in\W_{\Sigma,Q}} H^n(B\cap\Sigma\cap I)
\lesssim
\sup_{Q\in \F_2} \sum_{I\in\W_{\Sigma,Q}} H^n(B\cap\Sigma\cap I)
\lesssim
r^n.
$$
This and the corresponding estimate for $\F_1$ gives that $S_1\lesssim r^n$.

We next consider $S_2$. We first observe that $\#\W_{\Sigma}^{\top}$ is uniformly bounded. Indeed if $I$, $I'\in \W_{\Sigma}^{\top}$ then $Q_{I}^*$, $Q_{I'}^*\in\F_{\top}$ and therefore $\ell(I)\approx \ell(Q_I^*)\approx\ell(Q_0)\approx \ell(Q_{I'}^*)\approx \ell(I')$ and also $\dist(I,I')\lesssim \ell(Q_0)$. This readily implies that $\#\W_{\Sigma}^{\top}\le C$. On the other hand for every $I\in \W_{\Sigma}^{\top}$ we have that $\ell(I)\approx \ell(Q_0)$ and, since $0<r\lesssim \ell(Q_0)$, we clearly have that $H^n(B\cap\Sigma\cap I)\lesssim r^n$.
Thus,
$$
S_2
=
\sum_{I\in\W_\Sigma^{\top}} H^n(B\cap\Sigma\cap I)
\lesssim
\sup_{I\in\W_\Sigma^{\top}} H^n(B\cap\Sigma\cap I)
\lesssim
r^n.
$$
This completes the proof of the upper ADR condition.
\end{proof}

The following results are adaptations of some auxiliary lemmas from \cite{HM-I}.

\begin{proposition}\label{prop:sawtooth-contain}  Suppose that $E$ is a closed ADR set.
Fix $Q_0\in \dd$, and let $\F\subset\dd_{Q_0}$ be a disjoint family.  Then
\begin{equation}\label{eq5.0}
Q_0\setminus \left(\cup_\F Q_j\right)\subset E\cap\partial\Omega_{\F,Q_0}
\subset \overline{Q_0} \setminus \left(\cup_\F \,\,{\rm int}\!\left(Q_j\right)\right)
\end{equation}
\end{proposition}

\begin{proof} We first prove the right hand containment. Suppose that
$x\in E \cap\partial\Omega_{\F,Q_0}$.  Then there is a sequence
$X^k\in\Omega_{\F,Q_0}$, with $X^k\to x$.  By definition of $\Omega_{\F,Q_0}$,
each $X^k$ is contained in $I^*_k$ for some $I_k\in \W_{\F,Q_0}$ (cf. \eqref{Def-WF}-\eqref{eq3.saw}),
so that $\ell(I_k) \approx \delta(X^k)\to 0$.  Moreover, again by definition, each $I_k$
belongs to some $\W_{Q^k}$, $Q^k\in\dd_{\F,Q_0}$ so that,
$$
\dist(Q^k,I_k)\le C_0\,\ell(Q^k) \approx C_0\,\ell(I_k) \to 0.
$$
Consequently, $\dist(Q^k,x)\to 0$.  Since each $Q^k\subset Q_0$, we have $x\in \overline{Q_0}$.
On the other hand, if $x\in {\rm int}(Q_j)$, for some $Q_j\in\F$,  then there is an $\epsilon >0$ such that
$\dist(x,Q)>\epsilon$  for every $Q\in \dd_{\F,Q_0}\,$ with $\ell(Q)\ll\epsilon$, because no
$Q\in\dd_{\F,Q_0}$ can be contained in any $Q_j$.  Since this cannot happen if
$\ell(Q^k) +\dist(Q^k,x)\to 0$, the right hand containment is established.

Now suppose that $x\in Q_0\setminus (\cup_\F Q_j)$.  By definition,
if $x\in Q\in\dd_{Q_0}$, then
$Q\in \dd_{\F,Q_0}$.   Therefore, we may choose a sequence $\{Q^k\}\subset\dd_{\F,Q_0}$
shrinking to $x$,
whence there exist $I_k\in \W_{Q^k}\subset\W_{\F,Q_0}$ (where we are using that $\W_{Q^k}\neq\emptyset$) with $\dist(I_k,x)\to 0$.
The left hand containment now follows.
\end{proof}

\begin{lemma}\label{lemma4.9}
Suppose that $E$ is a closed ADR set.   Let $\F\subset \dd$ be a pairwise disjoint family.
Then for every $Q\subseteq Q_j \in \F$, there is a ball
$B'\subset\ree\setminus \overline{\Omega_{\F}}$,  centered at $E$,
with radius $r'\approx\ell(Q)/C_0$, and $\Delta':=B'\cap E\subset Q$.
\end{lemma}

\begin{proof}
Recall that there exist $B_Q:=B(x_Q,r)$ and $\Delta_Q :=B_Q\cap\partial \Omega
\subset Q$  where $r \approx \ell(Q)$.
We now set
$$
B' = B\left(x_Q,(M\, C_0)^{-1}r\right)\,,
$$
where $M$ is a sufficiently large number to be chosen momentarily.
We need only verify that $B'\cap\Omega_{\F}=\emptyset.$  Suppose not.  Then by definition of
$\Omega_\F$, there is a Whitney cube
$I\in \W_\F$ (see \eqref{Def-WF}) such that $I^*$ meets $B'$.  Since $I^*$ meets $B'$, there is a point
$Y_I\in I^*\cap B'$ such that
$$
\ell(I)\approx \dist(I^*,\partial\Omega)\leq |Y_I-x_Q|\leq r/(M\, C_0)\approx \ell(Q)/(M\, C_0).
$$
On the other hand, since $I\in \W_\F$,
there is a $Q_I\in \dd_\F$ (hence $Q_I$ is not contained in $Q_j$)
with $\ell(I)\approx\ell(Q_I)$, and
$\dist(Q_I,Y_I)\approx\dist(Q_I,I)\le C_0\,\ell(I)\lesssim \ell(Q)/M.$
Then by the triangle inequality,
$$|y-x_Q|\lesssim \ell(Q)/M\,,\qquad \forall y\in Q_I.$$
Thus, if $M$ is chosen large enough,
$Q_I\subset \Delta_Q\subset Q\subset Q_j$,  a contradiction.
\end{proof}

\begin{lemma}\label{lemma:Cks-Sawtooth}
Suppose that $E$ is a closed ADR set.  There exists $0<c<1$ depending only in dimension, the ADR constant of $E$ and $m_0$, $C_0$ such that for every $Q_0\in\dd$,  for every disjoint family $\F\subset\dd_{Q_0}$, for every surface ball $\Delta_\star=\Delta_\star(x,r)=B(x,r)\cap \pom_{\F,Q_0}$ with $x\in \pom_{\F,Q_0}$ and $0<r\lesssim \ell(Q_0)$ there exists $X_{\Delta_\star}$ such that $B(X_{\Delta_\star},c\,r)\subset B(x,r)\cap\Omega_{\F,Q_0}$.
\end{lemma}

This result says that the open set $\Omega_{\F, Q_0}$ satisfies the (interior) corkscrew condition.

\begin{proof}
We fix $Q_0\in \dd$, and a pairwise disjoint family $\{Q_j\}=\F\subset \dd_{Q_0}$.  Set
$$\td:=\td(x,r):=B(x,r)\cap\partial\Omega_{\F,Q_0},$$
with $r\lesssim \ell(Q_0)$ and $x\in \partial\Omega_{\F,Q_0}$.

We suppose first that $x\in \partial\Omega_{\F,Q_0}\cap E$.
Let $M\ge 1$ large enough to be chosen. Following the proof of Proposition \ref{prop:sawtooth-contain} we can find $k\ge 1$ such that $\dist(Q^{k},x)+\ell(Q^{k})<r/M^2$ with $Q^{k}\in\dd_{\F,Q_0}$. In particular, we can pick $x'\in Q^k$ such that $|x-x'|<r/M^2$. We now take an ancestor of $Q^k$, we call it $Q$, with the property that $\ell(Q)\approx r/M<\ell(Q_0)$. Clearly $Q^{k}\in\dd_{\F,Q_0}$ implies that $Q\in\dd_{\F,Q_0}$. Let us pick $I_Q\in \W_Q$  (since $\W_Q$ is not empty) and write $X(I_Q)$ for the center of $I_Q$.

Set $X_{\Delta_\star}=X(I_Q)$ and we shall see that $B(X_{\Delta_\star}, r/M^2)\subset B(x,r)\cap\Omega_{\F,Q_0}$ provided $M$ is large enough. First of all, by construction $I_Q\subset \Omega_{\F,Q_0}$ and therefore $B(X_{\Delta_\star}, r/M^2)\subset \Omega_{\F,Q_0}$ since $r/M^2\approx \ell(Q)/M\le 2^{m_0}\,\ell(I_Q)/M<\ell(I_Q)/4$ if $M$ is large. On the other hand for every $Y\in B(X_{\Delta_\star}, r/M^2)$ we have
\begin{multline*}
|Y-x|
\lesssim
|Y-X_{\Delta_\star}|+\ell(I_Q)+\dist(I_Q,Q)+\ell(Q)+|x'-x|
\\
\lesssim
\frac{r}{M^2}+(2^{m_0}+C_0)\,\ell(Q)
\lesssim
\frac{r}{M^2}+\frac{(2^{m_0}+C_0)\,r}{M}
<r,
\end{multline*}
provided $M$ is taken large enough depending on dimension, ADR, $m_0$ and $C_0$. This completes the proof of the case $x\in \partial\Omega_{\F,Q_0}\cap E$.

Next, we suppose that $x\in \partial\Omega_{\F,Q_0}\setminus E,$
where as above $\td:=\td(x,r).$  Then by definition of the sawtooth region,
$x$ lies on a face of a fattened Whitney cube $I^*=(1+\tau)I$, with
$I\in \W_Q$, for some
$Q\in \dd_{\F,Q_0}.$  If $r\lesssim \ell(I)$, then trivially there
is a point $X^{\!\star}\in I^*$ such that $B(X^{\!\star},cr)\subset
B(x,r)\cap\interior(I^*)\subset B(x,r)\cap\Omega_{\F,Q_0}$.  This $X^{\!\star}$ is then a Corkscrew point
for $\td$.  On the other hand, if $\ell(I)<r/M,$ with $M$ sufficiently large to be chosen momentarily, then there is a $Q'\in \dd_{\F,Q_0}$,
with $\ell(Q')\approx r/M,$ and $Q\subseteq Q'$.  Now fix $I_{Q'}\in \W_{Q'}$ and set  $X_{\Delta_\star}=X(I_{Q'})$. We see that $B(X_{\Delta_\star}, r/M^2)\subset B(x,r)\cap\Omega_{\F,Q_0}$ provided $M$ is large enough.
By construction $I_{Q'}\subset \Omega_{\F,Q_0}$ and therefore $B(X_{\Delta_\star}, r/M^2)\subset \Omega_{\F,Q_0}$ since $r/M^2\approx \ell(Q')/M\le 2^{m_0}\,\ell(I_{Q'})/M<\ell(I_Q)/4$ provided $M$ is large. On the other hand for every $Y\in B(X_{\Delta_\star}, r/M^2)$ we have
\begin{multline*}
|Y-x|
\lesssim
|Y-X_{\Delta_\star}|
+
\ell(I_{Q'})+\dist(I_{Q'},Q')+\ell(Q')+\ell(Q)+\dist(Q,I)+\ell(I)
\\
\lesssim
\frac{r}{M^2}+(2^{m_0}+C_0)\,(\ell(Q')+\ell(I))
\lesssim
\frac{r}{M^2}+\frac{(2^{m_0}+C_0)\,r}{M}
<r,
\end{multline*}
if we take  $M$ large enough depending on dimension, ADR, $m_0$ and $C_0$.
\end{proof}

\begin{proof}[Proof of Proposition \ref{prop:Sawtooths-ADR}: Lower ADR bound]
We are now ready to establish that for every $x\in\pom_{\star}$ and $0<r\lesssim\ell(Q_0)$ we have that
\begin{equation}\label{eqn:lower-ADR-sawtooh}
H^n\big(B(x,r)\cap \pom_\star\big)\gtrsim r^n
\end{equation}
where the implicit constant only depends on dimension, the ADR constant of $E$ and the parameters $m_0$ and $C_0$.

Write $B:=B(x,r)$ and $\td=\td(x,r):=B\cap\pom_{\star}$.  We consider two main cases.  As usual,
$M$ denotes a sufficiently large number to be chosen.

\noindent{\bf Case 1}:  $\delta(x)\geq r/(M\,C_0)$.
In this case, for some $J$ with
$\interior(J^*)\subset \Omega_{\star}$, we have that
$x$ lies on a subset $F$ of a (closed)
face of $J^*$, satisfying $H^n(F)\gtrsim (r/(M\,C_0))^n$,  and $F\subset \pom_{\star}$.
Thus, $H^n(B\cap \pom_{\star})\geq H^n(B\cap F)\gtrsim (r/(M\,C_0))^n,$ as desired.

\noindent{\bf Case 2}:  $\delta(x)< r/(M\,C_0)$.   In this case, we have that
$\dist(x,Q_0) \lesssim r/M.$
Indeed, if $x\in E\cap\pom_{\star}$, then by Proposition
\ref{prop:sawtooth-contain}, $x\in \overline{Q_0}$, so that $\dist(x,Q_0) =0$.
Otherwise, there is some cube
$Q\in \dd_{\F,Q_0}$ such that $x$ lies on the face
of a fattened Whitney cube $I^*$, with $I\in \W_{Q}^*$, and
$\ell(Q)\approx\ell(I)\approx \delta(x)<r/(M\,C_0)$.   Thus,
$$
\dist(x,Q_0)\lesssim\dist(I,Q)\le C_0\,\ell(Q)\lesssim r/M.
$$
Consequently,  we may choose $\hat{x}\in Q_0$
such that $|x-\hat{x}|\lesssim r/M$.
Fix now $\widehat{Q}\in \dd_{Q_0}$
with $\hat{x}\in \widehat{Q}$ and $\ell(\widehat{Q})\approx r/M$.  Then
for $M$ chosen large enough we have that
$\widehat{Q}\subset B(\hat{x},r/\sqrt{M})
\subset B(x,r)$.  We now consider two sub-cases.

\noindent{\it Sub-case 2a}: $B(\hat{x},r/\sqrt{M})$  meets a $Q_j\in\F$ with
$\ell(Q_j)\geq r/M$.  Then in particular, there is a $Q\subseteq Q_j$,
with $\ell(Q)\approx r/M$, and $Q\subset B(\hat{x},2r/\sqrt{M})$.  By Lemma
\ref{lemma4.9}, there is a ball $B'\subset \ree\setminus \overline{\Omega_{\star}}$,
with radius $r' \approx \ell(Q)/C_0\approx r/(C_0\,M)$, such that $B'\cap E\subset Q$, and thus also
$B'\subset B$ (for $M$ large enough).
On the other hand, we can apply Lemma \ref{lemma:Cks-Sawtooth} to find $B''=B(X_{\Delta_{\star}},c\,r)\subset B(x,r)\cap\Omega_{\star}$.
Therefore, by the isoperimetric inequality and the structure theorem
for sets of locally finite perimeter (cf. \cite{EG}, pp. 190 and 205, resp.)
we have $H^n(\td)\gtrsim c_{C_0} r^n$ (note that $\pom_\star$  is of local finite perimeter since we have already shown the upper ADR property).

\noindent{\it Sub-case 2b}:  there is no $Q_j$ as in sub-case 2a.   Thus,
if $Q_j\in\F$ meets $B(\hat{x},r/\sqrt{M})$, then $\ell(Q_j)\leq r/M$.  Since $\hat{x}\in Q_0$,
there is a surface ball
\begin{equation}\label{Delta1}
\Delta_1:= \Delta(x_1,cr/\sqrt{M})\subset Q_0\cap B(\hat{x},r/\sqrt{M})\subset Q_0\cap B.
\end{equation}
Let $\F_1$ denote the collection of those
$Q_j\in\F$ which meet $\Delta_1$.  We then have the covering
$$\Delta_1\subset \left(\cup_{\F_1} Q_j \right)\cup \left(\Delta_1\setminus(\cup_{\F_1} Q_j)\right).$$
If \begin{equation}\label{eqA.6*}
\sigma\left(\frac12 \Delta_1\setminus (\cup_{\F_1} Q_j)\right)
\geq\frac12 \sigma\left(\frac12 \Delta_1\right) \approx  r^n,\end{equation}
then we are done, since $\Delta_1\setminus (\cup_{\F_1} Q_j)\subset
(Q_0\setminus(\cup_{\F} Q_j))\cap B\subset \td,$
by Proposition \ref{prop:sawtooth-contain}.

Otherwise, if \eqref{eqA.6*} fails, then
\begin{equation}\label{eqA.7**}
\sum_{Q_j\in\F_1'}\sigma(Q_j) \gtrsim  r^n,
\end{equation}
where $\F_1'$ is the family of cubes $Q_j \in\F_1$ meeting
$\frac12 \Delta_1$.

We apply Lemma \ref{lemma4.9} with $Q=Q_j$ and there is a ball
$B_j=B(x_j,r_j)\subset\ree\setminus\overline{\Omega_{\F}}\subset\ree\setminus \overline{\Omega_\star}$ with $x_j\in E$ (indeed $x_j$ is the ``center'' of $Q_j$), $r_j\approx \ell(Q_j)/C_0$ and $B_j\cap E\subset Q_j$. Also, the dyadic parent $\widetilde{Q}_j$ of $Q_j$ belongs to $\dd_{\F, Q_0}$. Thus, we can find $I_j\in \W_{\widetilde{Q}_j}$ so that $I_j\subset\Omega_\star$. If we write $X(I_j)$ for the center of $I_j$ we have
$$
|x_j-X(I_j)|
\lesssim
\ell(\widetilde{Q}_j)+\dist(\widetilde{Q}_j, I_j)+\ell(I_j)
\lesssim
(2^{m_0}+C_0)\,\ell(Q_j).
$$
Note that $X(I_j) \in I_j\subset \Omega_\star$ and $x_j\in \ree\setminus\overline{\Omega_\star}$. Thus we can find $x_j^\star\in \pom_{\star}$ in the segment that joins $x_j$ and $X(I_j)$. We now consider $B_j^\star=B(x_j^\star, C\, (2^{m_0}+C_0)\,\ell(Q_j))$ which is a ball centered at $\pom_{\star}$. We first see that $
B_j\subset B_j^\star\setminus\overline{\Omega_{\star}}$.
We already know that $B_j \subset \ree\setminus \overline{\Omega_\star}$
and on the other hand if $y\in B_j$ we have
$$
|y-x_j^\star|
\le
|y-x_j|+|x_j-x_j^\star|
<
r_j+|x_j-X(I_j)|
\lesssim
(2^{m_0}+C_0)\,\ell(Q_j),
$$
and therefore $B_j\subset B_j^\star$. On the other hand, we can also show that $B(X(I_j),\ell(I_j)/4)\subset B_j^\star\cap \Omega_{\star}$. Indeed, $B(X(I_j),\ell(I_j)/4)\subset I_j\subset \Omega_{\star}$ and for every $y\in B(X(I_j),\ell(I_j)/4)$ we have
$$
|y-x_j^\star|
\le
|y-X(I_j)|+|X(I_j)-x_j^\star|
<
\ell(I_j) +|X(I_j)-x_j|
\lesssim
(2^{m_0}+C_0)\,\ell(Q_j)
$$
which yields that $B(X(I_j),\ell(I_j)/4)\subset B_j^\star$. Therefore, by the isoperimetric inequality and the structure theorem
for sets of locally finite perimeter (cf. \cite{EG}, pp. 190 and 205, resp.)
we have
\begin{equation}\label{eqn:est-Bjstar}
H^n(B_j^\star\cap\pom_{\star} )
\gtrsim
c_{C_0, m_0} \ell(Q_j)^n
\approx \sigma(Q_j).
\end{equation}
(note that $\pom_\star$  is of local finite perimeter since we have already shown the upper ADR property).

On the other hand, if we write $\hat{B}_{Q_j}=B(x_{Q_j}, C_1\,\ell(Q_j))$   such that $Q_j\subset \hat{B}_{Q_j}\cap E$ (see \eqref{cube-ball}) we can find $N=N(m_0,C_0)$ such that $B_j^\star\subset N\,\hat{B}_{Q_j}$. Indeed if $Y\in B_j^\star$ we have
\begin{multline}\label{dist:Y-xQj}
|Y-x_{Q_j}|
\le
|Y-x_j^\star|+|x_j^\star-x_j|
\le
C\,(2^{m_0}+C_0)\ell(Q_j)
+
|X(I_j)-x_j|
\\
\le
C'\,(2^{m_0}+C_0)\ell(Q_j)
<N\,C_1\,\ell(Q_j),
\end{multline}
where we have used that $x_j=x_{Q_j}$.

From \eqref{eqA.7**} it follows that we can find a finite family $\F_2\subset\F_1'$ such that
\begin{equation}\label{eqA.7***}
\sum_{Q_j\in\F_2}\sigma(Q_j)
\ge
\frac12
\sum_{Q_j\in\F_1'}\sigma(Q_j)
\gtrsim  r^n.
\end{equation}
From $\F_2$, following a typical covering argument, we can now take a
subcollection $\F_3$ so that the family $\{N\, \hat{B}_{Q_j}\}_{Q_j\in\F_3}$ is disjoint and also satisfies that if $Q_j\in\F_2\setminus \F_3$ then there exists $Q_k\in \F_3$ such that $r(\hat{B}_{Q_k})\ge r(\hat{B}_{Q_j})$  and $N\,\hat{B}_{Q_j}$ meets $N\, \hat{B}_{Q_k}$. Then it is trivial to see that
$$
\bigcup_{Q_j\in\F_2} Q_j
\subset
\bigcup_{Q_j\in\F_2} \hat{B}_{Q_j}
\subset
\bigcup_{Q_j\in\F_3} (2\,N+1)\hat{B}_{Q_j}
$$
Notice that the fact that the family $\{N\, \hat{B}_{Q_j}\}_{Q_j\in\F_3}$ is comprised of pairwise disjoint balls yields that the balls $\{B_j^\star\}_{Q_j\in\F_3}$ are also pairwise disjoint. Thus the previous considerations and \eqref{eqn:est-Bjstar} give
\begin{multline*}
H^n
\Big(
\bigcup_{Q_j\in\F_3} B_j^\star\cap\pom_{\star}
\Big)
=
\sum_{Q_j\in\F_3} H^n(B_j^\star\cap\pom_{\star})
\gtrsim
\sum_{Q_j\in\F_3} \sigma(Q_j)
\\
\gtrsim
\sigma
\Big(
\bigcup_{Q_j\in\F_3} (2\,N+1)\,\hat{B}_{Q_j}\cap E
\Big)
\ge
\sigma
\Big(
\bigcup_{Q_j\in\F_2} Q_j
\Big)
=
\sum_{Q_j\in\F_2} \sigma(Q_j)
\gtrsim r^n.
\end{multline*}

To complete the proof given $Q_j\in\F_3\subset\F_1'\subset\F$ we have that $Q_j$ meets $\frac12 \Delta_1$ and we can pick $z_j$ belonging to both sets. Notice that by \eqref{Delta1} in the present subcase we must have $\ell(Q_j)\le r/M$. This, \eqref{Delta1} and \eqref{dist:Y-xQj} imply that for every  $Y\in B_j^\star$ we have
\begin{multline*}
|Y-x|
\le
|Y-x_{Q_j}|+|x_{Q_j}-z_j|+|z_j-x_1|+|x_1-\hat{x}|+|\hat{x}-x|
\lesssim
\frac{r}{\sqrt{M}}+\frac{r}{M}
\lesssim
\frac{r}{\sqrt{M}}<r
\end{multline*}
provided $M$ is large enough, and therefore $B_j^\star\subset B$. This in turn gives as desired that
$$
H^n(B\cap\pom_{\star})
\ge
H^n
\Big(
\bigcup_{Q_j\in\F_3} B_j^\star\cap\pom_{\star}
\Big)
\gtrsim r^n.
$$
\end{proof}


\begin{thebibliography}{AHMTT1}
\parskip=0.1cm



\bibitem[AHLT]{AHLT} P. Auscher, S. Hofmann, J.L. Lewis and P. Tchamitchian,  Extrapolation of Carleson measures and the analyticity of Kato's square-root operators,  {\em Acta Math.}     \textbf{187}  (2001),  no. 2, 161--190.

\bibitem[AHMTT]{AHMTT} P. Auscher, S. Hofmann, C. Muscalu, T. Tao and C. Thiele,
Carleson measures, trees, extrapolation, and $T(b)$ theorems,  {\em Publ. Mat.} {\bf 46}
(2002),  no. 2, 257--325.

\bibitem[AHMNT]{AHMNT} J. Azzam, S. Hofmann,  J. M.  Martell, K. Nystr{\"o}m, T. Toro, 
A new characterization of chord-arc domains, preprint 2014.  arXiv:1406.2743


\bibitem[BJ]{BiJo} C. Bishop and P. Jones, Harmonic measure and arclength,
{\it Ann. of Math. (2)} {\bf 132} (1990), 511--547.

\bibitem[Car]{Car} L. Carleson,  Interpolation by bounded analytic functions and the corona problem,
{\it Ann. of Math. (2)} {\bf 76} (1962), 547--559.

\bibitem[CG]{CG} L. Carleson and J. Garnett, Interpolating sequences and separation properties,
{\it J. Analyse Math.} {\bf 28} (1975), 273--299.


\bibitem[Chr]{Ch} M. Christ,  A $T(b)$ theorem with remarks on analytic
capacity and the Cauchy integral, {\it Colloq. Math.}, {\bf LX/LXI} (1990), 601--628.

\bibitem[Da1]{D1} B. Dahlberg,
On estimates for harmonic measure,
{\it Arch. Rat. Mech. Analysis} {\bf 65} (1977), 272--288.

\bibitem[Da2]{D2} B. Dahlberg,
Weighted norm inequalities for the Lusin area integral and the non- tangential
maximal function for functions harmonic in a Lipschitz domain,  {\it Studia Math.} {\bf 67}
(1980), 297-314.

\bibitem[Da3]{D} B. Dahlberg,  Approximation of harmonic functions
{\it Ann. Inst. Fourier (Grenoble)} {\bf 30} (1980) 97-107.

\bibitem[DJK]{DJK} B.E. Dahlberg, D.S. Jerison and C.E. Kenig, Area integral estimates for elliptic differential operators with nonsmooth coefficients, {\it Ark. Mat.} \textbf{22} (1984), no. 1, 97--108.


\bibitem[DJ]{DJe} G. David and D. Jerison, Lipschitz approximation
to hypersurfaces, harmonic measure,
and singular integrals, {\it Indiana Univ. Math. J.} {\bf 39} (1990),
no. 3, 831--845.

\bibitem[DS1]{DS1} G. David and S. Semmes,
Singular integrals and rectifiable sets in $\re^n$: Beyond Lipschitz graphs, {\it Asterisque} {\bf 193} (1991).

\bibitem[DS2]{DS2} G. David and S. Semmes, {\it Analysis of and on
Uniformly
Rectifiable Sets}, Mathematical Monographs and Surveys {\bf 38}, AMS
1993.

\bibitem[EG]{EG} L.C. Evans and R.F. Gariepy, {\it Measure Theory and Fine
Properties of Functions}, Studies in Advanced Mathematics, CRC Press,
Boca Raton, FL, 1992.

\bibitem[FS]{FS} C. Fefferman, E.M. Stein, {\it $H\sp{p}$ spaces
of several variables}, Acta Math. 129 (1972), no. 3-4, 137--193.

\bibitem[Gar]{G} J. Garnett, {\it Bounded Analytic Functions}, Academic Press, San Diego, 1981.

\bibitem[HKMP]{HKMP} S. Hofmann, C. Kenig, S. Mayboroda, and J. Pipher,
Square function/Non-tangential maximal function
estimates and the Dirichlet problem for
non-symmetric elliptic operators, preprint 2012. arXiv:1202.2405

\bibitem[HL]{HL}  S. Hofmann and J.L. Lewis, The Dirichlet problem for parabolic operators
with singular drift terms, {\it Mem. Amer. Math. Soc.} \textbf{151} (2001), no. 719.

\bibitem[HM1]{HM-TAMS} S. Hofmann and J.M. Martell, $A_\infty$ estimates via extrapolation of Carleson measures
and applications to divergence form elliptic operators, {\it Trans. Amer. Math. Soc.} \textbf{364} (2012), no. 1, 65--101

\bibitem[HM2]{HM-I} S. Hofmann and J.M. Martell, Uniform rectifiability and harmonic measure I: Uniform rectifiability implies Poisson kernels in $L^p$, {\it Ann. Sci. \'Ecole Norm. Sup.}, to appear.

\bibitem[HMM]{HMM} S. Hofmann, J.M. Martell, and S. Mayboroda,  Uniform rectifiability and harmonic Measure III: Riesz transform bounds imply uniform rectifiability of boundaries of 1-sided NTA domains,  {\it Int. Math. Res. Not.} \textbf{2014}, no. 10, 2702--2729.


\bibitem[HMU]{HMU} S. Hofmann, J.M. Martell and I. Uriarte-Tuero, Uniform rectifiability
and harmonic measure II:  Poisson kernels in $L^p$ imply uniform rectifiability, 
{\it Duke Math. J.} \textbf{163} (2014), no. 8, 1601--1654.


\bibitem[HMay]{HM08} S. Hofmann,  S. Mayboroda, Hardy and BMO spaces associated to divergence form elliptic operators, {\it Math. Ann.} {\bf 344} (2009), no. 1, 37--116.



\bibitem[HMMM]{HMMM} S. Hofmann, D.  Mitrea, M. Mitrea, A. Morris,
$L^p$-Square Function Estimates on Spaces of Homogeneous Type and on Uniformly Rectifiable Sets,
preprint 2103.  arXiv:1301.4943



\bibitem[JK]{JK} D. Jerison and C. Kenig,  Boundary behavior of
harmonic functions in nontangentially accessible domains, {\it Adv. in Math.}
\textbf{46} (1982), no. 1, 80--147.

\bibitem[KKPT]{KKPT}  C. Kenig, H. Koch, H. J. Pipher and T. Toro,  A new approach to absolute continuity of elliptic measure, with applications to non-symmetric equations. {\it Adv. Math.} {\bf 153}  (2000),  no. 2, 231--298.

\bibitem[KKT]{KKT} C. Kenig, B. Kirchheim, T. Toro, manuscript in preparation.

\bibitem[LM]{LM} J. Lewis and M. Murray,  The method of layer potentials for the heat equation in time-varying domains, {\it Mem. Amer. Math. Soc.} \textbf{114} (1995), no. 545.

\bibitem[MMV]{MMV} P. Mattila, M. Melnikov and J.\,Verdera, The Cauchy integral, analytic capacity, and uniform rectifiability,
{\it Ann. of Math. (2)} \textbf{144} (1996), no. 1, 127--136.


\bibitem[NToV]{NToV} F. Nazarov, X. Tolsa, and A. Volberg, On the uniform rectifiability of ad-regular measures with bounded Riesz
transform operator: The case of codimension 1, {\it Acta Math.}, to appear.

\bibitem[RR]{Rfm} F. and M. Riesz, \"Uber die randwerte einer analtischen funktion,
{\it Compte Rendues du Quatri\`eme Congr\`es des Math\'ematiciens Scandinaves}, Stockholm 1916,
Almqvists and Wilksels, Upsala, 1920.

\bibitem[Ste]{St} E. M. Stein, {\it Singular Integrals and Differentiability Properties
of Functions}, Princteon University Press, Princeton, NJ, 1970.

\bibitem[Var]{Va} N. Varopoulos, A remark on functions of bounded mean oscillation and bounded harmonic functions,
{\it Pacific J. Math.} {\bf 74}, 1978, 257--259.

\end{thebibliography}
\end{document}